\documentclass[aos]{imsart}
\RequirePackage{amsthm,amsmath,amsfonts,amssymb,mathtools, bbm}
\RequirePackage[numbers]{natbib}
\RequirePackage[colorlinks,citecolor=blue,urlcolor=blue]{hyperref}

\startlocaldefs
\theoremstyle{plain}
\newtheorem{theorem}{Theorem}[section]
\newtheorem{lemma}[theorem]{Lemma}
\newtheorem{corollary}[theorem]{Corollary}
\newtheorem{hypothesis}{Hypothesis}
\theoremstyle{remark}
\newtheorem{definition}[theorem]{Definition}
\newtheorem{remark}{Remark}

\newcommand{\bs}[1]{\boldsymbol{#1}}
\newcommand{\proj}{P_{E_D}}
\newcommand{\pertur}{\bar{\psi}_D}
\newcommand{\perturb}{\bar{\bs{\psi}}_D}
\newcommand{\hperturb}{ \lambda^\intercal \bs{h}_N }
\newcommand{\abs}[1]{\left\lvert #1 \right\rvert}

\newcommand{\forward}{\mathcal{G}}
\newcommand{\truth}{\theta_0}
\newcommand{\plim}{\stackrel{P}{\longrightarrow}}
\newcommand{\dlim}{\stackrel{d}{\longrightarrow}}
\newcommand{\regset}{\Theta_{N, M}}
\newcommand{\renorm}{i_D^{\frac{1}{2}}}
\newcommand{\lin}{\mathbb{I}_{\truth}}
\newcommand{\fisher}{\mathcal{I}}
\newcommand{\fishermatrix}{\mathcal{I}_D}
\newcommand{\rk}{\mathbb{R}^k}
\newcommand{\rkbyk}{\mathbb{R}^{k \times k}}
\newcommand{\h}{\mathbb{H}}
\newcommand{\ip}[3]{\left\langle #1, #2 \right\rangle_{#3}}
\newcommand{\norm}[2]{\left\Vert #1 \right\Vert_{#2}}
\newcommand{\hprod}[2]{\langle #1, #2 \rangle_{\h}}

\newcommand{\rknorm}[1]{\Vert #1 \Vert_{\rk}}

\newcommand{\eigmin}[1]{\lambda_{\text{min}}\left( #1 \right)}
\newcommand{\eigmax}[1]{\lambda_{\text{max}}\left( #1 \right)}
\newcommand{\Ptheta}{\mathbb{P}_{\truth}}
\newcommand{\Etheta}{\mathbb{E}_{\truth}}
\newcommand{\Vtheta}{\mathbb{V}_{\truth}}
\newcommand{\lr}[1]{\left( #1 \right)}
\newcommand{\curly}[1]{\left\{ #1 \right \}}
\newcommand{\sq}[1]{\left[ #1 \right]}
\newcommand{\likelihood}[1]{\ell_N\left( #1 \right)}
\newcommand{\posterior}[1]{\Pi_N\left( #1 \mid \mathcal{D}_N \right)}
\newcommand{\prior}[1]{\Pi_N\left( #1 \right)}

\newcommand{\Epost}[1]{\mathbb{E}[#1 \mid \mathcal{D}_N]}
\newcommand{\Vpost}[1]{\mathbb{V}[#1 \mid \mathcal{D}_N]}
\newcommand{\Epostreg}[2]{\mathbb{E}^{#1}[#2 \mid \mathcal{D}_N]}
\newcommand{\pspace}{\mathcal{H}_N}
\newcommand{\qform}[3]{\big(#1)^\intercal #2 \big(#3 \big)}
\newcommand{\variance}{\hat{\Sigma}_N}
\newcommand{\bigpar}[1]{\big( #1 \big)}

\endlocaldefs

\begin{document}

\begin{frontmatter}
\title{Valid Credible Ellipsoids for Linear Functionals by a Renormalized Bernstein-von Mises Theorem}
\runtitle{Valid Credible Ellipsoids for Linear Functionals}

\begin{aug}
\author[A]{\fnms{Gustav}~\snm{Rømer}\ead[label=e1]{cgr32@cam.ac.uk}}
\address[A]{Department of Pure Mathematics and Mathematical Statistics,
University of Cambridge\printead[presep={,\ }]{e1}}
\end{aug}

\begin{abstract}
    We consider an infinite-dimensional Gaussian regression model, equipped with a high-dimensional Gaussian prior. We address the frequentist validity of posterior credible sets for a vector of linear functionals. 
  
    We specify conditions for a 'renormalized' Bernstein-von Mises theorem (BvM), where the posterior, centered at its mean, and the posterior mean, centered at the ground truth, have the same normal approximation. This requires neither a solution to the information equation nor a $\sqrt{N}$-consistent estimator. 
    
    We show that our renormalized BvM implies that a credible ellipsoid, specified by the mean and variance of the posterior, is an asymptotic confidence set. For a single linear functional, we identify a credible ellipsoid with a symmetric credible interval around the posterior mean. We bound the diameter.
      
    We check our conditions for Darcy's problem, where the information equation has no solution in natural settings. For the Schrödinger problem, we recover an efficient semi-parametric BvM from our renormalized BvM.
\end{abstract}



\end{frontmatter}
\section{Introduction}
Consider a regression model on a domain $\mathcal{X} \subset \mathbb{R}^d$ with real-valued regression functions $\{\mathcal{G}_\theta : \theta \in \Theta \}$,
\begin{align*}
  Y_i = \forward_\theta(X_i) + \varepsilon_i, \quad i \in \{1, \ldots, N\}.
\end{align*}
Assume $\varepsilon_i \sim \mathcal{N}\lr{0, \sigma_0^2}$, where $\sigma_0^2 > 0$ is known, that $X_i$ is uniformly distributed on $\mathcal{X}$, and that $X_1, \varepsilon_1 \ldots, X_N, \varepsilon_N$ are independent. Let $\Theta$ be a dense subspace of a separable Hilbert space $\h$. Fix linearly independent $\psi_1, \ldots, \psi_k \in \h$, and define the functional $\Psi \theta \coloneqq \lr{\ip{\theta}{\psi_i}{\h}}_{i \le k}$ for $\theta \in \Theta$. We construct asymptotic confidence ellipsoids for $\Psi \theta_0$ based on the observations $\mathcal{D}_N \coloneqq ((X_i, Y_i))_{i\le N} \sim \Ptheta^N$.

When the 'forward map' $\theta \mapsto \mathcal{G}_\theta$ is non-linear and difficult to invert, this as a 'non-linear inverse problem' \cite{Stuart2010,Nickl2023}. For practical inference, Bayesian Markov chain Monte Carlo methods with Gaussian process priors are popular \cite{Stuart2010,cotter2013mcmc,beskos2017geometric}. We therefore consider a posterior $\Pi_N(\cdot \mid \mathcal{D}_N)$ derived from a Gaussian prior $\Pi_N$, defined on an approximation space $E_D \subset L^2(\mathcal{X})$ with dimension $D=D(N)$. We then study 'credible ellipsoids' of $\Pi_{N}^\Psi(\cdot \mid \mathcal{D}_N) \coloneqq \Pi_N(\cdot \mid \mathcal{D}_N) \circ \Psi^{-1}$. 

Consider first a scalar functional $\Psi =  \ip{\cdot}{\psi}{\h}$, and the 'credible interval' 
\begin{align}\label{eq:credinterval}
  I_N \coloneqq \big[\hat{\Psi}_N - R_N, \hat{\Psi}_N + R_N \big],
\end{align}
where $\hat{\Psi}_N \coloneqq \int  \Psi \theta \, d\Pi_N(\theta \mid \mathcal{D}_N)$, and $R_N$ is a statistic such that $\Pi_N(I_N \mid \mathcal{D}_N) = 1 -\alpha$. If a semi-parametric 'Bernstein-von Mises theorem' (BvM) holds, then $I_N$ is an asymptotic $(1 - \alpha)$-confidence interval \cite[Chapter~4]{Nickl2023}. 

A semi-parametric BvM involves a posterior normal approximation with the 'efficient variance' $\frac{1}{N}i_\psi^{-1}(\theta_0)$. Here, the efficient variance is the supremum of the Cramér-Rao bound for $\Psi \theta_0$ along appropriately differentiable one-dimensional sub-models \cite[Section~3.1.2]{Nickl2023}. 
\begin{definition}\label{eq:infdef}
  Let $\sigma_N^2 \coloneqq \frac{1}{N}i_\psi^{-1}(\theta_0)$, and assume $\sigma_N \in (0, \infty)$. 
  
  We say that $\Pi_N^\Psi(\cdot \mid \mathcal{D}_N)$ satisfies a semi-parametric BvM, if
  \begin{align}\label{eq:ap1}
    \hat{\Psi}_N &\stackrel{d}{\approx} \mathcal{N}\big(\Psi \theta_0, \sigma_N^2 \big), \\
    \Pi_N^\Psi(\cdot \mid \mathcal{D}_N)  &\approx \mathcal{N}\big(\hat{\Psi}_N,\sigma_N^2\big) , \label{eq:ap2}
  \end{align}
  as $N \longrightarrow \infty$, where:
  \begin{enumerate}
    \item $\stackrel{d}{\approx}$ denotes a weak $\mathcal{N}(0, 1)$ limit, after transformation by $x \mapsto \sigma_N^{-1}\lr{x - \Psi \truth}$,
    \item $\approx$ denotes a weak $\mathcal{N}(0, 1)$ limit in probability, after transformation by $x \mapsto \sigma_N^{-1}\big(x - \hat{\Psi}_N \big)$.
  \end{enumerate}
\end{definition}
\begin{theorem}[{\cite[Section~4.1.3]{Nickl2023}}]\label{thm:informal}
  If $\Pi_N^\Psi(\cdot \mid \mathcal{D}_N)$ satisfies a semi-parametric BvM, then 
  \begin{align*}
    \Ptheta^N\lr{\Psi \theta_0 \in I_N} \longrightarrow 1- \alpha, \quad N \longrightarrow \infty.
  \end{align*}
\end{theorem}
The current semi-parametric BvM proofs \cite{Nickl_2020,Monard_2021, Nickl2023} require a solution $\phi$ of the 'information equation' $\mathcal{I}_{\theta_0}\phi = \psi$, where the 'information operator' $\mathcal{I}_{\theta_0} : \h \longrightarrow \h$ generalizes the information matrix. In a sufficiently differentiable finite-dimensional model, the information matrix is invertible, when the model is locally identifiable \cite{24d82336-c2f2-361f-aeb7-920f25df191e}. In contrast, the information equation may have \textit{no solution} in an identifiable semi-parametric model, that is 'differentiable in quadratic mean' \cite[Section~25.3]{Vaart_1998}.

This is the case for "Darcy's problem" \cite[Section~1.1]{Nickl2023}. Here, the forward map takes a sufficiently smooth function $\theta : \mathcal{X} \longrightarrow \mathbb{R}$ to the unique solution $u = u_\theta$ of the PDE
\begin{align}
  \nabla \cdot \big(e^\theta \nabla u \big) = f, \quad u|_{\partial \mathcal{X}} = g,
\end{align}
where $f \in C^\infty(\bar{\mathcal{X}})$ and $g\in C^\infty(\partial \mathcal{X})$ are fixed. This can model the pressure of a fluid in a porous medium with 'permeability' $\theta$, see \cite[Section~3.7]{Stuart2010}. In \cite{Nickl2023x}, they consider the space $\h \coloneqq L^2(\mathcal{X})$ and show that the model is differentiable along $C_c^\infty(\mathcal{X})$ 'tangent vectors'. They then construct natural examples, where $i_\psi^{-1}(\theta_0) = \infty$ for large classes of $\psi \in C_c^\infty(\mathcal{X})$. Clearly, a semi-parametric BvM cannot hold, when the efficient variance is infinite. More generally, a $\sqrt{N}$-consistent estimator for $\Psi \theta_0$ is unattainable \cite[Theorem~3]{Nickl2023x}. 

This paper stems from the observation that the proof of Theorem~\ref{thm:informal} does not rely on an \textit{efficient} normal approximation.
\begin{definition}
We say that $\Pi_N^\Psi(\cdot \mid \mathcal{D}_N)$ satisfies a renormalized BvM, if there exists \textit{some} positive sequence $\lr{\sigma_N^2}$ such that \eqref{eq:ap1} and \eqref{eq:ap2} hold.
\end{definition}
\begin{theorem}
  If $\Pi_N^\Psi(\cdot \mid \mathcal{D}_N)$ satisfies a renormalized BvM, then 
  \begin{align*}
    \Ptheta^N \lr{\Psi \theta_0 \in I_N} \longrightarrow 1- \alpha, \quad N \longrightarrow \infty.
  \end{align*}
\end{theorem}
We specify conditions for a renormalized BvM, where $\sigma_N^2$ represents the Cramér-Rao bound for $\Psi \theta_0$ in the high-dimensional model $\theta_0 + E_D$. This bound is well-defined, when the corresponding 'information matrix' $\fishermatrix(\theta_0) \in \mathbb{R}^{D \times D}$ is invertible, which is the case for Darcy's problem. Our general conditions include the regularity of $\Pi_N, (\psi_i)_{i \le k}$ and $\theta_0$, and the stability of $\fishermatrix(\theta_0)$.

We generalize the renormalized BvM for the vector-valued functional $\Psi : \Theta \longrightarrow \rk$. We then obtain conclusions for a credible ellipsoid, centered at the posterior mean, where the posterior variance captures the near-Gaussian shape of the posterior. We state an informal result for Darcy's problem.
\begin{theorem}
  Take $\mathcal{G}$ from Darcy's problem. Let $\mathcal{D}_N$ be generated from $\theta_0 \in C_c^\infty(\mathcal{X})$, and let $\psi_1, \ldots, \psi_k \in C_c^\infty(\mathcal{X})$ be linearly independent. 
  
  Equip a Gaussian prior $\Pi_N$ with regularity $\beta = \beta(d)$ large enough on an approximation space $E_D$ of dimension $D \asymp N^a$ with $a = a(\beta, d) > 0$ small enough.
  
  Let $\hat{\Psi}_N \in \rk$ and $\variance \in \rkbyk$ denote the mean and variance of $\Pi_N^\Psi(\cdot \mid \mathcal{D}_N)$, respectively. Calibrate the statistic $R_N$ such that
  \begin{align*}
     C_N\coloneqq \left\{ x \in \mathbb{R}^{k} : \big(x - \hat{\Psi}_N \big)^\intercal \hat{\Sigma}^{-1}_N \big(x - \hat{\Psi}_N \big)  \le R_N \right\},
  \end{align*}
  satisfies $\Pi_N^\Psi(C_N\mid \mathcal{D}_N) = 1- \alpha$. Then, 
  \begin{align}\label{eq:conintro}
    \Ptheta^N \lr{\Psi \theta_0 \in C_N} \longrightarrow 1 - \alpha \text{ and }\text{diam}\lr{C_N} = O_P\lr{D^{\frac{3}{d}} / \sqrt{N}}.
  \end{align}
  When $k = 1$, then \eqref{eq:conintro} holds with $C_N$ replaced by a credible interval of the form \eqref{eq:credinterval}.
\end{theorem}
We structure the remaining content as follows: We consider a semi-parametric Gaussian regression model and establish conditions for a renormalized BvM. We show that a renormalized BvM has implications for the frequentist validity of credible ellipsoids. We verify the conditions for Darcy's problem. We show that 'the Schrödinger problem' recovers a semi-parametric BvM from a renormalized BvM. We leave all proofs, including justifications for remarks, in the supplement \cite{Supplement}.
\section{Main Results}\label{sec:invprob}
\subsection{Notation}\label{subsec:notation}
Let $\mathcal{X} \subset \mathbb{R}^d$ be a Borel set with positive and finite Lebesgue mass. Let $\h$ be a Hilbert space with orthonormal basis $\{e_i \}_{i\in\mathbb{N}}$. Let $\Theta$ be a linear subspace of $\h$ containing $\text{span}\{e_i\}_{i \in \mathbb{N}}$. Let $\lr{V, \langle \cdot ,\cdot \rangle_V}$ be a finite-dimensional inner product space. Let $L^2(\mathcal{X} ; V)$ denote the square integrable functions from $\mathcal{X}$ to $V$ with respect to the uniform distribution $\lambda_\mathcal{X}$, see \cite[Chapter~4]{Diestel1977-id}. Let $\forward : \Theta \longrightarrow L^2(\mathcal{X} ; V)$ be measurable, and set $\mathcal{G}_\theta \coloneqq \mathcal{G}(\theta)$ for $\theta \in \Theta$. Let $\mathcal{N}_V$ denote the centered Gaussian measure on $(V, \langle \cdot ,\cdot \rangle_V)$ with identity covariance operator \cite[Section~2.3]{Bogachev1998-aw}. We consider the model
\begin{align*}
  Y_i = \forward_\theta(X_i) + \varepsilon_i, \quad i \in \{1, \ldots, N\},
\end{align*}
 with parameter $\theta \in \Theta$. We assume $X_i \sim \lambda_\mathcal{X}$, that $\varepsilon_i \sim \mathcal{N}_V$, and that $X_1, \varepsilon_1 \ldots, X_N, \varepsilon_N$ are independent.
 
\textit{Frequentist model}: Fix a ground truth $\truth \in \Theta$. Define the probability space $(\Omega, \mathcal{A}, \Ptheta )$ as follows: Let $\Omega \coloneqq (\mathcal{X} \times V)^\mathbb{N}$ with cylindrical $\sigma$-algebra $\mathcal{A}$. Let $\Ptheta  \coloneqq \mu_{\truth}^{\mathbb{N}}$, where $\mu_{\truth}$ is the law of $(X_1, Y_1)$ under $\truth$. Let $\Etheta $ denote expectation under $\Ptheta$. We represent the observations $\mathcal{D}_N \coloneqq ((X_i, Y_i))_{i\le N} $ as a coordinate projection $\mathcal{D}_N: \Omega \longrightarrow (\mathcal{X} \times V)^N$. Let $\bs{\psi} = \lr{\psi_i}_{i\le k} \in \h^k$ be linearly independent. We consider the functional $\Psi\theta \coloneqq \lr{\hprod{\theta}{\psi_i}}_{i\le k}$ for $\theta \in \Theta$.

\textit{Bayesian model}: Let $E_D \coloneqq \text{span}\{e_i\}_{i \le D}$, where $D = D(N) \longrightarrow \infty$. Define the projector $P_{E_D} : \h \longrightarrow E_D$. Let $\langle \cdot, \cdot \rangle_{\pspace}$ be an inner-product on $E_D$, let the prior $\Pi_N$ be a centered Gaussian measure on $\lr{E_D, \langle \cdot, \cdot \rangle_{\pspace}}$ with identity covariance operator, and let $\theta_N \sim \Pi_N$. For $f \in L^1(\Pi_N)$, define $\mathbb{E}^{\Pi}f(\theta_N) \coloneqq \int_{E_D} f \, d\Pi_N$. Define the posterior 
\begin{align}\label{eq:oripos}
  \Pi_N\lr{A \mid \mathcal{D}_N} \coloneqq \frac{\int_A e^{\ell_N(\theta)} \, d\Pi_N(\theta)}{ \int_{E_D}e^{\ell_N(\theta)} \, d\Pi_N(\theta)}, \quad \text{for Borel sets }A\in \mathcal{B}(E_D),
\end{align}
 where $\ell_N(\theta) \coloneqq -\frac{1}{2}\sum_{i=1}^N \norm{Y_i - \forward_\theta(X_i)}{V}^2$ is the log-likelihood, up to an additive constant. For a measurable function $f: E_D \times \lr{\mathcal{X} \times V}^N \longrightarrow \mathbb{R}^n$, define $\mathcal{L}\lr{f(\theta_N, \mathcal{D}_N) \mid \mathcal{D}_N}$ as the push-forward of $\Pi_N(\cdot \mid \mathcal{D}_N)$ under $f(\cdot, \mathcal{D}_N)$. Define $\Pi^\Psi_N(\cdot \mid \mathcal{D}_N) \coloneqq \mathcal{L}\lr{\Psi \theta_N \mid \mathcal{D}_N}$. We use $\mathbb{E}[ f(\theta_N, \mathcal{D}_N) \mid \mathcal{D}_N]$ or $\int_{E_D} f(\theta_N, \mathcal{D}_N) \, d\Pi_N(\cdot \mid \mathcal{D}_N)$ to denote the expectation of $\mathcal{L}\lr{f(\theta_N, \mathcal{D}_N) \mid \mathcal{D}_N}$. We use $\mathbb{V}[ f(\theta_N, \mathcal{D}_N) \mid \mathcal{D}_N]$ for its variance. By \cite[Lemma~4.1]{Supplement}, we can define $\hat{\Psi}_N \coloneqq \mathbb{E}[\Psi \theta_N \mid \mathcal{D}_N]$ and $\variance \coloneqq \mathbb{V}(\Psi \theta_N \mid \mathcal{D}_N)$. For $B \in \mathcal{B}(E_D)$ satisfying $\Pi_N(B) > 0$, we define the conditional posterior
 \begin{align}\label{eq:condpos}
  \Pi_N^B\lr{A \mid \mathcal{D}_N} \coloneqq \frac{\int_{A \cap B} e^{\ell_N(\theta)} \, d\Pi_N(\theta)}{\int_{B}e^{\ell_N(\theta)} \, d\Pi_N(\theta)}, \quad A\in \mathcal{B}(E_D).
\end{align}
Similar to before, we introduce the notation $\mathcal{L}^B\lr{f(\theta_N, \mathcal{D}_N) \mid \mathcal{D}_N}$, $\mathbb{E}^B[f(\theta_N, \mathcal{D}_N)  \mid \mathcal{D}_N]$ and $\int_{E_D} f(\theta_N, \mathcal{D}_N) \, d\Pi_N^B(\cdot \mid \mathcal{D}_N)$.

\textit{Limit notation}: We take limits for $N \longrightarrow \infty$, unless otherwise specified. We define the standard notation $o_P, O_P$ and $\stackrel{P}{\longrightarrow}$ with respect to $\Ptheta $. For sets $A_N \subset \Theta$, and functions $f_N, g_N : E_D \times \lr{\mathcal{X} \times V}^N \longrightarrow \mathbb{R}$, we say
\begin{align}\label{eq:unifop}
  f_N(\theta, \mathcal{D}_N) = g_N(\theta, \mathcal{D}_N) + o_P(1), \quad \text{uniformly over $\theta \in A_N$},
\end{align}
if $\sup_{\theta \in A_N}\lvert f_N(\theta, \mathcal{D}_N) - g_N(\theta, \mathcal{D}_N)\rvert \plim 0$. We use similar notation, replacing $o_P(1)$ by $o(1)$ and $\mathcal{D}_N$ by $d_N \in \mathbb{R}$. Let $d_{\text{BL}}$ denote the bounded Lipschitz metric for weak convergence. Let $(\mu_N)$ and $\mu$ be arbitrary random distributions on $(\mathbb{R}^n, \mathcal{B}(\mathbb{R}^n))$.  We say  
$\mu_N \dlim \mu$ in $\Ptheta$-probability, if $d_{\text{BL}}\lr{\mu_N , \mu} \plim 0$.

\textit{General notation}: Let $\lesssim$ denote inequality up to a constant independent of $N$, unless otherwise specified. Let $\bs{e}_i$ denote the $i$th standard basis vector in $\rk$. Let $I_k \in \rkbyk$ denote the identity matrix, and $\text{GL}_k(\mathbb{R}) \coloneqq \curly{A \in \rkbyk : A \text{ is invertible}}$. For $A \in \rkbyk \setminus \text{GL}_k(\mathbb{R})$, we arbitrarily set $A^{-1} 
\coloneqq I_k$. Let $\Vert \cdot \Vert_{\mathbb{R}^n}$ denote the standard Euclidean norm on $\mathbb{R}^n$, and let $\Vert \cdot \Vert_{\text{op}}$ denote the operator norm for a matrix or a general linear map. Let $\lambda_{\text{min}}$ and $\lambda_{\text{max}}$ denote the minimum and maximum eigenvalue functions, respectively. For a metric space $(X, d)$, let $N(A, d; \varepsilon)$ denote the $\varepsilon$-covering number of the set $A \subset X$. Let $Q_{\chi^2_k}$ denote the quantile function of the $\chi_k^2$ distribution. Let $\mathcal{L}(Y)$ denote the law of a random variable $Y$.

\textit{Vectorized notation}: Let $(W, \Vert \cdot \Vert_W)$ be any normed vector space, and let $\bs{h} \in W^m$ where $m \in \mathbb{N}$. Define $\Vert \bs{h} \Vert_{W} \coloneqq \max_{i \le m}\Vert h_i \Vert_W$. If $\norm{\cdot}{W}$ is induced by an inner product $\langle \cdot, \cdot \rangle_W$, define $\langle g, \bs{h} \rangle_{W} \coloneqq \lr{\langle g, h_i \rangle_{W}}_{i \le m}$. For $A \in \mathbb{R}^{n \times m}$, define $A\bs{h} \in W^n$ by $(A\bs{h} )_i = \sum_{j = 1}^m A_{ij}h_j$ for $i \le n$. Clearly, $A\langle g, \bs{h} \rangle_W = \langle g, A\bs{h} \rangle_W$. For a linear operator $T : W \longrightarrow W'$ mapping to another normed vector space $(W', \Vert \cdot \Vert_{W'})$, define $T\bs{h} = \lr{Th_i}_{i \le m}$. Note $AT\bs{h} = TA\bs{h}$ and $\Vert AT\bs{h} \Vert_{W'} \le C \Vert A \Vert_{\text{op}} \Vert T\Vert_{\text{op}} \Vert \bs{h} \Vert_{W}$ for some $C = C(n, m) > 0$. Let $B_W(x, r)$ denote the closed ball centered at $x \in W$ with radius $r > 0$, and $B_W(r) \coloneqq B_W(0, r)$.

\subsection{Hypotheses}\label{sec:hyp}
We impose a posterior contraction condition, inspired by \cite[Theorem~1.3.2]{Nickl2023}, \cite[Lemma~1.3.3]{Nickl2023}, the remark from the beginning of the proof of \cite[Theorem~2.2.2]{Nickl2023}, and the stability condition \cite[Condition~2.1.4]{Nickl2023}.
\begin{hypothesis}\label{cond:posteriorcontraction}
  There exists a normed vector space $\mathcal{R} \coloneqq \{x \in \h : \Vert x \Vert_{\mathcal{R}} < \infty \}$ with $\truth \in \mathcal{R}$. There exist sequences $\delta_N^\forward = N^{-r_1}$ and $\delta_N^\h = N^{-r_2}$ with $r_1, r_2 \in (0, 1/2)$. For all $M > 0$, there exists $M' > 0$ such that
  \begin{align*}
    \regset \coloneqq \curly{\theta \in E_D :   \Vert \forward_\theta - \forward_{\truth} \Vert_{L^2} \le M\delta_N^\forward  ,  \Vert \theta \Vert_\mathcal{R} \le M }
  \end{align*}
  satisfies $\regset \subset \curly{\theta \in \h : \Vert \theta - \truth\Vert_\h \le M' \delta^\h_N}$ for all $N \in \mathbb{N}$.
  
  There exist $M_0 > 0$ and $0 < a < b < c$ with the following property: If $\mathcal{B}_N \subset \Theta_{N, M_0}$ and $\prior{\mathcal{B}_N} \ge 1 - e^{-cs_N}$ for $N$ large enough, then
   \begin{enumerate}
    \item[1.] $\posterior{\mathcal{B}_N} = 1 - O_P\lr{e^{-bs_N}}$,
    \item[2.] $\Ptheta \left(\int_{\mathcal{B}_N} e^{\ell_N - \ell_N(\truth)} \, d\Pi_N > e^{-as_N}\right) \longrightarrow 1$,
  \end{enumerate}
  where $s_N \coloneqq N\lr{\delta_N^\mathcal{G}}^2 \longrightarrow \infty$. 
\end{hypothesis}
\begin{hypothesis}\label{cond:2}
There exists a bounded linear operator $\lin  : \h \longrightarrow L^2(\mathcal{X}; V)$ such that for all $h\in\Theta$, we have $\norm{\forward(\truth + \varepsilon h) - \forward(\truth) - \varepsilon\lin [h]}{L^2}  = o(\varepsilon)$ as $\varepsilon \longrightarrow 0$. 
\end{hypothesis}
Define the linearization error $R_\theta \coloneqq \forward_\theta - \forward_{\truth} - \lin [\theta - \truth]$ for $\theta \in \Theta$.
\begin{hypothesis}\label{cond:4}
  There exists a metric $d_\Theta$ on $\Theta$, and sequences
  \begin{enumerate}
    \item $\delta^{\Theta}_N \gtrsim \sup_{\theta \in \regset}d_\Theta(\theta, \truth)$,
    \item $\sigma_N \gtrsim \sup_{\theta \in \regset} \norm{R_\theta}{L^2}$,
    \item $g_N \gtrsim \sup_{\theta, \theta' \in \regset} \norm{\forward_\theta - \forward_{\theta'}}{L^\infty}/ d_\Theta(\theta, \theta')$,
    \item $r_N \gtrsim \sup_{\theta, \theta' \in \regset} \norm{R_\theta - R_{\theta'}}{L^\infty}/ d_\Theta(\theta, \theta')$,
    \item $J_{N}(t) \gtrsim \int_0^t \sqrt{\log 2 N(\regset, d_\Theta; \varepsilon)} \, d\varepsilon$,
  \end{enumerate}
  where the constants are independent of $N \in \mathbb{N}$ and $t \ge 0$, but may depend on $M > 0$. 
\end{hypothesis}

We define the 'information operator' $\mathcal{I} \coloneqq \lin ^*\lin $, as in \cite{Nickl2023x}, and let $\fishermatrix \coloneqq P_{E_D}\mathcal{I}\iota_{E_D} $, where $\iota_{E_D} : E_D \longrightarrow \h$ is the inclusion map. 
\begin{hypothesis}\label{cond:inv}
The operator $\fishermatrix : E_D \longrightarrow E_D$ is invertible for all $D\in \mathbb{N}$.
\end{hypothesis}
Let $[\fishermatrix] \in \mathbb{R}^{D \times D}$ represent $\fishermatrix$ in the basis $\{e_1, \ldots, e_D\}$:
\begin{align}\label{eq:finiteFisher1}
  [\fishermatrix]_{ij} \coloneqq \ip{\fishermatrix e_i}{e_j}{L^2} , \quad i, j \le k.
\end{align}
\begin{remark}\label{rem:submodel}
  We can interpret $\sq{\fishermatrix}$ as the information matrix at $\truth$ in the high-dimensional model $\truth + E_D =\left\{\truth + \sum_{i = 1}^D h_ie_i : h \in \mathbb{R}^D \right\}$. In this model, the Cramér-Rao bound for $\Psi \truth$ is then represented by
  \begin{align}\label{eq:crdef}
    \Sigma_N \coloneqq \frac{1}{N}\mathcal{J}_D [\fishermatrix]^{-1} \mathcal{J}_D^\intercal,
  \end{align}
  where $\mathcal{J}_D \in \mathbb{R}^{k \times D}$ is the Jacobian of the map $h \in \mathbb{R}^D \mapsto \Psi\lr{\truth + \sum_{i = 1}^D h_ie_i} $ given by
  \begin{align*}
    \lr{\mathcal{J}_D}_{ij} \coloneqq \langle \psi_i, e_j \rangle_{\h}, \quad \text{for }i \le k \text{ and } j \le D.
  \end{align*}
\end{remark}
\subsection{A Renormalized Bernstein-von Mises Theorem}\label{sec:ren}
In this section, we establish a $\mathcal{N}\bigpar{0, \Sigma_N}$ approximation for $\mathcal{L}\bigpar{\Psi \theta_N - \hat{\Psi}_N \mid \mathcal{D}_N}$ and $\mathcal{L}\bigpar{\hat{\Psi}_N - \Psi \theta_0}$. 

By \cite[Lemma~4.2]{Supplement}, we can properly define
\begin{align}\label{eq:renormdef}
  i_D \coloneqq \lr{\mathcal{J}_D [\fishermatrix]^{-1} \mathcal{J}_D^\intercal}^{-1}, \quad \text{for $D$ large enough}.
\end{align}
We present the $\mathcal{N}(0, \Sigma_N)$ approximation as a weak $\mathcal{N}(0, I_k)$ limit, after multiplication by $\Sigma_N^{-\frac{1}{2}} = \sqrt{N}\renorm$. We draw inspiration from the semi-parametric BvM techniques \cite{Nickl_2020,Monard_2021, Nickl2023, castillo2014bernstein,Castillo2015, Castillo2021}.

Define the statistic
\begin{align}\label{eq:centering}
  \Psi_N \coloneqq \Psi\truth + \frac{1}{N}\sum_{i = 1}^N \ip{\varepsilon_i}{\lin \perturb(X_i)}{V},
\end{align}
where $\perturb \in E_D^k$ has components $\bar{\psi}_{D}^i \coloneqq \fishermatrix^{-1}P_{E_D}\psi_i$ for $i \le k$. Here, we recall the vectorized notation from Section~\ref{subsec:notation}, leading to coordinates
\begin{align*}
  \Psi_N^j \coloneqq  \lr{\Psi\truth}_j + \frac{1}{N}\sum_{i = 1}^N \ip{\varepsilon_i}{\lin \pertur^j(X_i)}{V}, \quad j \le k.
\end{align*}
The reader should think of $\Psi_N $ as a preliminary statistic, that will later be replaced by $\hat{\Psi}_N$. 

Recall the notation $\norm{\lin \perturb}{L^\infty} = \text{max}_{i \le k} \Vert \lin \pertur^i \Vert_{L^\infty}$ from Section~\ref{subsec:notation}. Arguing by the central limit theorem, we obtain a $\mathcal{N}\lr{0, \Sigma_N}$ approximation for $\mathcal{L}\bigpar{\Psi_N - \Psi \theta_0}$:
  \begin{lemma}\label{lem:asymptnorm}
  Suppose $\norm{\lin \perturb}{L^\infty} \lesssim N^r$ for some $r \in (0, 1/2)$. Then, 
  \begin{align}\label{eq:inlem}
    \sqrt{N}\renorm\lr{\Psi_N - \Psi\truth} \dlim \mathcal{N}(0, I_k).
  \end{align}
\end{lemma}
We proceed to establish a $\mathcal{N}(0, \Sigma_N)$ approximation for $\mathcal{L}\lr{\Psi \theta_N - \Psi_N\mid \mathcal{D}_N}$. More precisely, let 
\begin{align}\label{eq:zeq}
  Z_N \coloneqq \sqrt{N}\renorm\lr{\Psi \theta_N - \Psi_N},
\end{align}
and recall the notion of weak convergence in probability from Section~\ref{subsec:notation}. We will show $\mathcal{L}\lr{Z_N \mid \mathcal{D}_N} \dlim \mathcal{N}(0, I_k)$ in $\Ptheta$-probability. We argue by convergence of the Laplace transform
\begin{align}\label{eq:abovelimit}
  \Epostreg{\bar{\Theta}_N}{\exp\lr{\lambda^\intercal Z_N}} \plim \exp\lr{\frac{1}{2} \norm{\lambda}{\rk}^2}, \quad \text{for all }\lambda \in \rk,
\end{align}
where $\bar{\Theta}_N \subset \Theta$ are sets of high probability.

The left-hand side of \eqref{eq:abovelimit} is a ratio of prior expectations, which include a likelihood factor. To obtain the limit, we shift the likelihood using an appropriate 'perturbation vector', and control the difference.

Recall matrix multiplication with elements of $\h^k$ from Section~\ref{subsec:notation}, and define the perturbation vector
\begin{align}\label{eq:perdef}
   \bs{h}_N  \coloneqq \frac{1}{\sqrt{N}}\renorm\perturb \in \h^k.
\end{align} 
It will be important to control the magnitude of $\bs{h}_N$.

Our next two lemmas establish the aforementioned likelihood shift. The reader should recall \eqref{eq:unifop}, and $s_N$ from Hypothesis~\ref{cond:posteriorcontraction}.
\begin{lemma}\label{lem:empiricalprocess}
  Fix $\lambda \in \rk$ and $M > 0$. Assume:
  \begin{enumerate}
    \item $\sup_{\theta \in \regset} \norm{\forward\lr{\theta +  \lambda^\intercal \bs{h}_N }  - \forward\lr{\theta} }{L^2} = o\lr{\delta_N^\mathcal{G}}$, 
    \item $\norm{\perturb}{\mathcal{R}} = o\lr{\sqrt{N}}$,
    \item $\sqrt{N}r_N  J_N\lr{\sigma_N /r_N} \longrightarrow 0$,
    \item $\sqrt{\log N} r_N^3 \delta_N^\Theta J_{N}(\sigma_N /r_N)^2 /\sigma_N^2 \longrightarrow 0$,
    \item $\sqrt{N} g_N^2 \delta_N^\Theta J_{N}\lr{\delta_N^\Theta} \longrightarrow 0$,
    \item $g_N J_{N}\lr{\delta_N^\Theta} \longrightarrow 0$.
  \end{enumerate}
  Then, uniformly over $\theta \in \regset \cup \lr{\regset -  \lambda^\intercal \bs{h}_N }$, we have
  \begin{align}\label{eq:emplimits1}
    \sum_{i=1}^N \ip{\varepsilon_i}{R_\theta(X_i)}{V} &= o_P(1), \\ \sum_{i = 1}^N \lr{\lvert \forward_{\truth, \theta}(X_i) \rvert_V^2 - \Vert \forward_{\truth, \theta} \Vert_{L^2}^2} &= o_P(1). \label{eq:emplimits2}
  \end{align}
\end{lemma}
\begin{remark}\label{rem:lip}
  If $\forward : \lr{\Theta \cap B_\mathcal{R}(M), \norm{\cdot}{\h}} \longrightarrow \lr{L^2(\mathcal{X}), \norm{\cdot}{L^2}}$ is Lipschitz continuous for all $M > 0$, then $\sqrt{s_N} \eigmin{\fishermatrix} \longrightarrow \infty$ implies Assumption~1.
\end{remark} 
\begin{lemma}\label{lem:lla}
  Fix $\lambda \in \rk$ and $M > 0$. Grant the assumptions of Lemma~\ref{lem:empiricalprocess}. Assume:
  \begin{enumerate}
    \item $N\lr{\sigma_N^2 + \sigma_N \delta^\h_N } \longrightarrow 0$,
    \item $\sqrt{N}\norm{\truth - P_{E_D}\truth}{\h} \longrightarrow 0$,
    \item $\sqrt{N} \delta^\h_N \norm{\bs{\psi} - P_{E_D}\bs{\psi}}{\h} \longrightarrow 0$.
  \end{enumerate}
  Let $Z_N(\theta) \coloneqq \sqrt{N}\renorm\lr{\Psi\theta - \Psi_N}$ for $\theta \in \Theta$. Then,
\begin{align}\label{eq:llr}
  \ell_N(\theta) - \ell_N\lr{\theta -  \lambda^\intercal \bs{h}_N }  = \frac{1}{2}\Vert \lambda \Vert^2_{\rk} - \lambda^\intercal Z_N(\theta) + o_P(1), \quad \text{uniformly over }\theta \in \regset.
\end{align} 
\end{lemma}
We now obtain a posterior normal approximation around $\Psi_N$, along with convergence of moments. Before stating the result, we recall $a < b$ from Hypothesis~\ref{cond:posteriorcontraction}, and \eqref{eq:zeq}.
\begin{lemma}\label{lem:prjbvm}
  Grant the assumptions of Lemma~\ref{lem:lla}. Assume:
  \begin{enumerate}
    \item $\delta_N \norm{\bar{\bs{\psi}}_{D}}{\pspace} \longrightarrow 0$,
    \item $N e^{-(b - a)s_N} \mathbb{E}^\Pi \norm{\theta_N}{\h}^4  \longrightarrow 0$.
  \end{enumerate}
   Then,
   \begin{align}\label{eq:distlim}
    \mathcal{L}\lr{Z_N \mid \mathcal{D}_N} &\dlim \mathcal{N}(0, I_k), \quad \text{in $\Ptheta$-probability}, \\
    \mathbb{E}[Z_N \mid \mathcal{D}_N] &\plim 0 \label{eq:expp}, \\
    \mathbb{V}[Z_N \mid \mathcal{D}_N] &\plim I_k \label{eq:var}.
\end{align}
\end{lemma}
We now interchange $\Psi_N$ with $\hat{\Psi}_N$ in \eqref{eq:inlem} and \eqref{eq:distlim}. We justify this by Slutsky's lemma and \eqref{eq:expp}, which states $\sqrt{N}\renorm\bigpar{\Psi_N - \hat{\Psi}_N } \plim 0$.

This yields our first main result: the posterior $\Pi_N^\Psi(\cdot \mid \mathcal{D}_N)$, centered as its mean, and the posterior mean $\hat{\Psi}_N$, centered at the ground truth $\Psi \theta_0$, both have an asymptotic $\mathcal{N}(0, \Sigma_N)$ approximation.
\begin{theorem}[Renormalized BvM]\label{thm:rnm}
  Grant the conclusions of Lemma~\ref{lem:asymptnorm} and Lemma~\ref{lem:prjbvm}. Then,
  \begin{align}\label{eq:finaldlim}
    \mathcal{L}\lr{\sqrt{N}i_{D}^{1/2}\bigpar{\Psi\theta_N-\hat{\Psi}_N} \mid \mathcal{D}_N} &\dlim \mathcal{N}\lr{0, I_k}, \quad \text{in } \Ptheta \text{-probability},
    \\ \label{eq:finalplim}
    \sqrt{N}i_{D}^{1/2}\bigpar{\hat{\Psi}_N - \Psi\truth} &\dlim \mathcal{N}\lr{0, I_k}.
  \end{align}
\end{theorem}
Conditions for a semi-parametric BvM were derived in \cite{Nickl_2020,Monard_2021, Nickl2023}. A central condition is the existence of a sufficiently regular solution $\bs{\phi}$ to the information equation $\mathcal{I} \bs{\phi} = \bs{\psi}$. We obtain a similar result:
\begin{theorem}[Semi-Parametric BvM]\label{thm:recov}
  Assume there exist $\bs{\phi} = (\phi_1, \ldots, \phi_k) \in \mathbb{H}^k$ such that:
  \begin{enumerate}
    \item $\mathcal{I}\bs{\phi} = \bs{\psi}$,
    \item $\norm{\bs{\phi} - \proj \bs{\phi}}{\h} / \eigmin{\fishermatrix} \longrightarrow 0$,
    \item $\lin $ is injective on $\text{span}\{\phi_1, \ldots, \phi_k \}$.
  \end{enumerate}
  Then $i_{D}^{-1} \longrightarrow  L$, where $L_{ij} \coloneqq \ip{\phi_i}{\psi_j}{\h}$ for $i,j \le k$. 
  
  Assume further the conclusions of Theorem~\ref{thm:rnm}. Then,
  \begin{align}\label{eq:sem2}
    \mathcal{L}\lr{\sqrt{N}\bigpar{\Psi\theta_N-\hat{\Psi}_N} \mid \mathcal{D}_N} &\dlim \mathcal{N}(0, L), \quad \text{in } \Ptheta \text{-probability},
    \\ \label{eq:sem1}
    \sqrt{N}\bigpar{\hat{\Psi}_N - \Psi\truth } &\dlim \mathcal{N}(0, L).
  \end{align}
\end{theorem}

\subsection{Validity of Credible Sets}\label{sec:cred}
In the next lemma, we construct pivotal quantities from Theorem~\ref{thm:rnm}. In short, we apply $\norm{\cdot}{\rk}^2$ to \eqref{eq:finaldlim} and \eqref{eq:finalplim}, and then interchange $N i_D$ with $\hat{\Sigma}_N^{-1}$. We justify the interchange by \eqref{eq:var}, which yields the consistency result $N\renorm \variance \renorm \plim I_k$.
\begin{lemma}\label{lem:pivot}
  Grant the conclusions of Theorem~\ref{thm:rnm}, and \eqref{eq:var}.
  
  Then, $\Ptheta\lr{\variance \in \text{GL}_k(\mathbb{R})}\longrightarrow 1$, and
  \begin{align}\label{eq:pivot1}
    \mathcal{L}\lr{\qform{\Psi \theta_N - \hat{\Psi}_N}{\variance^{-1}}{\Psi \theta_N - \hat{\Psi}_N} \mid \mathcal{D}_N} &\dlim \chi_k^2, \quad \text{in }\Ptheta\text{-probability},\\ \label{eq:pivot2}
    \qform{\hat{\Psi}_N - \Psi \theta_0}{\variance^{-1}}{\hat{\Psi}_N - \Psi \theta_0}  &\dlim \chi_k^2.
  \end{align}
\end{lemma}
We arrive at our second main result: A credible ellipsoid, specified by the mean and variance of $\Pi_N^\Psi(\cdot \mid \mathcal{D}_N)$, is an exact asymptotic confidence set for $\Psi \theta_0$. The same holds for a 'Wald-type' ellipsoid. In both cases, the diameter is asymptotically equivalent, in probability, to $\sqrt{\eigmax{\Sigma_N}}$, where we recall \eqref{eq:crdef} and \eqref{eq:renormdef}.
\begin{theorem}[Credible Ellipsoids]\label{thm:crd}
Grant the conclusions of Lemma~\ref{lem:pivot}. 

Fix $\alpha \in (0, 1)$. Let
  \begin{align}\label{eq:confelp}
    C_N \coloneqq \left\{ x \in \rk :\big(x -  \hat{\Psi}_N \big)^\intercal \variance^{-1}\big(x -  \hat{\Psi}_N \big)\le R_N \right \}, 
  \end{align}
  where
  \begin{enumerate}
    \item $R_N = Q_{\chi^2_k}(1 - \alpha)$ for all $N \in \mathbb{N}$, or \label{en:case1}
    \item $R_N$ is a statistic such that $\Pi_N^\Psi( C_N \mid \mathcal{D}_N) = 1 -\alpha$ for all $N \in \mathbb{N}$. \label{en:case2}
  \end{enumerate}
  Then,
  \begin{align}\label{eq:consstate}
    \Ptheta \lr{\Psi\truth \in C_N} \longrightarrow 1 - \alpha.
  \end{align}
  Furthermore,
  \begin{align}\label{eq:diambound}
    \text{diam}\lr{C_N} = \lr{1 + o_P(1)} \sqrt{4Q_{\chi_k^2}(1 - \alpha) \cdot \frac{\eigmax{i_D^{-1}}}{N}},
  \end{align}
  and $\text{diam}\lr{C_N} \lesssim (1 + o_P(1))/\sqrt{N\eigmin{\fishermatrix}}$.
\end{theorem}
\begin{remark}
  We provide some intuition for the proof. Consider Case~\ref{en:case1} of $R_N$. Then, \eqref{eq:consstate} follows directly by \eqref{eq:pivot2}. Consider Case~\ref{en:case2} of $R_N$. Then, $R_N$ is a $(1 - \alpha)$-quantile of the left-hand side of \eqref{eq:pivot1}. By convergence of quantiles, $R_N \plim Q_{\chi_k^2}(1 - \alpha)$. This allows us to transfer the asymptotic properties from Case~\ref{en:case1} to Case~\ref{en:case2}.
\end{remark}
If the posterior $\Pi_N^\Psi(\cdot \mid \mathcal{D}_N)$ were truly Gaussian, then $\partial C_N$ from \eqref{eq:confelp} would represent a level set of its density. However, under a renormalized BvM, the posterior is only approximately Gaussian. We can thus, informally, think of $\partial C_N$ as an approximate level set of the posterior density, that captures its near-Gaussian shape.

When $\Psi$ is a scalar functional, we can identify credible intervals and credible ellipsoids. This yields the following corollary:
\begin{corollary}[Credible Intervals]\label{cor:int}
  Grant the conclusions of Theorem~\ref{thm:crd}. 
  
  Let $k = 1$ and $\alpha \in (0, 1)$. For all $N \in \mathbb{N}$, let $R_N$ be a statistic such that
  \begin{align}\label{eq:interval}
    I_N = \big[\hat{\Psi}_N - R_N, \hat{\Psi}_N + R_N \big]
  \end{align}
  satisfies $\Pi_N^\Psi \lr{ I_N  \mid \mathcal{D}_N} = 1 - \alpha$. Then, $\Ptheta \lr{\Psi\truth \in I_N} \longrightarrow 1 - \alpha$. Furthermore, $\text{diam} \lr{I_N}$ has the properties stated for $\text{diam}\lr{C_N}$ in Theorem~\ref{thm:crd}.
\end{corollary}
\section{Applications}\label{sec:4}
\subsection{Setting}\label{sec:setopno}
Let $\mathcal{X} \subset \mathbb{R}^d$ be a smooth domain and $\h \coloneqq L^2(\mathcal{X})$. Take $V \coloneqq \mathbb{R}$, corresponding to $\varepsilon_1 \sim \mathcal{N}\lr{0, \sigma_0^2}$ for some fixed $\sigma_0^2 > 0$. Define $\Theta \coloneqq C_{0, \Delta}^\infty(\mathcal{X})$, where 
\begin{align*}
  C_{0, \Delta}^\infty(\mathcal{X}) \coloneqq \left\{ f \in C^\infty(\bar{\mathcal{X}}) : \Delta^n f|_{\partial \mathcal{X}} = 0 \text{ for all  } n \in \mathbb{N}_0 \right\}. 
\end{align*}
Let $H^\beta_0(\mathcal{X})$ denote the Sobolev-Hilbert space with weak derivatives up to order $\beta \in \mathbb{N}$, and vanishing trace on $\partial \mathcal{X}$. Let $\curly{e_i}_{i \in \mathbb{N}}$ be an orthonormal eigenbasis of the negative Laplacian 
\begin{align*}
  -\Delta : H^2_0(\mathcal{X}) \longrightarrow L^2(\mathcal{X})
\end{align*}
with eigenvalues $\curly{\lambda_i}_{i \in \mathbb{N}}$ in ascending order, and set $E_D \coloneqq \text{span}\{e_1, \ldots, e_D \}$. Define the 'spectral' Sobolev space
\begin{align*}
  h^\beta(\mathcal{X}) \coloneqq \curly{x \in L^2(\mathcal{X}) : \Vert x \Vert_{h^\beta} < \infty }, \quad \text{where }\Vert x \Vert_{h^\beta}^2 \coloneqq \sum_{j \in \mathbb{N}} \lvert \langle x, e_j \rangle \rvert^2 \lambda_j^\beta.
\end{align*}
Let $\alpha \in \mathbb{N}$ satisfy $\alpha > 1 + d/2$, and define
 \begin{align}\label{eq:recpi}
  \Pi_N \coloneqq \mathcal{L}\lr{N^{-\frac{d}{4\alpha + 2d}}\sum_{i = 1}^D \lambda_i^{-\frac{\alpha}{2}}  W_i e_i}, \quad W_i \stackrel{\text{iid}}{\sim} \mathcal{N}(0, 1).
 \end{align}
\subsection{Darcy's Problem}\label{sec:darcy}
Recall Darcy's problem from the introduction. Take $d \le 3$ for simplicity. Fix $g\in C^\infty(\partial \mathcal{X})$, and a positive $f \in C^\infty(\bar{\mathcal{X}})$. Let $\forward$ map $\theta \in \Theta$ to the solution of 
\begin{align}\label{eq:darcypde}
  \nabla \cdot \bigpar{e^\theta\nabla u} = f, \quad u|_{\partial \mathcal{X}} = g.
\end{align}
A unique solution $u = u_\theta \in C^\infty(\bar{\mathcal{X}})$ is guaranteed by \cite[Section~A.2]{Nickl2023}. 

A general semi-parametric BvM cannot exist for $\psi_1, \ldots, \psi_k \in C_c^\infty(\mathcal{X})$, see \cite{Nickl2022}. Instead, we obtain the following result:
\begin{theorem}\label{thm:drc}
  Let $\psi_1, \ldots, \psi_k \in C_{0, \Delta}^\infty(\mathcal{X})$ and $\alpha \ge 14$. Assume $N^l \lesssim D \lesssim N^u$, where 
  \begin{align}\label{eq:boundonu}
    0 < l < u < \frac{d^2}{(4\alpha + 2d)(\alpha + 6)}.
  \end{align}
  Then, the conclusions of Theorem~\ref{thm:rnm}, Theorem~\ref{thm:crd} and Corollary~\ref{cor:int} hold. In addition, the diameters of \eqref{eq:confelp} and \eqref{eq:interval} are smaller than $C (1 + o_P(1))D^{\frac{3}{d}} / \sqrt{N}$ for some $C > 0$.
\end{theorem}
\begin{remark}
  Assume instead $\theta_0 \in \Theta \coloneqq h^{\gamma_1}(\mathcal{X})$ and $\psi_1, \ldots, \psi_k \in h^{\gamma_2}(\mathcal{X})$, where $\gamma_1, \gamma_2 > 0$. Take also $\gamma_1$ large enough that $\mathcal{G}$ retains the estimates from \cite[Section~4.0.1]{Supplement}.
  
  To confirm the conditions from Section~\ref{sec:invprob}, we only need to reevaluate Assumption 2-3 of Lemma~\ref{lem:lla}. By \cite[Lemma~4.5]{Supplement}, we can derive a sufficient condition of the form $l > l^*(\alpha, d, \gamma_1, \gamma_2)$, where $l^*(\alpha, d, \gamma_1, \gamma_2) \longrightarrow 0$ for $\gamma_1, \gamma_2 \longrightarrow \infty$. 
\end{remark}
\subsection{The Schrödinger Problem}\label{sec:schr}
Fix a positive $g \in C^\infty(\partial \mathcal{X})$. Let $\mathcal{G}$ map $\theta \in \Theta$ to the solution of
\begin{align}\label{eq:schrodinger}
  -\frac{1}{2}\Delta u + e^\theta u = 0, \quad u|_{\partial \mathcal{X}} = g.
\end{align}
A unique solution $u = u_\theta \in C^\infty(\bar{\mathcal{X}})$ is guaranteed by \cite[Section~A.2]{Nickl2023}. We consider $\psi_1, \ldots, \psi_k \in C_c^\infty(\mathcal{X})$. A semi-parametric BvM was proved in similar settings, see \cite{Nickl_2020,Nickl2023}.

 We can verify the conditions for our renormalized BvM by constraining $\alpha$ and $D$. This proceeds similarly to Darcy's problem, using the results from \cite{Nickl2023}. We recover a semi-parametric BvM, and a $\sqrt{N}$-efficient diameter for our credible ellipsoid:
\begin{theorem}
Assumptions 1-3 of Theorem~\ref{thm:recov} hold. Furthermore, the diameters of \eqref{eq:confelp} and \eqref{eq:interval} take the form $C(1 + o_P(1))/\sqrt{N}$, where $C > 0$. 
\end{theorem}
\section{Discussion}
We finish by noting down some desirable properties for satisfying our conditions from Section~\ref{sec:invprob}:
\begin{enumerate}
  \item Posterior contraction rates of the form $N^{-r}$, where $r \approx 1/2$.
  \item A linearization error, $\norm{R_\theta}{L^2} \lesssim \norm{\theta - \theta_0}{\h}^p$ for some $p > 1$.
  \item A stability estimate, $\eigmin{\fishermatrix} \gtrsim D^{-\kappa}$ for some $\kappa > 0$. \label{en:prop}
  \item A projection error, $\norm{\theta_0 - P_{E_D}\theta_0}{\h} + \norm{\bs{\psi} - P_{E_D}\bs{\psi}}{\h} \lesssim D^{-r}$ for some $r > 0$.
\end{enumerate}

The diameter of the credible ellipsoid from Theorem~\ref{thm:crd} may be far from $\sqrt{N}$-efficient. However, if the above stability estimate holds, then the diameter is $O_P(N^{-\frac{1}{2} + \varepsilon})$ for $\varepsilon > 0$, when $D \lesssim N^{\frac{\varepsilon}{2\kappa}}$. This follows by the last bound provided in Theorem~\ref{thm:crd}.
\begin{acks}[Acknowledgments]
  The author thanks Professor Richard Nickl for discussions that have significantly shaped this research. In particular, he suggested that a renormalized Bernstein-von-Mises Theorem may hold for a scalar linear functional.
\end{acks}
\begin{funding}
  The author was supported by an ERC Advanced Grant (UKRI G116786).
  \end{funding}
\begin{supplement}
  \stitle{Proofs and Justifications for Remarks}
  \sdescription{We organize the proofs by the corresponding sections of the main paper. In addition, we include a last section entitled 'General Results' for external results.}
  \end{supplement}
  \section*{Proofs for Section~\ref{subsec:notation}}
  \begin{lemma}\label{lem:sq}
    $\Epost{\norm{\theta_N}{\h}^2} < \infty$.
  \end{lemma}
  \begin{proof}
    The definition \eqref{eq:oripos} states 
  \begin{align*}
    \frac{d\Pi_N(\cdot \mid \mathcal{D}_N)}{d\Pi_N} = \frac{e^{\ell_N}}{\int_{E_D} e^{\ell_N} \, d\Pi_N} .
  \end{align*}
  Thus,
    \begin{align*}
      \Epost{\norm{\theta_N}{\h}^2} =  \int_{E_D} \norm{\theta}{\h}^2 \, d\Pi_N(\theta \mid \mathcal{D}_N) = \frac{\int_{E_D} \norm{\theta}{\h}^2 e^{\ell_N(\theta)} \, d\Pi_N(\theta)}{\int_{E_D} e^{\ell_N(\theta)} \, d\Pi_N(\theta)}.
    \end{align*}
    We show that the numerator is finite. Note $\ell_N \le 0$, so $e^{\ell_N} \le 1$. Hence,
    \begin{align*}
      \int_{E_D} \norm{\theta}{\h}^2 e^{\ell_N(\theta)} \, d\Pi_N(\theta) \le \int_{E_D} \norm{\theta}{\h}^2  \, d\Pi_N(\theta).
    \end{align*}
    Since $E_D$ is finite-dimensional, $\norm{\cdot}{\h} \le C \norm{\cdot}{\mathcal{H}_N}$, where $C = C(N)$ is a positive constant. Coupling this with the previous display, we have
    \begin{align*}
      \int_{E_D} \norm{\theta}{\h}^2 e^{\ell_N(\theta)} \, d\Pi_N(\theta) \le C \int_{E_D} \norm{\theta}{\mathcal{H}_N}^2  \, d\Pi_N(\theta) < \infty,
    \end{align*}
    by Fernique's theorem \cite[Corollary~2.8.6]{Bogachev1998-aw}.
  \end{proof}
\section*{Proofs for Section~\ref{sec:hyp}}
\begin{proof}[Justification for Remark~\ref{rem:submodel}]
  We provide two, somewhat informal, justifications. The first is inspired by semi-parametric theory, and the second by classical parametric theory.
  \\\\
  \textit{Justification 1}: Consider the model $\theta_0 + E_D$. Here, the $L^2$-linearization of $\mathcal{G}$ at $\theta_0$ is $\lin \iota_{E_D}$. Thus, 
  \begin{align*}
    \lr{\lin \iota_{E_D}}^*\lr{\lin \iota_{E_D}} = \iota_{E_D}^* \lin^* \lin \iota_{E_D}.
  \end{align*}
  represents the information operator at $\theta_0$, as in \cite{Nickl2023}. Now, $\lin^* \lin = \mathcal{I}$, and $\iota_{E_D}^* = P_{E_D}$ by \eqref{eq:pswap}. Thus, $\mathcal{I}_D = P_{E_D}\mathcal{I}\iota_{E_D}$ represents the information operator at $\theta_0$, and $[\mathcal{I}_D]$ is the corresponding 'information matrix' with respect to the basis $\curly{e_1, \ldots, e_D}$.
  \\\\
  \textit{Justification 2}: For simplicity, let $(V, \ip{\cdot}{\cdot}{V}) \coloneqq (\mathbb{R}, *)$, and assume that $\partial \forward_{\truth + t e_i}(x) / \partial t \Big|_{t = 0}$ exists for all $x \in \mathcal{X}$ and $i \le k$. 
  
  Let $S_D(\theta_0)$ denote the score function at $\theta_0$ in the model $\theta_0 + E_D$ with $N = 1$:
\begin{align*}
  S_D^i(\theta_0) \coloneqq \frac{\partial \ell_1(\theta_0 + te_i)}{\partial t}\Big|_{t = 0} , \quad i \le D.
\end{align*}
Recalling $\ell_1(\theta) =  -\frac{1}{2} \lr{Y_1 - \forward_\theta(X_1)}^2$ for $\theta \in \Theta$, we have
\begin{align}\label{eq:recscore}
  S_D^i(\theta_0) = \varepsilon_1 \frac{\partial \forward_{\truth + te_i}(X_1)}{\partial t} \Big|_{t = 0}.
\end{align}
Let $I_D(\truth)$ denote the information matrix at $\theta_0$ in the model $\theta_0 + E_D$:
\begin{align}\label{eq:infdef1}
 I_D^{ij}(\theta_0) \coloneqq \Etheta \sq{S_D^i(\theta_0)S_D^j(\theta_0)}, \quad i,j\le D.
\end{align}
We show $I_D(\theta_0) = [\fishermatrix]$.

Substitute \eqref{eq:recscore} into \eqref{eq:infdef1}. Use $\Etheta \varepsilon_1^2 = 1$, and that $X_1$ is independent of $\varepsilon_1$. Then,
\begin{align*}
  I_D^{ij}(\theta_0) &=   \Etheta \left[ \frac{\partial \forward_{\truth + te_i}(X_1)}{\partial t}  \frac{\partial \forward_{\truth + te_j}(X_1)}{\partial t} \Big|_{t = 0}   \right]\Etheta \varepsilon_1^2, \\  &= \Etheta \left[ \frac{\partial \forward_{\truth + te_i}(X_1)}{\partial t}  \frac{\partial \forward_{\truth + te_j}(X_1)}{\partial t} \Big|_{t = 0}   \right].
\end{align*}
We now show that $[\fishermatrix]_{ij}$ equals the latter quantity.

Recall \eqref{eq:finiteFisher1} and $\mathcal{I}_D = P_{E_D} \fisher \iota_{E_D}$. Then,
\begin{align}\label{eq:wproj}
  [\fishermatrix]_{ij} = \ip{P_{E_D}\mathcal{I}\iota_{E_D}e_i}{e_j}{\h} = \ip{P_{E_D}\mathcal{I}e_i}{e_j}{\h}.
\end{align}
By \cite[Remark 26.15(i)]{Schilling2017-ex}, we have
\begin{align}\label{eq:pswap}
  \ip{x}{y}{\h} = \ip{P_{E_D}x}{y}{\h}, \quad \text{for all }x \in \h \text{ and }y \in E_D.
\end{align}
Take $x \coloneqq \mathcal{I}e_i$ and $y \coloneqq e_j$. Then, \eqref{eq:wproj} reveals
\begin{align}\label{eq:reveals}
  [\fishermatrix]_{ij} =  \ip{\mathcal{I}e_i}{e_j}{\h}.
\end{align}
Using $\mathcal{I} = \lin^* \lin$, and subsequently $X_1 \sim \lambda_\mathcal{X}$, we have
  \begin{align}\label{eq:finiteFisher2}
    [\fishermatrix]_{ij}  = \ip{\lin e_i}{\lin  e_j}{L^2} = \Etheta \sq{ \lin e_i(X_1) \lin  e_j(X_1) }.
  \end{align}
  Recall Hypothesis~\ref{cond:2}. Then, as $t \longrightarrow 0$, 
  \begin{align*}
    \lr{\forward_{\truth + t e_i} - \forward_{\truth}}/t &\longrightarrow \lin e_i , \quad &\text{in $L^2(\lambda_\mathcal{X})$},\\
    \lr{\forward_{\truth + t e_i}(x) - \forward_{\truth}(x)}/t &\longrightarrow \partial \forward_{\truth + t e_i}(x) / \partial t \Big|_{t = 0}, \quad &\text{for all $x \in \mathcal{X}$}.
  \end{align*}
  We identify the above limits using \cite[Corollary~13.8]{Schilling2017-ex}:
  \begin{align}\label{eq:derop}
    \frac{\partial \forward_{\truth + t e_i}(x)}{\partial t} \Big|_{t = 0}  = \lin e_i(x), \quad \text{for $\lambda_\mathcal{X}$ almost all $x$}.
  \end{align}
  Plugging \eqref{eq:derop} into \eqref{eq:finiteFisher2}, we have
\begin{align}\label{eq:finf}
  [\fishermatrix]_{ij}   = \Etheta \left[ \frac{\partial \forward_{\truth + te_i}(X_1)}{\partial t}  \frac{\partial \forward_{\truth + te_j}(X_1)}{\partial t} \Big|_{t = 0}   \right].
\end{align}
\end{proof}
\section*{Proofs for Section~\ref{sec:ren}}
\begin{lemma}\label{lem:boundinv}
  \hspace*{0.5cm}
  \begin{enumerate}
    \item $\mathcal{J}_D [\fishermatrix]^{-1} \mathcal{J}_D^\intercal \succ 0$ for $D$ large enough,
    \item $\norm{\renorm}{\text{op}} \lesssim 1$.
  \end{enumerate}
\end{lemma}
\begin{proof}
  \hspace*{0.5cm}\\\\
  \textit{Claim~1}: The first equality of \eqref{eq:finiteFisher2} reveals $[\fishermatrix] \succeq 0$. Indeed, letting $x \in \mathbb{R}^D$, 
  \begin{align*}
    x^\intercal [\fishermatrix] x = \sum_{i = 1}^D\sum_{j = 1}^D x_i\ip{\lin e_i}{\lin  e_j}{L^2}x_j = \norm{\sum_{i = 1}^D x_i \lin e_i}{L^2}^2 \ge 0.
  \end{align*}  
  Recalling Hypothesis~\ref{cond:inv}, we now have $[\fishermatrix]^{-1} \succ 0$. Thus,
  \begin{align*}
    \mathcal{J}_D [\fishermatrix]^{-1} \mathcal{J}_D^\intercal \succ 0 \Longleftrightarrow \text{Rank}\lr{\mathcal{J}_D^\intercal} = k \Longleftrightarrow \text{Rank}\lr{\mathcal{J}_D\mathcal{J}_D^\intercal} = k \Longleftrightarrow \mathcal{J}_D\mathcal{J}_D^\intercal \in \text{GL}_k(\mathbb{R}).
  \end{align*}
  We refer the reader to \cite[Observation~7.1.8(b)]{Horn_Johnson_2012}, \cite[Section~0.4.6(d)]{Horn_Johnson_2012}, and \cite[Section~0.5(b-c)]{Horn_Johnson_2012} for these equivalences. It now suffices to show $\mathcal{J}_D\mathcal{J}_D^\intercal \in \text{GL}_k(\mathbb{R})$ for $D$ large enough.

  Let $n, m \le k$. Then, 
  \begin{align}\label{eq:feqjf}
    \lr{\mathcal{J}_D \mathcal{J}_D^\intercal}_{nm} = \bs{e}_n\mathcal{J}_D \mathcal{J}_D^\intercal \bs{e}_m = \lr{\mathcal{J}_D^\intercal \bs{e}_n}^\intercal \mathcal{J}_D^\intercal \bs{e}_m.
  \end{align}
  By definition, and subsequently using \eqref{eq:pswap}, we have
  \begin{align}\label{eq:repvec}
    \mathcal{J}_D^\intercal \bs{e}_n = \lr{\ip{\psi_n}{e_i}{\h}}_{i \le D} = \lr{\ip{P_{E_D} \psi_n}{e_i}{\h}}_{i \le D} \eqqcolon [P_{E_D}\psi_n].
  \end{align}
  Note that $[P_{E_D}\psi_n]$ represents $P_{E_D} \psi_n$ in the $\curly{e_1, \ldots, e_D}$ basis. Now, \eqref{eq:feqjf} reveals
  \begin{align}\label{eq:bthus}
    \lr{\mathcal{J}_D \mathcal{J}_D^\intercal}_{nm} = [P_{E_D}\psi_n]^\intercal [P_{E_D}\psi_m].
  \end{align}

  Recall that $\curly{e_1, \ldots, e_D}$ is an orthonormal basis for $E_D$. Hence, for all $x, y \in E_D$:
  \begin{align}\label{eq:basisrep}
    \ip{x}{y}{\h} = [x]^\intercal [y], \quad \text{where } [x] \coloneqq \lr{\ip{x}{e_i}{\h}}_{i \le D} \text{ and } [y] \coloneqq \lr{\ip{y}{e_i}{\h}}_{i \le D}.
  \end{align}
  Thus, \eqref{eq:bthus} states
  \begin{align*}
    \lr{\mathcal{J}_D \mathcal{J}_D^\intercal}_{nm} =  \ip{P_{E_D}\psi_n}{P_{E_D}\psi_m}{\h}.
  \end{align*}

  By \cite[Theorem~26.21]{Schilling2017-ex}, we have $\norm{\psi_n - P_{E_D}\psi_n}{\h} \longrightarrow 0$. By bicontinuity of $\ip{\cdot}{\cdot}{\h}$, the above display then yields
  \begin{align}\label{eq:jacobianSquared}
    \mathcal{J}_D\mathcal{J}_D^\intercal  \longrightarrow P \coloneqq  \lr{\ip{\psi_n}{\psi_m}{\h}}_{i, j \le k}.
  \end{align}
  Since $\psi_1, \ldots, \psi_k$ are linearly independent, $P \succ 0$ by \cite[Theorem~7.2.10(b)]{Horn_Johnson_2012}. In particular, $P \in \text{GL}_k(\mathbb{R})$. Since $\text{GL}_k(\mathbb{R})$ is open \cite[Example~1.27]{Lee2012}, the above limit yields $\mathcal{J}_D\mathcal{J}_D^\intercal \in \text{GL}_k(\mathbb{R})$ for $D$ large enough.
  \\\\
  \textit{Claim~2}: Let $S^{k-1}\coloneqq \curly{x \in \rk : \norm{x}{\rk} = 1}$. We note down some matrix identities. We refer the reader to \cite[Theorem~4.2.2]{Horn_Johnson_2012} and \cite[Section~5.6.6]{Horn_Johnson_2012}. Let $A \in \rkbyk$ be positive definite. Then,
  \begin{align}
    \eigmax{A} &= \sup_{x \in S^{k - 1}} x^\intercal A x ,\label{eq:eigmax}\\ 
    \eigmin{A} &= \inf_{x \in S^{k - 1}} x^\intercal A x, \label{eq:eigmin}  \\
    \eigmax{A} &= \norm{A^{1/2}}{\text{op}}^2, \label{eq:eigsq}\\
    \eigmax{A} &= \frac{1}{\eigmin{A^{-1}}}. \label{eq:eiginv}
  \end{align}
  
  Coupling \eqref{eq:eigsq} with \eqref{eq:eiginv} yields 
  \begin{align*}
    \norm{\renorm}{\text{op}}^2 = \frac{1}{\eigmin{i_D^{-1}} }. 
  \end{align*}
  The claim thus follows, if we can show $\eigmin{i_D^{-1}} \gtrsim 1$. Using \eqref{eq:eigmin} and recalling \eqref{eq:renormdef}, we have
  \begin{align}\label{eq:mineig}
    \eigmin{i_D^{-1}} = \inf_{x \in S^{k - 1}}x^\intercal \mathcal{J}_D[\fishermatrix]^{-1}\mathcal{J}_D^\intercal x.
  \end{align}
  One can quickly confirm that
  \begin{align*}
    x^\intercal \mathcal{J}_D[\fishermatrix]^{-1}\mathcal{J}_D^\intercal x = I \cdot II ,
  \end{align*}
  where
  \begin{align*}
    I &\coloneqq y^\intercal [\fishermatrix]^{-1} y, \quad y \coloneqq \mathcal{J}_D^\intercal x / \norm{\mathcal{J}_D^\intercal x}{\mathbb{R}^D},\\
    II &\coloneqq \norm{\mathcal{J}_D^\intercal x}{\mathbb{R}^D}^2.
  \end{align*}
  Recalling \eqref{eq:mineig}, it suffices to establish a positive lower bound for $I$ and $II$, that is uniform in $x \in S^{k - 1}$ and $D \ge D_0$, for some $D_0 \in 
\mathbb{N}$.
  \\\\
  \textit{Bounding I:} Note $y\in S^{k-1}$. Thus, \eqref{eq:eigmin} and \eqref{eq:eiginv} yield
  \begin{align}\label{eq:firstLB}
    y^\intercal [\fishermatrix]^{-1} y \ge \eigmin{[\fishermatrix]^{-1}} = 1 / \eigmax{[\fishermatrix]}.
  \end{align}
  Hence, it suffices to show $ \eigmax{[\fishermatrix]} \lesssim 1 $. By \eqref{eq:reveals}, we have for $x \in S^{k -1}$,
  \begin{align}\label{eq:finiteFisherEig}
    x^\intercal [\fishermatrix] x = \sum_{i = 1}^k \sum_{j = 1}^k x_i\ip{\mathcal{I}e_i}{e_j}{\h}x_j = \ip{\mathcal{I}z}{z}{\h}, 
  \end{align} 
  where $z \coloneqq \sum_{i = 1}^k x_i e_i$ satisfies $\norm{z}{\h} = 1$ by Pythagoras' Theorem. Thus using \eqref{eq:eigmax}, the above display, and subsequently Cauchy-Schwarz,
  \begin{align}\label{eq:idsqbound}
    \eigmax{[\fishermatrix]} =  \sup_{x \in S^{k - 1}} x^\intercal [\fishermatrix] x\le \sup_{\norm{z}{\h} = 1}\ip{\mathcal{I}z}{z}{\h} \le \norm{\mathcal{I}}{\text{op}}.
  \end{align}
  Clearly, $\Vert \mathcal{I} \Vert_{\text{op}} > 0$. Indeed, if $\mathcal{I} = 0$, we would have $\mathcal{I}_D = 0$, which contradicts Hypothesis~\ref{cond:inv}. Applying \eqref{eq:idsqbound} to \eqref{eq:firstLB}, we have
  \begin{align*}
    y^\intercal [\fishermatrix]^{-1} y \ge \frac{1}{\norm{\mathcal{I}}{\text{op}}}.
  \end{align*}
  \textit{Bounding II}: 
  We establish a positive lower bound for 
  \begin{align*}
    \inf\curly{\norm{\mathcal{J}_D^\intercal x}{\mathbb{R}^D}^2 : x \in S^{k - 1}} 
  \end{align*}
  that is uniform in $D \ge D_0$ for some $D_0 \in \mathbb{N}$. Using \eqref{eq:eigmin} in the last equality,
  \begin{align}\label{eq:lmineq}
    \inf\curly{\norm{\mathcal{J}_D^\intercal x}{\mathbb{R}^D}^2 : x \in S^{k - 1}} = \inf\curly{x^\intercal \mathcal{J}_D\mathcal{J}_D^\intercal x : x \in S^{k - 1}} = \eigmin{\mathcal{J}_D \mathcal{J}_D^\intercal}.
  \end{align}
  By \eqref{eq:jacobianSquared} and continuity of $\lambda_{\text{min}}$ on the space of symmetric $k  \times k$ matrices \cite[Corollary~6.3.8]{Horn_Johnson_2012},
  \begin{align*}
    \eigmin{\mathcal{J}_D \mathcal{J}_D^\intercal}  \longrightarrow \eigmin{P}.
  \end{align*}
  Since $P \succ 0$, we have $\eigmin{P} > 0$. Thus, the above limit yields $D_0 \in \mathbb{N}$ such that
  \begin{align*}
    \eigmin{\mathcal{J}_D \mathcal{J}_D^\intercal} \ge  \frac{1}{2}\eigmin{P}, \quad \text{for }D \ge D_0.
  \end{align*}
  Recalling \eqref{eq:lmineq}, we have
  \begin{align*}
    \inf \curly{\norm{\mathcal{J}_D^\intercal x}{\mathbb{R}^D}^2 : x \in S^{k - 1}}  \ge  \frac{1}{2}\eigmin{P}, \quad \text{for }D \ge D_0.
  \end{align*}
\end{proof}
\begin{proof}[Proof of Lemma~\ref{lem:asymptnorm}]
  We consider $N$ large enough that $\mathcal{J}_D [\fishermatrix]^{-1} \mathcal{J}_D^\intercal \succ 0$. This is possible by Lemma~\ref{lem:boundinv}. We can then write 
  \begin{align}\label{eq:idinv}
    i_D^{-1} = \mathcal{J}_D [\fishermatrix]^{-1} \mathcal{J}_D^\intercal.
  \end{align}

  Plugging \eqref{eq:centering} into \eqref{eq:inlem}, we need to show 
  \begin{align}\label{eq:fell}
    \sum_{i = 1}^N W_{iN} \dlim \mathcal{N}(0, I_k),
  \end{align}
  where
  \begin{align}\label{eq:fellerTerm}
    W_{iN} \coloneqq \frac{1}{\sqrt{N}}\renorm \ip{\varepsilon_i}{\lin \perturb(X_i)}{V} , \quad i \le N.
  \end{align}
  
  Before starting, we briefly study 
  \begin{align*}
    A_D \coloneqq \ip{\varepsilon_1}{ \lin \perturb(X_1)}{V}.
  \end{align*}
  Since $\varepsilon_1$ and $X_1$ are independent, and $\varepsilon_1$ is a centered Gaussian,
  \begin{align}\label{eq:llaww}
  \mathcal{L}\lr{A_D \mid X_1 = x} = \mathcal{L}\lr{\ip{\varepsilon_1}{ \lin \perturb(x)}{V}} =  \mathcal{N}\lr{0, V_D(x)}, \quad x \in \mathcal{X},
  \end{align}
  where $V_D(x) \in \rkbyk$ is defined by
  \begin{align*}
  V_D^{ij}(x) \coloneqq \Etheta \sq{\ip{\varepsilon_1}{ \lin \pertur^i(x) }{V}\ip{\varepsilon_1}{ \lin \pertur^j(x) }{V} }, \quad i,j \le k.
\end{align*}
As the covariance operator of $\varepsilon_1$ is the identity with respect to $\lr{V, \ip{\cdot}{\cdot}{V}}$, we have
\begin{align}\label{eq:vdvariance}
  V_D^{ij}(x) = \ip{\lin \pertur^i(x)}{\lin \pertur^j(x)}{V}.
\end{align}

We now establish \eqref{eq:fell} by the Lindeberg-Feller central limit theorem \cite[Corollary~18.2]{doi:10.1137/1.9780898719895}. Clearly, $\curly{W_{iN}}_{i \le N}$ are i.i.d. Furthermore, $\Etheta W_{iN} = 0$. Indeed, let $W_N \coloneqq W_{1N}$. Recalling \eqref{eq:fellerTerm}, we have
\begin{align}\label{eq:secondad}
  W_N =  \frac{1}{\sqrt{N}}\renorm A_D.
\end{align}
Thus, $\Etheta W_{N} = \frac{1}{\sqrt{N}}\renorm \Etheta[\Etheta[A_D \mid X_1]] = 0$ by \eqref{eq:llaww}. 

It now suffices to show:
  \begin{enumerate}
    \item $\Vtheta  \left(\sum_{i = 1}^N W_{iN} \right) = I_k$,
    \item $\sum_{i = 1}^N \Etheta \left[\Vert W_{iN} \Vert_{\rk}^2 ; \Vert W_{iN} \Vert_{\rk} \ge \eta \right] \longrightarrow 0$ for all $\eta > 0$.
  \end{enumerate}
\textit{Condition 1}: Since $\curly{W_{iN}}_{i \le N}$ are i.i.d, 
\begin{align*}
  \Vtheta  \left(\sum_{i = 1}^N W_{iN} \right) = I_k \Longleftrightarrow N \Vtheta  W_N = I_k \Longleftrightarrow \Vtheta  W_N = \frac{1}{N}I_k.
\end{align*}
We show the latter statement. By \eqref{eq:secondad}, we have
\begin{align}\label{eq:fellerVariance}
  \Vtheta W_N = \frac{1}{N}\renorm \lr{\Vtheta  A_D}\renorm.
\end{align}
Using the law of total variance and \eqref{eq:llaww}, we have
\begin{align}\label{eq:varAD}
  \Vtheta  A_D &= \Etheta \sq{\Vtheta \sq{A_D \mid X_1}} + \Vtheta \sq{\Etheta \sq{A_D \mid X_1}} = \Etheta V_D(X_1).
\end{align}
We study the components of $\Etheta V_D(X_1)$. Plugging in \eqref{eq:vdvariance}, and noting $X_1 \sim \lambda_\mathcal{X}$:
\begin{align*}
  \Etheta V_D^{ij}(X_1) = \Etheta \sq{\ip{\lin \pertur^i(X_1)}{\lin \pertur^j(X_1)}{V}} = \ip{\lin \pertur^i}{\lin \pertur^j}{L^2}.
\end{align*}
Hence, recalling \eqref{eq:varAD} and $\mathcal{I} = \lin ^*\lin $,
\begin{align*}
  \lr{\Vtheta  A_D}_{ij} = \ip{\lin \pertur^i}{\lin \pertur^j}{L^2}  = \ip{\mathcal{I}\pertur^i}{\pertur^j}{\h}.
\end{align*}
We now substitute 
\begin{align}\label{eq:substperturb}
  \pertur^j =  \fishermatrix^{-1}\proj\psi_j,
\end{align}
and subsequently apply \eqref{eq:pswap}:
\begin{align}\label{eq:pinto}
  \lr{\Vtheta  A_D}_{ij} = \ip{\mathcal{I}\pertur^i}{\fishermatrix^{-1}\proj\psi_j}{\h} = \ip{P_{E_D}\mathcal{I}\pertur^i}{\fishermatrix^{-1}\proj\psi_j}{\h}.
\end{align}
Recalling \eqref{eq:substperturb}, we can rewrite the first argument,
\begin{align*}
  P_{E_D}\mathcal{I}\pertur^i = \fishermatrix \pertur^i = P_{E_D}\psi_i. 
\end{align*}
We plug the previous equation into \eqref{eq:pinto}:
\begin{align}\label{eq:finalAD}
  \left(\Vtheta  A_D \right)_{ij} = \ip{\proj\psi_i}{\fishermatrix^{-1}\proj\psi_j}{\h}. 
\end{align}
Let $[\proj\psi_i]$ represent $\proj\psi_n$ in the $\curly{e_1, \ldots, e_D}$ basis. Then, $\fishermatrix^{-1}\proj\psi_j$ has representation $[\fishermatrix]^{-1}[\proj\psi_j]$ with respect to the same basis. Applying \eqref{eq:basisrep}, we have
\begin{align}\label{eq:recall}
  \ip{\proj\psi_i}{\fishermatrix^{-1}\proj\psi_j}{\h} = [\proj\psi_i]^\intercal [\fishermatrix]^{-1} [\proj\psi_j].
\end{align}
Recalling \eqref{eq:repvec}, we have $[\proj\psi_i] = \mathcal{J}_D^\intercal \bs{e}_i$. Thus, 
\begin{align*}
  \ip{\proj\psi_i}{\fishermatrix^{-1}\proj\psi_j}{\h} &= \lr{\mathcal{J}_D^\intercal \bs{e}_i}^\intercal [\fishermatrix]^{-1} \mathcal{J}_D^\intercal \bs{e}_j \\ &=   \bs{e}_i^\intercal \mathcal{J}_D [\fishermatrix]^{-1} \mathcal{J}_D^\intercal \bs{e}_j, \\
  &= \lr{\mathcal{J}_D [\fishermatrix]^{-1} \mathcal{J}_D^\intercal}_{ij}.
\end{align*}
Plugging this into \eqref{eq:finalAD} and subsequently recalling \eqref{eq:idinv}, we have
\begin{align}\label{eq:coup}
  \Vtheta  A_D = \mathcal{J}_D [\fishermatrix]^{-1} \mathcal{J}_D^\intercal = i_D^{-1}.
\end{align}
We plug the above display into \eqref{eq:fellerVariance}, and obtain 
\begin{align*}
  \Vtheta W_N = \frac{1}{N}\renorm i_D^{-1} \renorm  =  \frac{1}{N}I_k.
\end{align*}
\textit{Condition 2}: Since $\curly{W_{iN}}_{i \le N}$ are i.i.d, we can equivalently show 
\begin{align*}
  N\Etheta \sq{ \Vert W_N \Vert_{\rk}^2 \mathbbm{1}\curly{\norm{W_N}{\rk} \ge \eta} }\longrightarrow 0, \quad \text{for all $\eta > 0$}.
\end{align*} 
By Cauchy-Schwarz, 
\begin{align}\label{eq:lindeberg}
  N\Etheta \sq{ \Vert W_N \Vert_{\rk}^2 \mathbbm{1}\curly{\norm{W_N}{\rk} \ge \eta} } \le N\sqrt{I \cdot II},
\end{align}
where 
\begin{align}
  I &\coloneqq\Etheta \Vert W_N \Vert_{\rk}^4,\label{eq:11} \\ 
  II &\coloneqq \Ptheta \lr{\Vert W_N \Vert_{\rk} \ge \eta}. \label{eq:22}
\end{align}
\textit{Bounding I}: Apply Cauchy-Schwarz to \eqref{eq:fellerTerm}. Then,
  \begin{align}\label{eq:termbound1}
    \Vert W_N \Vert_{\rk} \le \frac{C_1}{\sqrt{N}}\norm{\lin \perturb}{L^\infty} \norm{\varepsilon_1}{V},
  \end{align} 
  where $C_1 \coloneqq  \sup_{D \in \mathbb{N}} \norm{\renorm}{\text{op}} < \infty$ by Lemma~\ref{lem:boundinv}. We plug \eqref{eq:termbound1} into \eqref{eq:11}:
  \begin{align}\label{eq:Ibound}
    I \le  C_2 \frac{\norm{\lin \perturb}{L^\infty}^4}{N^2}, \quad \text{where } C_2 \coloneqq C_1^4\Etheta \norm{\varepsilon_1}{V}^4.
  \end{align}
  \textit{Bounding II}: Recall \eqref{eq:22}. The tower property then yields
  \begin{align}\label{eq:tower}
    II = \Etheta [\Ptheta (\Vert W_N \Vert_{\rk} \ge \eta \mid X_1) ].
  \end{align}
  Note that $\mathcal{L}\lr{W_N \mid X_1 = x}$ is a centered Gaussian for $x \in \mathcal{X}$. By the Gaussian tail bound,
  \begin{align}\label{eq:pluginto}
    \Ptheta \lr{\Vert W_N \Vert_{\rk} \ge \eta \mid X_1} \le 2 \exp \lr{- \frac{C_3}{\Etheta \sq{\rknorm{W_N}^2 \mid X_1}} },
  \end{align}
  where $C_3 > 0$ is a constant. To be precise, we apply \cite[Theorem~2.1.20]{Giné_Nickl_2021}, noting that $\norm{\cdot}{\rk}$ and $x \in \rk \mapsto \max_{i \le k} \abs{x_i}$ are equivalent norms.
  
  We can now bound the exponent in \eqref{eq:pluginto} using \eqref{eq:termbound1}:
  \begin{align*}
    \Etheta \sq{\rknorm{W_N}^2 \mid X_1} \le C_4  \frac{ \norm{\lin \perturb}{L^\infty}^2}{N} , \quad \text{where }C_4 \coloneqq C_1^2 \Etheta \norm{\varepsilon_1}{V}^2.
  \end{align*}
  As $x \mapsto 2\exp(-C_3 / x)$ is increasing, we can plug the above bound into \eqref{eq:pluginto}:
  \begin{align}\label{eq:anbound}
    \Ptheta \lr{\Vert W_N \Vert_{\rk} \ge \eta \mid X_1} \le 2\exp\lr{- C_5 a_N}, \quad \text{where } C_5 \coloneqq \frac{C_3}{C_4} \text{ and }a_N \coloneqq \frac{N}{\norm{\lin \perturb}{L^\infty}^2}.
  \end{align}
  Take the expectation on both sides, and recall \eqref{eq:tower}. Then,
  \begin{align}\label{eq:IIbound}
    II \lesssim \exp\left(- C_5 a_N \right).
  \end{align}
  Plug the bounds \eqref{eq:Ibound} and \eqref{eq:IIbound} into \eqref{eq:lindeberg}. Then,
  \begin{align}\label{eq:lindlast}
    III \coloneqq N\Etheta [\Vert W_N \Vert_{\rk}^2 ; \Vert W_N  \Vert_{\rk} \ge \eta]  \lesssim  \norm{\lin \perturb}{L^\infty}^2 \exp\left(- \frac{1}{2}C_5 a_N \right).
  \end{align}
  By assumption, $\Vert \lin \perturb\Vert_{L^\infty} = o\lr{\sqrt{N}}$. Recalling the definition of $a_N$ from \eqref{eq:anbound}, this is equivalent to $a_N \longrightarrow \infty$. 
  
  Recall that we need to show $III \longrightarrow 0$. Suppose we can show
  \begin{align}\label{eq:suffcond}
    \norm{\lin \perturb}{L^\infty}^2  \lesssim a_N^q, \quad \text{for some $q >0$}.
  \end{align}
 Then, by \eqref{eq:lindlast}, we have
 \begin{align*}
  III \lesssim a_N^q \exp\left(- \frac{1}{2}C_5 a_N \right) \longrightarrow 0,
 \end{align*}
where we pass the limit by polynomial to exponential comparison.
 
 Recall the definition of $a_N$ from \eqref{eq:anbound}. Then,
 \begin{align*}
  \norm{\lin \perturb}{L^\infty}^2 \lesssim a_N^q \Longleftrightarrow \norm{\lin \perturb}{L^\infty}^2 \lesssim \norm{\lin \perturb}{L^\infty}^{-2q} N^{q} \Longleftrightarrow \norm{\lin \perturb}{L^\infty}^{2 + 2q} \lesssim N^{q}.
 \end{align*}
 Apply now $x \mapsto x^{\frac{1}{2 + 2q}}$. Then, \eqref{eq:suffcond} is equivalent to
  \begin{align*}
    \norm{\lin \perturb}{L^\infty} \lesssim N^{\frac{q}{2+2q}}, \quad \text{ for some }q > 0.
  \end{align*}
  Since $\frac{q}{2+2q}$ can take on all values in $\lr{0, \frac{1}{2}}$, the above display holds by assumption.
\end{proof}
\begin{lemma}\label{lem:reginc}
  Fix $M > 0$. Assume:
  \begin{enumerate}
    \item $\sup_{\theta \in \regset} \norm{\forward\lr{\theta +  \lambda^\intercal \bs{h}_N }  - \forward\lr{\theta} }{L^2} = o\lr{\delta_N^\mathcal{G}}$ for all $\lambda \in \rk$.
    \item $\norm{\perturb}{\mathcal{R}}= o\lr{\sqrt{N}}$.
  \end{enumerate}
  Then, there exists $M' > 0$ such that for all $\lambda \in \rk$,
  \begin{align}\label{eq:sthat}
    \regset +  \lambda^\intercal \bs{h}_N  \subset \Theta_{N, M'}, \quad \text{ for }N \text{ large enough}.
  \end{align}
\end{lemma}
\begin{proof}
  We show that $M' \coloneqq M + 1$ works. Fix $\lambda \in \rk$. We need to show
  \begin{align}\label{eq:firstcond}
    \norm{\theta +  \lambda^\intercal \bs{h}_N }{\mathcal{R}} &\le M + 1, \\\label{eq:secondcond}
    \norm{\mathcal{G}\lr{\theta +  \lambda^\intercal \bs{h}_N } - \mathcal{G}\lr{\truth}}{L^2} &\le (M + 1) \delta_N^\mathcal{G},
  \end{align}
  for $N$ large enough and $\theta \in \regset$.
  \\\\
  \eqref{eq:firstcond}: By the triangle inequality,
  \begin{align}\label{eq:reccaling}
    \norm{\theta +  \lambda^\intercal \bs{h}_N }{\mathcal{R}} \le M + \norm{ \lambda^\intercal \bs{h}_N }{\mathcal{R}}, \quad \theta \in \Theta_{N, M}.
  \end{align}
  Recall \eqref{eq:perdef} and Section~\ref{subsec:notation} on vectorized notation. Then, 
  \begin{align}\label{eq:asdin}
    \norm{ \lambda^\intercal \bs{h}_N }{\mathcal{R}} \lesssim \norm{\bs{h_N}}{\mathcal{R}} \lesssim \norm{i_D^{\frac{1}{2}}}{\text{op}} \frac{\norm{\perturb}{\mathcal{R}}}{\sqrt{N}} \lesssim \frac{\norm{\perturb}{\mathcal{R}}}{\sqrt{N}} \longrightarrow 0.
  \end{align}
  Here, we use Lemma~\ref{lem:boundinv} in the last inequality, and pass the limit by Assumption~2.
  In particular, $\norm{ \lambda^\intercal \bs{h}_N }{\mathcal{R}} \le 1$ for $N$ large enough. Recalling \eqref{eq:reccaling}, we are done.
  \\\\
  \eqref{eq:secondcond}: Using Assumption~1, we have
  \begin{align*}
  \sup_{\theta \in \Theta_{N, M}} \norm{\forward\lr{\theta +  \lambda^\intercal \bs{h}_N }  - \forward\lr{\theta} }{L^2} \le \delta_N^\forward, \quad \text{for $N$ large enough}.
  \end{align*}
  Thus, for $N$ large enough and $\theta \in \regset$, the triangle inequality yields
  \begin{align*}
    \norm{\mathcal{G}\lr{\theta +  \lambda^\intercal \bs{h}_N } - \mathcal{G}\lr{\truth}}{L^2} \le \norm{\mathcal{G}_{\theta} - \mathcal{G}_{\theta_0}}{L^2} + \norm{\mathcal{G}\lr{\theta +  \lambda^\intercal \bs{h}_N } - \mathcal{G}\lr{\theta}}{L^2} 
    \le M\delta_N^\forward +\delta_N^\forward.
  \end{align*}
\end{proof}
\begin{proof}[Proof of Lemma~\ref{lem:empiricalprocess}]
  We proceed similarly to \cite[Lemma~4.1.7]{Nickl2023} and \cite[Lemma~4.1.8]{Nickl2023}. Fix $M > 0$ and $\lambda \in \rk$. Construct $M' \ge \norm{\theta_0}{\mathcal{R}}$ from Lemma~\ref{lem:reginc}, so that $\theta_0 \in \Theta_{N, M'}$ and \eqref{eq:sthat} hold.
  \\\\
  \eqref{eq:emplimits1}: By Markov's inequality, it suffices to show
  \begin{align}\label{eq:desiredlimit}
    \Etheta \sup_{\theta \in \Theta_{N, M'}}\left\lvert  \sum_{i=1}^N\ip{\varepsilon_i}{R_\theta(X_i)}{V} \right\rvert \longrightarrow 0.
  \end{align}
  Since $\Etheta  \langle \varepsilon_1, R_\theta(X_1) \rangle_V =\Etheta \sq{ \Etheta  \sq{\langle \varepsilon_1, R_\theta(X_1) \mid X_1 }} = 0$, we can employ the bound from Theorem~\ref{thm:empb}. The bound is specified up to a universal constant. Hence, we can allow the function class to depend on $N$. We denote the quantities $I, u, d, s, \mathcal{F}$ from Theorem~\ref{thm:empb} by $I_N, u_N, d_N, s_N, \mathcal{F}_N$. 
  
  Let $\mathcal{F}_N \coloneqq \{f_\theta : \theta \in \Theta_{N, M'}\}$, where
  \begin{align*}
    f_\theta(x, e) \coloneqq \ip{e}{R_{\theta}(x)}{V}, \quad (x, e) \in \mathcal{X} \times V.
  \end{align*}
  Take $P 
  \coloneqq \lambda_\mathcal{X} \otimes \mathcal{N}_V$. Note that $\lr{\Theta_{N, M'}, d_\Theta}$ is separable as $\Theta_{N, M'} \subset E_D$, where we recall that $E_D$ is finite-dimensional. 
  
  We now argue that the map $\theta \in \Theta_{N, M'} \mapsto f_\theta(x, e)$ is $d_\Theta$-Lipschitz continuous for all $(x, e) \in \mathcal{X} \times V$. By Cauchy-Schwarz,
  \begin{align}\label{eq:diff1}
    \abs{f_\theta(x, e) - f_{\theta'}(x, e)} &\le \norm{e}{V}\norm{R_\theta(x) - R_{\theta'}(x)}{V}.
  \end{align}
  Using Hypothesis~\ref{cond:4}, we have
  \begin{align*}
    \norm{R_\theta(x) - R_{\theta'}(x)}{V} = \frac{\norm{R_\theta(x) - R_{\theta'}(x)}{V}}{d_\Theta(\theta, \theta')}d_\Theta(\theta, \theta') \le C_1 r_N d_\Theta(\theta, \theta'),
  \end{align*}
  where $C_1$ is independent of $N \in \mathbb{N}, x \in \mathcal{X}$ and $\theta, \theta' \in \Theta_{N, M'}$, satisfying $\theta \neq \theta'$.
  Plugging the above display into \eqref{eq:diff1} yields
  \begin{align}
    \abs{f_\theta(x, e) - f_{\theta'}(x, e)} \le C_1 \norm{e}{V} r_N d_\Theta(\theta, \theta'), \quad \label{eq:argas}
  \end{align}
  for $\theta, \theta' \in \Theta_{N, M'}$ and $(x, e) \in \mathcal{X} \times V$. Thus, we have $d_\Theta$-Lipschitz continuity.
  
  We have to check $0 \in \mathcal{F}_N$. This is clear, since $ f_{\theta_0} = 0$ follows by $R_{\theta_0} = 0$.
  
  We now construct $F_N : \mathcal{X} \times V \longrightarrow (0, \infty)$ satisfying
  \begin{align*}
    F_N(x, e) \ge \sup_{\theta \in \Theta_{N, M'}} \abs{ f_\theta(x, e)}, \quad (x, e) \in \mathcal{X} \times V.
  \end{align*}
  Set $\theta' = \theta_0$ in \eqref{eq:argas}, and recall that $f_{\theta_0} = 0$. Then for $\theta \in \Theta_{N, M'}$, we have 
  \begin{align*}
    \abs{f_\theta(x, e)} \le C_1 \norm{e}{V} r_N d_\Theta(\theta, \theta_0) \le C_2 \norm{e}{V} r_N\delta_N^\Theta,
  \end{align*}
  where $C_2 > 0$ is independent of $(x, e) \in \mathcal{X} \times V$ and $\theta \in \Theta_{N, M'}$. Here, we use Hypothesis~\ref{cond:4} in the last inequality. We can then set
  \begin{align}\label{eq:jsetting}
    F_N(x, e) \coloneqq C_2 \norm{e}{V} r_N\delta_N^\Theta,  \quad (x, e) \in \mathcal{X} \times V.
  \end{align}
  Clearly,
  \begin{align}\label{eq:l2bound}
    \Vert F_N \Vert_{L^2(P)} = C_3 r_N\delta^\Theta_N, \quad \text{where }C_3 \coloneqq C_2 \lr{\Etheta \norm{\varepsilon_1}{V}^2}^{\frac{1}{2}}.
  \end{align}
  We now construct $s_N$ such that
  \begin{align}
    s_N \ge \sup_{\theta \in \Theta_{N, M'}} \Vert f_\theta\Vert_{L^2(P)}. 
  \end{align}
  By Cauchy-Schwarz and Hypothesis~\ref{cond:4}, we have
  \begin{align*}
    \sup_{\theta \in \Theta_{N, M'}} \norm{f_\theta}{L^2(P)} \le \lr{ \Etheta \norm{\varepsilon_1}{V}^2}^{\frac{1}{2}} \sup_{\theta \in \Theta_{N, M'}} \Vert R_\theta \Vert_{L^2} \lesssim  \sigma_N.
  \end{align*}
  Thus, we can take $s_N \coloneqq C_4 \sigma_N$ for some constant $C_4 > 0$. Let
  \begin{align}\label{eq:ddef}
    d_N \coloneqq \frac{s_N}{\norm{F_N}{L^2(P)}}.
  \end{align}
  Recalling \eqref{eq:l2bound}, we have
  \begin{align}\label{eq:dnbound}
    d_N= \frac{C_4}{C_3}\frac{\sigma_N}{r_N \delta_N^\Theta}.
  \end{align}
  
  We proceed to bound the quantity 
  \begin{align}\label{eq:udef}
    u_N \coloneqq  \Vert \text{max}_{i\le N} F_N(X_i) \Vert_{L^2(P^N)}. 
  \end{align}
  Recalling \eqref{eq:jsetting}, we have
  \begin{align}\label{eq:uref}
    u_N  = C_2 \lr{\Etheta\max_{i\le N} \norm{\varepsilon_i}{V}^2}^{\frac{1}{2}}  r_N \delta_N^\Theta.
  \end{align}
  Let $n  \coloneqq \text{dim} \, V$, and let $v_1, \ldots, v_n$ be an orthonormal basis of $V$. Define $\varepsilon_{ij} \coloneqq \ip{\varepsilon_i}{v_j}{V}$ for $i \le N$ and $j \le n$. Then, $\varepsilon_{ij} \stackrel{\text{iid}}{\sim} \mathcal{N}(0, 1)$ and \cite[Lemma~2.3.3]{Giné_Nickl_2021} yields
  \begin{align*}
    \Etheta\max_{i\le N} \norm{\varepsilon_i}{V}^2 \le \sum_{j = 1}^{n}\Etheta \max_{i\le N} \norm{\varepsilon_{ij}}{V}^2 \lesssim \log \, N.
  \end{align*}
  We plug this bound into \eqref{eq:uref}:
  \begin{align}\label{eq:unbound}
    u_N \lesssim \sqrt{\log \, N} r_N \delta_N^\Theta.
  \end{align}
  Let $Q$ represent a distribution on $\mathcal{X} \times V$ with finite support, and let
  \begin{align*}
    I_N(t) \coloneqq  \int_0^t \sup_Q \sqrt{\log 2 N(\mathcal{F}_N, \Vert \cdot \Vert_{L^2(Q)}, \varepsilon \Vert F_N \Vert_{L^2(Q)})} \, d\varepsilon, \quad t > 0.
  \end{align*}
   Theorem~\ref{thm:empb} now yields
   \begin{align}\label{eq:empbound}
    \Etheta \sup_{\theta \in \Theta_{N, M'}}\left\lvert  \sum_{i=1}^N \ip{\varepsilon_i}{R_\theta(X_i)}{V} \right\rvert  \lesssim \max\curly{ I, II },
  \end{align}
  where
  \begin{align}
    I &\coloneqq \sqrt{N}\Vert F_N \Vert_{L^2(P)}I_N(d_N),\label{eq:firstmax}\\
    II &\coloneqq u_N\lr{\frac{I(d_N)}{d_N}}^2 .\label{eq:secmax}
  \end{align}
  
  We proceed to bound $I_N$ using $J_N$ from Hypothesis~\ref{cond:4}. Recalling \eqref{eq:jsetting}, we have
  \begin{align}\label{eq:qequality}
    \Vert F_N \Vert_{L^2(Q)} = C_2 s_Q r_N\delta_N^\Theta, \quad \text{where }s_Q^2 \coloneqq \int_{\mathcal{X}\times V} \norm{e}{V}^2 \, dQ(x, e).
  \end{align}
  Applying $\norm{\cdot}{L^2(Q)}$ to \eqref{eq:argas}, we have
  \begin{align*}
    \Vert f_{\theta} - f_{\theta'} \Vert_{L^2(Q)} \le C_1 r_N   s_Q d_{\h}(\theta, \theta').
  \end{align*}
  We solve for $s_Qr_N$ in \eqref{eq:qequality}, and plug the resulting expression into the above display. Then,
  \begin{align*}
      \Vert f_{\theta} - f_{\theta'} \Vert_{L^2(Q)} \le C_5\frac{\Vert F_N \Vert_{L^2(Q)} }{\delta_N^\Theta} d_{\h}(\theta, \theta'), \quad \text{where }C_5 \coloneqq \frac{C_1}{C_2}.
  \end{align*}
  Let $\varepsilon > 0$. By the above display, an $\varepsilon \delta_N^\Theta / C_5$ covering of $\Theta_{N, M'}$ in $d_\Theta$ distance, induces an $\varepsilon \Vert F_N \Vert_{L^2(Q)}$ covering of $\mathcal{F}_N$ in $L^2(Q)$ distance. Thus,
  \begin{align*}
    N\lr{\mathcal{F}_N, \Vert \cdot \Vert_{L^2(Q)}, \varepsilon \Vert F_N \Vert_{L^2(Q)}} \le N \lr{\Theta_{N, M'}, d_\Theta, \frac{\varepsilon \delta_N^\Theta}{C_5} }, \quad \varepsilon > 0.
  \end{align*}
  Apply the transformation $\varepsilon \mapsto \sqrt{\log 2\varepsilon}$, take the supremum over $Q$, and integrate over $[0, t]$. Then,
  \begin{align}\label{eq:intineq}
    I_N(t) \le \int_{0}^t  \sqrt{\log 2 N \lr{\Theta_{N, M'}, d_\Theta, \frac{\varepsilon \delta_N^\Theta}{C_5}}} \, d\varepsilon, \quad t > 0.
  \end{align}
  Substituting the variable $\varepsilon' = \varepsilon \delta_N^\Theta/ C_5$, we conclude that
  \begin{align*}
    I_N(t) \le \frac{C_5}{\delta_N^\Theta} J_{N}\lr{t\delta_N^\Theta / C_5 }.
  \end{align*}
  Plugging \eqref{eq:dnbound} into \eqref{eq:intineq}, we now have
  \begin{align}\label{eq:multconst}
    I_N\lr{d_N} \le \frac{C_5 }{\delta_N^\Theta}J_{N}\lr{d_N \delta_N^\Theta / C_5 } = \frac{C_5}{\delta_N^\Theta}J_{N}\lr{\frac{C_4}{C_3 C_5}\frac{\sigma_N}{r_N}}.
  \end{align}
  By \cite[Lemma~3.5.3a]{Giné_Nickl_2021},
  \begin{align*}
    J_N(c t) \le \max\curly{c, 1} J_N(t), \quad c, t > 0.
  \end{align*}
  Thus, \eqref{eq:multconst} yields
  \begin{align*}
    I_N\lr{d_N} \lesssim   \frac{J_{N}\lr{\sigma_N / r_N}}{\delta_N^\Theta} .
  \end{align*}
  We now bound $I$ and $II$ from \eqref{eq:firstmax} and \eqref{eq:secmax}. We plug in the above bound, \eqref{eq:unbound}, \eqref{eq:dnbound} and \eqref{eq:l2bound}. Then,
  \begin{align*}
    I &\lesssim \sqrt{N}r_N \delta_N^\Theta\frac{J_{N}\lr{\sigma_N / r_N}}{\delta_N^\Theta} = \sqrt{N}r_N J_{N}\lr{\sigma_N / r_N}, \\
    II &\lesssim \sqrt{\log \, N} r_N \delta_N^\Theta  \lr{\frac{J_{N}\lr{\sigma_N / r_N}/ \delta_N^\Theta}{\sigma_N /\lr{ r_N \delta_N^\Theta}}}^2 = \sqrt{\log N} r_N^3 \delta_N^\Theta J_{N}(\sigma_N /r_N)^2 /\sigma_N^2.
  \end{align*}
  By Assumption 3-4, we have $I = o(1)$ and $II = o(1)$. Thus, \eqref{eq:desiredlimit} holds by \eqref{eq:empbound}.
\\\\
 \eqref{eq:emplimits2}: We reuse the notation and proceed similarly. We show 
 \begin{align}\label{eq:seclimemp}
  \Etheta \sup_{\theta \in \Theta_{N, M'}} \abs{\sum_{i = 1}^N \lr{\lvert \forward_{\truth, \theta}(X_i) \rvert_V^2 - \Vert \forward_{\truth, \theta} \Vert_{L^2}^2}} \longrightarrow 0.
 \end{align}
 Define $f_\theta(x) \coloneqq \norm{\forward_{\truth, \theta}(x)}{V}^2 $ for $x \in \mathcal{X}$ and $\theta \in \Theta_{N, M'}$. Take $P \coloneqq \lambda_\mathcal{X}$. We first show that $\theta \in \Theta_{N, M'} \mapsto f_\theta(x)$ is $d_\Theta$-Lipschitz continuous for all $x \in \mathcal{X}$. Expanding by \eqref{eq:vec2}, we have
 \begin{align*}
  f_\theta(x) - f_{\theta'}(x) &= 2 \ip{\forward_{ \theta', \theta}(x)}{ \forward_{ \truth, \theta'}(x)}{V} + \norm{\forward_{ \theta', \theta}(x)}{V}^2.
 \end{align*}
By Cauchy-Schwarz,
\begin{align}
  \abs{f_\theta(x) - f_{\theta'}(x)} &\le 2\norm{\forward_{ \theta', \theta}(x)}{V}\norm{\forward_{\truth, \theta'}(x)}{V} + \norm{\forward_{ \theta', \theta}(x)}{V}^2, \\
  &= \norm{\forward_{ \theta', \theta}(x)}{V}\lr{2\norm{\forward_{\truth, \theta'}(x)}{V} + \norm{\forward_{ \theta', \theta}(x)}{V}}. \label{eq:lasteq}
\end{align}
Hypothesis~\ref{cond:4} now yields
\begin{align*}
 \norm{\forward_{ \theta', \theta}(x)}{V} \le C_1 g_N d_\Theta(\theta, \theta') \le C_2 g_N \delta_N^\Theta,
\end{align*}
where $C_1, C_2 > 0$ are independent of $N \in \mathbb{N}$, $\theta, \theta' \in \Theta_{N ,M'}$, and $x \in \mathcal{X}$. We apply the above display to \eqref{eq:lasteq}, and obtain Lipschitz continuity:
\begin{align}
  \abs{f_\theta(x) - f_{\theta'}(x)} &\le C_1 g_N d_\Theta(\theta, \theta') \cdot \lr{2C_2 g_N \delta_N^\Theta + C_2 g_N \delta_N^\Theta}, \\
   &= 3C_1 C_2 \cdot g_N^2 \delta_N^\Theta d_\Theta(\theta, \theta'). \label{eq:recthat}
\end{align}

We now construct $F_N$. Setting $\theta' \coloneqq \theta_0$ in the above display, and recalling Hypothesis~\ref{cond:4}, we have
\begin{align*}
  \abs{f_\theta(x)} \le 3C_1 C_2 \cdot g_N^2 \delta_N^\Theta d_\Theta(\theta, \theta_0) \le C_3 \lr{\delta_N^\Theta g_N}^2,
\end{align*}
for some constant $C_3$ independent of $x \in \mathcal{X}$ and $\theta \in \Theta_{N, M'}$. We can then choose 
\begin{align*}
  F_N(x) \coloneqq C_3 \lr{\delta_N^\Theta g_N}^2, \quad x \in \mathcal{X}.
\end{align*}
 Since $F_N$ is constant, we have
\begin{align}\label{eq:consteq}
  \norm{F_N}{L^2(P)} = u_N= s_N = C_3 \lr{\delta_N^\Theta g_N}^2,
\end{align}
where the last equality is a valid assignment for $s_N$.

We now bound $I_N$ using $J_N$. Let $Q$ represent a distribution on $\mathcal{X} \times V$ with finite support. Apply $\norm{\cdot}{L^2(Q)}$ to \eqref{eq:recthat}. Then,
 \begin{align}\label{eq:qbound}
  \Vert f_\theta - f_{\theta'} \Vert_{L^2(Q)} \le 3C_1 C_2 \cdot g_N^2 \delta_N^\Theta d_\Theta(\theta, \theta'), \quad \theta, \theta' \in \Theta_{N, M'}.
 \end{align}
 Since $F_N$ is constant, we have $\norm{F_N}{L^2(Q)} = C_3 \lr{\delta_N^\Theta g_N}^2$. Rearranging yields
 \begin{align*}
  g_N^2 \delta_N^\Theta = \frac{\norm{F_N}{L^2(Q)}}{ C_3 \delta_N^\Theta}.
 \end{align*}
 Plugging the above display into \eqref{eq:qbound} yields
 \begin{align}\label{eq:jkl}
  \Vert f_\theta - f_{\theta'} \Vert_{L^2(Q)} \le  \frac{3C_1 C_2}{C_3} \cdot \frac{\norm{F_N}{L^2(Q)}}{ \delta_N^\Theta} d_\Theta(\theta, \theta'), \quad \theta, \theta' \in \Theta_{N, M'}.
 \end{align}
 As before, we obtain
 \begin{align*}
  I_N(t) \lesssim \frac{J_{N}\lr{t\delta_N^\Theta}}{\delta_N^\Theta},
 \end{align*}
 with constant independent of $t > 0$ and $N$. Since $\norm{F_N}{L^2(P)} = s_N$ by \eqref{eq:consteq}, we have $d_N = 1$, where we recall the definition \eqref{eq:ddef}. Hence, by the above display,
\begin{align}
  I_N\lr{d_N} = I_N(1) \lesssim \frac{J_{N}\lr{\delta_N^\Theta}}{\delta_N^\Theta}. \label{eq:preveq}
\end{align}
With $I$ and $II$ from \eqref{eq:firstmax} and \eqref{eq:secmax}, we have
\begin{align}\label{eq:ineqii}
  \Etheta \sup_{\theta \in \Theta_{N, M'}} \abs{\sum_{i = 1}^N \lr{\lvert \forward_{\truth, \theta}(X_i) \rvert_V^2 - \Vert \forward_{\truth, \theta} \Vert_{L^2}^2}}\lesssim \max\curly{I, II}.
\end{align}
 Using \eqref{eq:preveq}, $d_N = 1$ and \eqref{eq:consteq}, we obtain
 \begin{align*}
  I &\lesssim \sqrt{N}\lr{\delta_N^\Theta g_N}^2\frac{J_{N}\lr{\delta_N^\Theta}}{\delta_N^\Theta} = \sqrt{N} g_N^2\delta_N^\Theta J_{N}\lr{\delta_N^\Theta},\\
  II &\lesssim \lr{\delta_N^\Theta g_N}^2\lr{\frac{J_{N}\lr{\delta_N^\Theta}}{\delta_N^\Theta}}^2 = \lr{g_NJ_{N}\lr{\delta_N^\Theta}}^2.
 \end{align*}
 By Assumption 5-6, we have $I = o(1)$ and $II = o(1)$. Thus, \eqref{eq:seclimemp} holds by \eqref{eq:ineqii}.
\end{proof}

\begin{proof}[Justification for Remark~\ref{rem:lip}]
  Fix $M > 0$. Arguing as in \eqref{eq:asdin}, we have 
  \begin{align}\label{eq:remsuf}
    \sup_{\theta \in \regset} \norm{\forward\lr{\theta -  \lambda^\intercal \bs{h}_N }  - \forward\lr{\theta} }{L^2} \lesssim \norm{ \lambda^\intercal \bs{h}_N }{\h} \lesssim \frac{1}{\sqrt{N}}\norm{\perturb}{\h}.
  \end{align}
  We proceed to bound $\norm{\perturb}{\h}$.
  
  Let $\curly{v_i}_{i \le D}$ be an orthonormal eigenbasis of $\fishermatrix^{-1} : E_D \longrightarrow E_D$ with respect to $\ip{\cdot}{\cdot}{\h}$. Denote the eigenvalues by $\curly{\lambda_i}_{i\le D}$. By basis expansion of $P_{E_D}\psi_j$, we obtain
  \begin{align*}
    \pertur^j = \fishermatrix^{-1} P_{E_D}\psi_j = \sum_{i = 1}^D \lambda_i \ip{P_{E_D}\psi_j}{v_i}{\h}v_i.
  \end{align*}
  Thus,
  \begin{align*}
    \norm{\pertur^j}{\h}^2 = \sum_{i = 1}^D \lambda_i^2 \abs{\ip{P_{E_D}\psi_j}{v_j}{\h}}^2 \le \lr{\eigmax{\fishermatrix^{-1}}}^2 \norm{P_{E_D}\psi_j}{\h}^2.
  \end{align*}
  By contraction of projections and \eqref{eq:eiginv}, we now have
  \begin{align}\label{eq:bspec}
  \norm{\perturb}{\h} \le \eigmax{\fishermatrix^{-1}}\norm{P_{E_D}\bs{\psi}}{\mathbb{\h}} \le \frac{\norm{\bs{\psi}}{\mathbb{\h}}}{\eigmin{\fishermatrix}}.
  \end{align}
  Plugging this into \eqref{eq:remsuf} yields 
  \begin{align*}
    \sup_{\theta \in \regset} \norm{\forward\lr{\theta -  \lambda^\intercal \bs{h}_N }  - \forward\lr{\theta} }{L^2} \lesssim \frac{1}{\sqrt{N}\eigmin{\fishermatrix}}.
  \end{align*}
  Lastly, we arrive at our sufficient condition:
  \begin{align*}
    \frac{1}{\sqrt{N}\eigmin{\fishermatrix}} = o\lr{\delta_N^\forward} \Longleftrightarrow \sqrt{N}\delta_N^\forward\eigmin{\fishermatrix}    \longrightarrow \infty \Longleftrightarrow \sqrt{s_N} \eigmin{\fishermatrix} \longrightarrow \infty.
  \end{align*}
\end{proof}

\begin{proof}[Proof of Lemma~\ref{lem:lla}]
  The technique is similar to \cite[Proposition~4.1.6]{Nickl2023}. We consider $N$ large enough that $\mathcal{J}_D [\fishermatrix]^{-1} \mathcal{J}_D^\intercal \succ 0$ by Lemma~\ref{lem:boundinv}. Before starting, we set up some notation.
  
  Fix $\lambda \in \rk$ and $M > 0$. We understand $o$ and $o_P$ notation uniformly over $\theta \in \Theta_{N, M}$ as per \eqref{eq:unifop}. Let $\lesssim$ denote inequality up to a constant, independent of $N$ and $\theta \in \Theta_{N, M}$. We note down the following generic identities for later reference:
  \begin{align}\label{eq:vec1}
    \norm{x + y}{}^2 &= \norm{x}{}^2 + \norm{y}{}^2 + 2\ip{x}{y}{},\\\label{eq:vec2}
    \Vert x \Vert^2 - \Vert y \Vert^2 &= \Vert x - y \Vert^2 + 2 \langle x -y, y\rangle.
  \end{align}
  We define $\forward_{\theta_1, \theta_2} \coloneqq \forward_{\theta_2} - \forward_{\theta_1}$ for $\theta_1, \theta_2 \in \Theta$. The reader should familiarize themselves with Section~\ref{subsec:notation} on vectorized notation. Throughout the proof, we consider $\theta \in \regset$ and make the symbolic substitution $\theta' \coloneqq \theta -  \lambda^\intercal \bs{h}_N$.

  Since $Y_i - \forward_{\theta}(X_i) = \varepsilon_i - \forward_{\truth, \theta}(X_i)$, we have
  \begin{align}
    \ell_N(\theta) - \ell_N\lr{\theta'} &= \frac{1}{2}\sum_{i = 1}^N \norm{\varepsilon_i - \forward_{\truth, \theta'}(X_i)}{V}^2 - \norm{\varepsilon_i - \forward_{\truth, \theta}(X_i)}{V}^2.
  \end{align}
  Apply now \eqref{eq:vec1} to both terms inside the sum, and cancel the $\norm{\varepsilon_i}{V}^2$ terms. Then,
  \begin{align}\label{eq:mainexpansion}
    \ell_N(\theta) - \ell_N\lr{\theta'} = I + \frac{1}{2}(II - III),
  \end{align}
  where we now work on the following terms:
  \begin{align}
    I &= \sum_{i = 1}^N \left\langle \varepsilon_i,  \forward_{\truth, \theta}(X_i) - \forward_{\truth, \theta'}(X_i) \right\rangle_V, \label{eq:fii}\\
    II &= \sum_{i = 1}^N \left\Vert \forward_{\truth, \theta'}(X_i) \right\Vert^2_V,\\
    III &= \sum_{i = 1}^N  \left\Vert \forward_{\truth, \theta}(X_i) \right\Vert^2_V.
  \end{align}
  \textit{Term I}:
  Note 
  \begin{align}\label{eq:lindecompt}
    \forward_{\truth, \theta} = \lin [\theta - \truth] + R_\theta, \quad \theta \in \Theta.
  \end{align}
  Thus, 
  \begin{align*}
    \forward_{\truth, \theta} - \forward_{\truth, \theta' }  = R_{\theta} - R_{\theta' }  + \lin [ \lambda^\intercal \bs{h}_N].
  \end{align*}
  Plug the previous equation into \eqref{eq:fii} and apply \eqref{eq:emplimits1}. Then,
    \begin{align}
      I &= \sum_{i = 1}^N \ip{\varepsilon_i}{ \lin  [\lambda^\intercal \bs{h}_N] (X_i)}{V} + \sum_{i = 1}^N\ip{\varepsilon_i}{R_{\theta}(X_i)}{V} + \sum_{i = 1}^N\ip{\varepsilon_i}{R_{\theta' }(X_i)}{V}, \\ &= \lambda^\intercal S_N + o_P(1), \label{eq:subs1}
    \end{align}
    where $S_N \coloneqq \sum_{i = 1}^N \ip{\varepsilon_i}{\lin  \bs{h}_N (X_i)}{V}$. 
  \\\\
  \textit{Terms II-III}: By Lemma~\ref{lem:reginc}, we can find $M' > 0$ such that
  \begin{align*}
    \Theta_{N, M} -  \lambda^\intercal \bs{h}_N  
    \subset \Theta_{N, M'}, \quad \text{for all }N \in \mathbb{N}.
  \end{align*}
  Alternatively, with our implicit notation: $\theta' \in \Theta_{N, M'}$ for all $\theta \in \Theta_{N, M}$. By \eqref{eq:emplimits2}, we thus have
  \begin{align*}
    II &= N \norm{\forward_{\truth, \theta' }}{L^2}^2 + o_P(1),\\ 
    III &= N \norm{\forward_{\truth, \theta}}{L^2}^2 + o_P(1).
  \end{align*}
  Substitute the above displays and \eqref{eq:subs1} into \eqref{eq:mainexpansion}. Then,
    \begin{align}\label{eq:ff}
        \ell_N(\theta) - \ell_N\lr{\theta' } &=   \lambda^\intercal S_N   + \frac{N}{2}\cdot IV + o_P(1),
    \end{align} 
    where
    \begin{align}\label{eq:ivref}
      IV \coloneqq \left\Vert \forward_{\truth, \theta' } \right\Vert_{L^2}^2 - \norm{\forward_{\truth, \theta}}{L^2}^2.
    \end{align}
    We work on $\norm{\forward_{\truth, \theta'}}{L^2}^2$. Recalling \eqref{eq:lindecompt} and expanding by \eqref{eq:vec1}, we have
    \begin{align}\label{eq:forwarddiff}
      \norm{\forward_{\truth, \theta'}}{L^2}^2 &= \norm{\lin [\theta' - \truth]}{L^2}^2 + \norm{R_{\theta'}}{L^2}^2 + 2\ip{R_{\theta'}}{\lin [\theta' - \truth]}{L^2}.
    \end{align}
    We control the two latter terms. Applying Cauchy-Schwarz, and recalling Hypothesis~\ref{cond:4}, we have
    \begin{align*}
      \abs{\norm{R_{\theta'}}{L^2}^2 + 2\ip{R_{\theta'}}{\lin [\theta' - \truth]}{L^2} } &\le \norm{R_{\theta'}}{L^2}^2 + 2\norm{\lin}{\text{op}}\norm{R_{\theta'}}{L^2}\norm{\theta' - \truth}{L^2}, \\
      &\lesssim \sigma_N^2 + 2\sigma_N\delta_N^\h = o(1 / N).
    \end{align*}
    Here, the last equality follows by Assumption~1.
     Thus, \eqref{eq:forwarddiff} yields 
    \begin{align*}
      \Vert \forward_{\truth, \theta'}\Vert_{L^2}^2 = \Vert \lin [\theta' - \truth] \Vert^2_{L^2} + o(1 /N).
    \end{align*}
    Arguing similarly, we also have
    \begin{align*}
      \Vert \forward_{\truth, \theta}\Vert_{L^2}^2 = \Vert \lin [\theta - \truth] \Vert^2_{L^2} + o(1 /N).
    \end{align*}
    Plugging previous two displays into \eqref{eq:ivref}, we obtain
    \begin{align*}
      IV  = \norm{\lin [\theta'  - \truth]}{L^2}^2  - \norm{\lin [\theta - \truth]}{L^2}^2 + o(1 /N).
    \end{align*}
    Applying \eqref{eq:vec2}, we have
    \begin{align}\label{eq:exp}
      IV = \norm{\lin  \sq{\lambda^\intercal \bs{h}_N} }{L^2}^2 - 2 \ip{\lin [\theta-\truth]}{\lin   \sq{\lambda^\intercal \bs{h}_N} }{L^2} + o(1 /N). 
    \end{align}
    We work on the first term. Recall from the proof of Lemma~\ref{lem:asymptnorm} that 
    \begin{align}\label{eq:wn}
      W_N \coloneqq \frac{1}{\sqrt{N}}\renorm \ip{\varepsilon_1}{\lin \perturb(X_1)}{V} 
    \end{align}
    satisfies $\Vtheta W_N = I_k / N$, and consequently $\Vtheta[\lambda^\intercal W_N] = \norm{\lambda}{\rk}^2 / N$. We argue
    \begin{align*}
      \Vtheta\sq{\lambda^\intercal W_N} =  \norm{\lin  \sq{\lambda^\intercal \bs{h}_N} }{L^2}^2
    \end{align*}
    so that
    \begin{align}\label{eq:varcond}
      \norm{\lin  \sq{\lambda^\intercal \bs{h}_N} }{L^2}^2 = \frac{1}{N}\norm{\lambda}{\rk}^2.
    \end{align}
    Recall \eqref{eq:wn} and \eqref{eq:perdef}. By linearity, 
    \begin{align*}
      \lambda^\intercal W_N  = \ip{\varepsilon_1}{\lin  \sq{\lambda^\intercal \bs{h}_N } (X_1)}{V}.
    \end{align*}
    Thus, for all $x \in \mathcal{X}$, we have
    \begin{align*}
      \mathcal{L}\lr{\lambda^\intercal W_N \mid X_1 = x} = \mathcal{L}\lr{\ip{\varepsilon_1}{\lin  \sq{\lambda^\intercal \bs{h}_N} (x)}{V}} = \mathcal{N}\lr{0, \norm{\lin  \sq{\lambda^\intercal \bs{h}_N} (x)}{V}^2}.
    \end{align*}
    Applying the law of total variance, and recalling $X_1 \sim \lambda_\mathcal{X}$, we have
    \begin{align*}
      \Vtheta\sq{\lambda^\intercal W_N} &= \Etheta  \sq{\Vtheta\sq{\lambda^\intercal W_N \mid X_1}} + \Vtheta\sq{\Etheta \sq{\lambda^\intercal W_N\mid X_1}}, \\ &= \Etheta \norm{\lin  \sq{\lambda^\intercal \bs{h}_N} (X_1)}{V}^2, \\
      &= \norm{\lin  \sq{\lambda^\intercal \bs{h}_N} }{L^2}^2.
    \end{align*}
    Thus, \eqref{eq:varcond} holds. Plug \eqref{eq:varcond} into \eqref{eq:exp} and recall subsequently $\mathcal{I} = \lin^* \lin$. Then,
    \begin{align*}
      IV &= \frac{1}{N}\norm{\lambda}{\rk}^2 - 2 \ip{\lin [\theta-\truth]}{\lin   \sq{\lambda^\intercal \bs{h}_N} }{L^2} + o(1 /N), \\
      &= \frac{1}{N}\norm{\lambda}{\rk}^2 - 2 \ip{\theta-\truth}{\fisher \sq{\lambda^\intercal \bs{h}_N} }{L^2} + o(1 /N).
    \end{align*}
    Plug now the above display into \eqref{eq:ff} and recall $\theta' = \theta -  \lambda^\intercal \bs{h}_N$. Then,
    \begin{align*}
      \ell_N\left(\theta\right) - \ell_N(\theta -  \lambda^\intercal \bs{h}_N ) =\frac{1}{2}\Vert \lambda \Vert^2_{\rk} +  \lr{\lambda^\intercal S_N  -  N\ip{\theta - \truth}{\mathcal{I} \sq{\lambda^\intercal \bs{h}_N}}{\h}} +  o_P(1).
    \end{align*}
    It is clear that \eqref{eq:llr} follows, if we can show that the above bracket equals $-\lambda^\intercal Z_N(\theta) + o_P(1)$. That is,
    \begin{align*}
       \lambda^\intercal S_N  - N\ip{\theta - \truth}{\mathcal{I}\sq{ \lambda^\intercal \bs{h}_N } }{\h} = -\lambda^\intercal Z_N(\theta) + o_P(1),
    \end{align*}
    or equivalently,
    \begin{align}\label{eq:lts}
      \lambda^\intercal S_N  = -\lambda^\intercal Z_N(\theta) + N\ip{\theta - \truth}{\mathcal{I}\sq{ \lambda^\intercal \bs{h}_N } }{\h} + o_P(1).
    \end{align}
    We work on $\lambda^\intercal S_N$. Recalling \eqref{eq:perdef}, and using linearity,
    \begin{align*}
      S_N  =\sum_{i = 1}^N \ip{\varepsilon_i}{\lin  \bs{h}_N (X_i)}{V} = \frac{1}{\sqrt{N}}\renorm\sum_{i = 1}^N \ip{\varepsilon_i}{\lin \perturb(X_i)}{V}.
    \end{align*}
    Thus, recalling \eqref{eq:centering}, we have
    \begin{align}
      S_N
      = \sqrt{N} \renorm\lr{\hat{\Psi}_N - \Psi\truth} = \sqrt{N}\renorm\hat{\Psi}_N -\sqrt{N} \renorm\Psi\truth. \label{eq:brack}
    \end{align}
    By definition of $Z_N(\theta)$, we have $\sqrt{N}\renorm\hat{\Psi}_N = -Z_N(\theta) + \sqrt{N}\renorm\Psi \theta$. Thus,
    \begin{align*}
      S_N = -Z_N(\theta) + \sqrt{N}\renorm \Psi\lr{\theta - \truth}.
    \end{align*}
    Hence, we can write
    \begin{align}\label{eq:vdef}
      \lambda^\intercal S_N = -\lambda^\intercal Z_N(\theta) + V, \quad \text{where }V \coloneqq \sqrt{N}\lambda^\intercal \renorm\Psi  (\theta - \truth).
    \end{align}
    Observe now that \eqref{eq:lts} follows by \eqref{eq:vdef} if
    \begin{align}\label{eq:lastsuff}
      V = N\ip{\theta - \truth}{\mathcal{I} \sq{\lambda^\intercal \bs{h}_N} }{\h} + o(1).
    \end{align}
    We establish the above equality by working from left to right. 
    
    Recall the definition of $V$ from \eqref{eq:vdef}, and $\Psi = \ip{\bs{\psi}}{\cdot}{\h}$. By linearity,
    \begin{align*}
      V &= \sqrt{N}\lambda^\intercal \renorm \ip{ \theta - \truth}{ \bs{\psi}}{\h} = \sqrt{N}\ip{ \theta - \truth}{\lambda^\intercal \renorm\bs{\psi}}{\h}.
    \end{align*}
    Thus, letting
    \begin{align*}
      I' \coloneqq \sqrt{N}\ip{ \theta - \truth}{\lambda^\intercal \renorm \lr{ \bs{\psi} - P_{E_D}\bs{\psi} }}{\h},
    \end{align*}
    we have
    \begin{align}
      V = \sqrt{N}\ip{ \theta - \truth}{\lambda^\intercal \renorm P_{E_D}\bs{\psi}}{\h} + I'. \label{eq:firststep}
    \end{align}
    We now show $I' = o(1)$. Applying Cauchy-Schwarz, Lemma~\ref{lem:boundinv} and Hypothesis~\ref{cond:4},
    \begin{align*}
      \lvert I' \rvert &\lesssim \sqrt{N} \norm{\theta - \truth}{\h}   \norm{\renorm}{\text{op}} \norm{\bs{\psi} - P_{E_D}\bs{\psi}}{\h}, \\ 
      &\lesssim  \sqrt{N}\delta_N^\h \norm{\bs{\psi} - P_{E_D}\bs{\psi}}{\h} \longrightarrow 0.
    \end{align*}
    Here, we pass the limit by Assumption 3. Plugging $I' = o(1)$ into \eqref{eq:firststep}, we have
    \begin{align}\label{eq:rec12}
      V = \sqrt{N}\ip{ \theta - \truth}{\lambda^\intercal \renorm P_{E_D}\bs{\psi}}{\h} + o(1).
    \end{align}
    We now find an alternative expression for $\lambda^\intercal \renorm P_{E_D}\bs{\psi}$. Recall \eqref{eq:perdef} and $\perturb = \fishermatrix^{-1} P_{E_D}\bs{\psi}$. By linearity, we have
    \begin{align*}
      \sqrt{N}\fishermatrix  \sq{\lambda^\intercal \bs{h}_N}  =  \mathcal{I_D}\sq{\lambda^\intercal \renorm  \perturb }= \lambda^\intercal \renorm \mathcal{I_D} \perturb = \lambda^\intercal \renorm P_{E_D}\bs{\psi}.
    \end{align*}
    Substituting this into \eqref{eq:rec12}, 
    \begin{align*}
      V &= N\ip{ \theta - \truth}{\fishermatrix  \sq{\lambda^\intercal \bs{h}_N} }{\h} + o(1),\\
      &= N\ip{ \theta - P_{E_D}\truth}{\mathcal{I}  \sq{\lambda^\intercal \bs{h}_N} }{\h} + o(1).
    \end{align*}
    Here, the last equality follows by self-adjointness of $P_{E_D}$. Thus, letting
    \begin{align*}
      II' \coloneqq N\ip{ \truth - P_{E_D}\truth}{\mathcal{I}\sq{ \lambda^\intercal \bs{h}_N} }{\h},
    \end{align*}
    we have
    \begin{align*}
      V = N\ip{ \theta - \truth}{\mathcal{I} \sq{\lambda^\intercal \bs{h}_N} }{\h} + II' + o(1).
    \end{align*}
    We obtain \eqref{eq:lastsuff}, if we can show $II' = o(1)$. 
    
    By Cauchy-Schwarz,
    \begin{align}\label{eq:theabove}
      \lvert II' \rvert & \le N\norm{\mathcal{I} \sq{ \lambda^\intercal \bs{h}_N } }{\h}\norm{\truth - P_{E_D}\truth}{\h}. 
    \end{align}
    Recalling $\mathcal{I} = \lin^* \lin$, and using \eqref{eq:varcond}, we have
    \begin{align*}
      \norm{\mathcal{I} \sq{ \lambda^\intercal \bs{h}_N} }{\h} \le \norm{\lin^* }{\text{op}}\norm{\lin \sq{ \lambda^\intercal \bs{h}_N} }{L^2} = \frac{1}{\sqrt{N}}\norm{\lin^*}{\text{op}}\norm{\lambda}{\rk}.
    \end{align*}
    Plugging this into \eqref{eq:theabove} yields 
    \begin{align*}
      \lvert II' \rvert \lesssim \sqrt{N}\norm{\truth - P_{E_D}\truth}{\h} \longrightarrow 0.
    \end{align*}
    Here, we use Assumption~2 to pass the limit.
\end{proof}
\begin{lemma}\label{lem:interchangelims}
  Let $f_N, g_N: E_D \longrightarrow \mathbb{R}$ be measurable functions satisfying
  \begin{align*}
    \int_{A_N} e^{f_N} \, d\Pi_N < \infty \text{ and } \int_{A_N} e^{g_N} \, d\Pi_N < \infty,
  \end{align*}
  where $A_N \in \mathcal{B}\lr{E_D}$ for all $N \in \mathbb{N}$. Assume 
  \begin{align*}
    \sup_{\theta \in A_N} \abs{f_N(\theta) - g_N(\theta) } \longrightarrow 0.
  \end{align*}
  Then,
  \begin{align*}
    \int_{A_N} e^{f_N} \, d\Pi_N  = (1 + o(1))\int_{A_N} e^{g_N} \, d\Pi_N.
  \end{align*}
\end{lemma}
\begin{proof}
  Let $I \coloneqq \int_{A_N} e^{f_N} \, d\Pi_N$ and $II \coloneqq \int_{A_N} e^{g_N} \, d\Pi_N$. We have to show $\abs{\frac{I -II}{II}} \longrightarrow 0$. Now,
  \begin{align*}
    \frac{I -II}{II} = \frac{\int_{A_N} e^{f_N} - e^{g_N} \, d\Pi_N}{\int_{A_N} e^{g_N} \, d\Pi_N}= \frac{\int_{A_N} e^{g_N}\lr{e^{f_N - g_N} - 1} \, d\Pi_N}{\int_{A_N} e^{g_N} \, d\Pi_N}.
  \end{align*}
  Let $a_N \coloneqq \sup_{\theta \in A_N} \abs{f_N(\theta) - g_N(\theta) } = o(1)$. Then,
  \begin{align*}
    \abs{\frac{I -II}{II}} &\le\frac{\int_{A_N} e^{g_N}\sup_{\theta \in A_N} \abs{e^{f_N(\theta) - g_N(\theta)} - 1} \, d\Pi_N}{\int_{A_N} e^{g_N} \, d\Pi_N}, \\ &= \sup_{\theta \in A_N} \abs{e^{f_N(\theta) - g_N(\theta)} - 1}, \\ &\le \max\curly{e^{a_N} - 1,1 - e^{-a_N}} \longrightarrow 0.
  \end{align*}
\end{proof}
\begin{proof}[Proof of Lemma~\ref{lem:prjbvm}]
  Before starting, we recall $M_0$ from Hypothesis~\ref{cond:posteriorcontraction}, which has the property
  \begin{align}\label{eq:fhyp}
    \Pi_N(\Theta_{N, M_0} \mid \mathcal{D}_N) = O_P(e^{-bs_N}).
  \end{align}
  We pick $M' \ge M_0$ from Lemma~\ref{lem:reginc} with $M = M_0$. We write $\bs{h}_N = (h_N^i)_{i \le k}$. \\\\
  \eqref{eq:distlim}:  Recall $c > 0$ from Hypothesis~\ref{cond:posteriorcontraction}. Let $A_N \coloneqq \cap_{i = 1}^k A_{N, i}$, where 
  \begin{align}\label{eq:tndef}
    A_{N, i} \coloneqq   \curly{\theta \in E_D :  \abs{\ip{\theta}{h_N^i}{\pspace}} \le  2\sqrt{c s_N}\norm{h_N^i}{\pspace}},
  \end{align}
  and define $\bar{\Theta}_{N} \coloneqq  \Theta_{N, M'} \cap A_N$. We first argue that \eqref{eq:distlim} is equivalent to
  \begin{align}\label{eq:condlimit}
    \mathcal{L}^{\bar{\Theta}_N} \lr{Z_N \mid \mathcal{D}_N} &\dlim \mathcal{N}(0, I_k), \quad \text{in }\Ptheta\text{-probability}.
  \end{align}
  By the triangle inequality, it suffices to show 
  \begin{align}\label{eq:distregset}
    d_{\text{BL}}\lr{\mathcal{L}\lr{Z_N \mid \mathcal{D}_N}, \mathcal{L}^{\bar{\Theta}_N} \lr{Z_N \mid \mathcal{D}_N}} \plim 0.
  \end{align}
  Let $d_{\text{TV}}$ denote the total variation metric. Then,
  \begin{align*}
    d_{\text{BL}}\lr{\mathcal{L}\lr{Z_N \mid \mathcal{D}_N}, \mathcal{L}^{\bar{\Theta}_N} \lr{Z_N \mid \mathcal{D}_N}}  \le d_{\text{TV}}\lr{\mathcal{L}\lr{Z_N \mid \mathcal{D}_N}, \mathcal{L}^{\bar{\Theta}_N} \lr{Z_N \mid \mathcal{D}_N}}.
  \end{align*}
  Using that $d_{\text{TV}}$ is a contraction under pushforwards,
  \begin{align*}
    d_{\text{TV}}\lr{\mathcal{L}\lr{Z_N \mid \mathcal{D}_N}, \mathcal{L}^{\bar{\Theta}_N} \lr{Z_N \mid \mathcal{D}_N}} \le d_{\text{TV}}\lr{\Pi_N(\cdot \mid \mathcal{D}_N), \Pi^{\bar{\Theta}_N}_N(\cdot \mid \mathcal{D}_N)}.
  \end{align*}
  One quickly confirms that $\Pi_N^{\bar{\Theta}_N}(\cdot \mid \mathcal{D}_N)$ from \eqref{eq:condpos}, is the posterior $\Pi_N(\cdot \mid \mathcal{D}_N)$ from \eqref{eq:oripos}, where we replace $\Pi_N$ by 
  \begin{align*}
    \Pi^{\bar{\Theta}_N}_N(A) \coloneqq \frac{\Pi_N(A \cap \bar{\Theta}_N)}{\Pi_N(\bar{\Theta}_N)}, \quad A \in \mathcal{B}\lr{E_D}.
  \end{align*}
  Thus, applying \cite[p.~142]{Vaart_1998}, we have
  \begin{align}\label{eq:tvineq}
    d_{\text{TV}}\lr{\Pi_N(\cdot \mid \mathcal{D}_N), \Pi^{\bar{\Theta}_N}_N(\cdot \mid \mathcal{D}_N)} \le 2 \Pi_N(\bar{\Theta}_N^c \mid \mathcal{D}_N).
  \end{align}
  Hence, to establish \eqref{eq:distregset}, it certainly suffices to show
  \begin{align}\label{eq:expcontract}
    \posterior{\bar{\Theta}_N^c}  = O_P\lr{e^{-bs_N}}.
  \end{align}
  
  By sub-additivity,
  \begin{align*}
    \posterior{\bar{\Theta}_N^c} \le \posterior{\Theta_{N, M'}^c} + \posterior{A_N^c}.
  \end{align*}
  Using $M_0 \le M'$ and subsequently \eqref{eq:fhyp}, we have
  \begin{align*}
    \posterior{\Theta_{N, M'}^c} \le \posterior{\Theta_{N, M_0}^c} = O_P\lr{e^{-bs_N}}.
  \end{align*}
  Furthermore, from Hypothesis~\ref{cond:posteriorcontraction}, we can establish
  \begin{align*}
    \posterior{A_N^c} = O_P\lr{e^{-bs_N}} 
  \end{align*}
  by showing $\Pi_N\lr{A_N^c} \le e^{-cs_N}$ for $N$ large enough. 
  
  Recall $A_N \coloneqq \cap_{i = 1}^k A_{N, i}$. Clearly,
  \begin{align}\label{eq:decompofan}
    \Pi_N\lr{A_N^c} \le \sum_{i = 1}^k \Pi_N(A_{N, i}^c).
  \end{align}
  Let $i \le k$. By definition of a Gaussian measure with identity covariance operator,
  \begin{align*}
    \ip{\theta_N}{h_N^i}{\pspace} \sim \mathcal{N}\lr{0, \norm{h_N^i}{\pspace}^2}.
  \end{align*}
  Let $\tilde{\Phi}$ denote the distribution function of $\abs{Z}$ where $Z \sim \mathcal{N}(0, 1)$. Recalling \eqref{eq:tndef}, the above display yields
  \begin{align*}
    \Pi_N\lr{A_{N, i}} = \begin{cases} 
      1, & \text{if } h_N^i = 0, \\
      \tilde{\Phi}\left( 2\sqrt{c s_N} \right), & \text{if } h_N^i \neq 0.
  \end{cases}
  \end{align*}
  Furthermore, by the Gaussian tail bound \cite[p.~37]{Giné_Nickl_2021}, we have 
  \begin{align*}
    1 - \tilde{\Phi}\left( 2\sqrt{c s_N} \right) \le 2e^{-2cs_N}.
  \end{align*}
  Hence,
  \begin{align*}
    \Pi_N\lr{A_{N, i}^c} \le 2e^{-2c s_N}.
  \end{align*}
Thus, recalling \eqref{eq:decompofan}, 
  \begin{align*}
    \Pi_N\lr{A_N^c} \le 2k e^{-2c s_N} \le e^{-c s_N}, 
  \end{align*}
  where the last inequality holds for $N$ large enough.
  
  We proceed to show \eqref{eq:condlimit}. As noted in \cite[Section~7.6]{Nickl_2020}, it suffices to show
  \begin{align}\label{eq:moment}
    M_N(\lambda) \coloneqq \mathbb{E}^{\bar{\Theta}_N}\sq{\exp\lr{\lambda^\intercal Z_N }\mid \mathcal{D}_N} \plim \exp\lr{\frac{1}{2}\Vert \lambda\Vert^2_{\rk}}, \quad \text{for all }\lambda \in \rk.
  \end{align}
  Fix $\lambda \in \rk$. The definition \eqref{eq:condpos} states 
  \begin{align*}
    \frac{d\Pi_N^{\bar{\Theta}_N}(\cdot \mid \mathcal{D}_N)}{d\Pi_N} = \frac{\mathbbm{1}_{\bar{\Theta}_N}e^{\ell_N}}{\int_{\bar{\Theta}_N} e^{\ell_N}\, d\Pi_N} .
  \end{align*}
  Thus,
  \begin{align*}
    M_N(\lambda)  &= \int \exp\lr{\lambda^\intercal Z_N(\theta)} d\Pi_N^{\bar{\Theta}_N}(\theta \mid \mathcal{D}_N), \\&= \frac{\int_{\bar{\Theta}_N}\exp\lr{\lambda^\intercal Z_N(\theta) + \ell_N(\theta)} \, d\Pi_N(\theta)}{\int_{\bar{\Theta}_N} \exp(\ell_N(\theta))\, d\Pi_N(\theta)}.
  \end{align*}
  For notational purposes, we note $1 + o(1) = \exp(o(1))$. Applying Lemma~\ref{lem:lla} and Lemma~\ref{lem:interchangelims},
  \begin{align}
    M_N(\lambda) = \frac{\int_{\bar{\Theta}_N} \exp\lr{\likelihood{\theta - \hperturb}} \, d\Pi_N(\theta)}{\int_{\bar{\Theta}_N} \exp(\ell_N(\theta))\, d\Pi_N(\theta)}\exp\left(\Vert \lambda\Vert^2_{\rk}/2 + o_P(1) \right). \label{eq:ratio2}
  \end{align}
  By the Cameron-Martin formula \cite[Theorem~2.6.13]{Giné_Nickl_2021},
  \begin{align}\label{eq:shift}
    \int_{\bar{\Theta}_N} \exp\lr{\likelihood{\theta - \hperturb}} \, d\Pi_N(\theta) = \int_{\bar{\Theta}_N -  \lambda^\intercal \bs{h}_N } \exp\left(\ell_N(\theta) - R_N(\theta)  \right) \, d\Pi_N(\theta),
  \end{align}  
where 
\begin{align*}
  R_N(\theta) \coloneqq \ip{\theta}{ \lambda^\intercal \bs{h}_N }{\pspace} +  \frac{1}{2}\norm{ \lambda^\intercal \bs{h}_N }{\pspace}^2, \quad \theta \in E_D.
\end{align*}
We argue 
\begin{align}\label{eq:rlimmit}
  \sup\curly{\abs{R(\theta)} : \theta \in \bar{\Theta}_N - \hperturb} \longrightarrow 0 ,
\end{align}
so that we can apply Lemma~\ref{lem:interchangelims}. Now,
\begin{align}
  \sup_{\theta \in \bar{\Theta}_N}\abs{R(\theta)} &= \sup_{\theta \in \bar{\Theta}_N}\abs{R\lr{\theta - \hperturb}},\\ 
  &= \sup_{\theta \in \bar{\Theta}_N}\abs{ \ip{\theta}{ \lambda^\intercal \bs{h}_N }{\pspace} - \frac{1}{2}\norm{ \lambda^\intercal \bs{h}_N }{\pspace}^2},
  \\ &\le  \sup_{\theta \in \bar{\Theta}_N}\abs{\ip{\theta}{ \lambda^\intercal \bs{h}_N }{\pspace}} + \frac{1}{2}\norm{ \lambda^\intercal \bs{h}_N }{\pspace}^2. \label{eq:lasteqr}
\end{align}
Noting $\bar{\Theta}_N \subset A_N$, and recalling \eqref{eq:tndef}, we have
\begin{align}
  \sup_{\theta \in \bar{\Theta}_N} \abs{\ip{\theta}{ \lambda^\intercal \bs{h}_N }{\pspace}} &\le  \sum_{i = 1}^k \abs{\lambda_i} \sup_{\theta \in A_N} \abs{\ip{\theta}{h_N^i}{\pspace}},  \\ &\le  2 \sqrt{c s_N} \sum_{i = 1}^k \abs{\lambda_i} \norm{h_N^i}{\pspace}, \\
  &\lesssim \sqrt{s_N} \norm{\bs{h}_N}{\pspace}. \label{eq:rterm1}
\end{align}
We plug \eqref{eq:rterm1} into \eqref{eq:lasteqr} and use $s_N \ge 1$:
\begin{align}
  \sup\curly{\abs{R(\theta)} : \theta \in \bar{\Theta}_N - \hperturb} 
  &\lesssim \sqrt{s_N}\norm{\bs{h}_N}{\pspace} + \norm{\bs{h}_N}{\pspace}^2, \\ &\lesssim \sqrt{s_N}\norm{\bs{h}_N}{\pspace} + \lr{\sqrt{s_N}\norm{\bs{h}_N}{\pspace}}^2. \label{eq:wepinto}
\end{align}
We now show $\sqrt{s_N}\norm{\bs{h}_N}{\pspace} \longrightarrow 0$. Arguing as in \eqref{eq:asdin}, 
\begin{align*}
  \norm{\bs{h}_N}{\pspace} \lesssim \frac{1}{\sqrt{N}} \norm{\perturb }{\pspace}.
\end{align*}
Thus, recalling $s_N = N\lr{\delta_N^\forward}^2$, it holds that
\begin{align}\label{eq:argasin}
  \sqrt{s_N}\norm{\bs{h}_N}{\pspace} \lesssim \sqrt{s_N}\cdot \frac{1}{\sqrt{N}} \norm{\perturb }{\pspace} = \delta_N^\forward \norm{\perturb }{\pspace} \longrightarrow 0,
\end{align}
where we pass the limit using Assumption~1. The bound \eqref{eq:wepinto} thus goes to zero, so that \eqref{eq:rlimmit} holds. 

Applying Lemma~\ref{lem:interchangelims} to the right-hand side of \eqref{eq:shift}, we obtain
\begin{align*}
  \int_{\bar{\Theta}_N} \exp\lr{\likelihood{\theta - \hperturb}} \, d\Pi_N(\theta) &= (1 + o(1))\int_{\bar{\Theta}_N -  \lambda^\intercal \bs{h}_N }\exp\lr{\ell_N(\theta)} \, d\Pi_N(\theta). 
\end{align*}
We plug this into \eqref{eq:ratio2}:
\begin{align*}
  M_N(\lambda) = \frac{\int_{\bar{\Theta}_N -  \lambda^\intercal \bs{h}_N }\exp\lr{\ell_N(\theta)} \, d\Pi_N(\theta)}{\int_{\bar{\Theta}_N} \exp(\ell_N(\theta))\, d\Pi_N(\theta)}\exp\lr{\frac{1}{2}\Vert \lambda\Vert^2_{\rk} + o_P(1)}.
\end{align*}
Dividing the numerator and denominator by $\int_{E_D} e^{\ell_N} \,  d\Pi_N$ and recalling \eqref{eq:oripos}, we have
  \begin{align*}
    M_N(\lambda) = \frac{\Pi_N\lr{\bar{\Theta}_N -  \lambda^\intercal \bs{h}_N  \mid \mathcal{D}_N}}{\Pi_N\lr{\Bar{\Theta}_N \mid \mathcal{D}_N}} \exp\lr{\frac{1}{2}\Vert \lambda\Vert^2_{\rk} + o_P(1)}.
  \end{align*}
  Recall that we wish to show \eqref{eq:moment}. Since $\Pi_N\lr{\Bar{\Theta}_N \mid \mathcal{D}_N} \plim 1$ by \eqref{eq:expcontract}, it now suffices to show 
  \begin{align}\label{eq:thussuff}
    \posterior{\bar{\Theta}_N -  \lambda^\intercal \bs{h}_N } \plim 1.
  \end{align}
  Since $\bar{\Theta}_N = \Theta_{N, M'} \cap A_N$, we have
  \begin{align*}
    \bar{\Theta}_N -  \lambda^\intercal \bs{h}_N = \lr{\Theta_{N, M'} - \lambda^\intercal \bs{h}_N} \cap \lr{ A_N -  \lambda^\intercal \bs{h}_N}.
  \end{align*}
  By construction, we have $ \Theta_{N, M'} -  \lambda^\intercal \bs{h}_N \supset \Theta_{N,M_0}$ for $N$ large enough. Coupling this with the above display,
  \begin{align*}
    \bar{\Theta}_N -  \lambda^\intercal \bs{h}_N \supset \Theta_{N,M_0}\cap \lr{ A_N -  \lambda^\intercal \bs{h}_N}, \quad \text{for } N \text{ large enough}.
  \end{align*}
  Recalling \eqref{eq:fhyp}, we have $\Pi_N(\Theta_{N, M_0} \mid \mathcal{D}_N) \plim 1$. Thus, we can establish \eqref{eq:thussuff} by showing 
  \begin{align*}
    \posterior{A_N -  \lambda^\intercal \bs{h}_N  } \plim 1.
  \end{align*}
  For this, it suffices to show
  \begin{align}\label{eq:onehalf}
    \frac{1}{2}A_N  \subset A_N -  \lambda^\intercal \bs{h}_N, \quad \text{for $N$ large enough},
  \end{align}
  since $\posterior{\frac{1}{2}A_N } \plim 1$. Indeed, the latter limit follows by the same argument we used to establish $\Pi_N(A_N \mid \mathcal{D}_N) \plim 1$. Here, we simply replace $2\sqrt{c s_N}$ with $\sqrt{c s_N}$ in \eqref{eq:tndef}.
  
  To establish \eqref{eq:onehalf}, suppose $x \in \frac{1}{2}A_N$. Then, by definition of $A_N$,
  \begin{align*}
    \abs{\ip{x}{h^i_N }{\pspace}} \le \sqrt{c s_N}\norm{h^i_N}{\pspace}, \quad i \le k.
  \end{align*}
  Using the triangle inequality, the above inequality, and Cauchy-Schwarz,
  \begin{align}
    \abs{\ip{x + \lambda^\intercal \bs{h}_N}{h^i_N}{\pspace} }&\le \abs{\ip{x}{h^i_N }{\pspace}}  + \abs{\ip{\lambda^\intercal \bs{h}_N}{h^i_N}{\pspace}},\\
    &= \sqrt{c s_N}\norm{h^i_N}{\pspace} + \norm{\hperturb}{\pspace}\norm{h^i_N}{\pspace}. \label{eq:pluggingthisinto}
  \end{align}
  Since $s_N \ge 1$, we have 
  \begin{align*}
    \norm{\hperturb}{\pspace} \lesssim \norm{\bs{h}_N}{\pspace} \lesssim \sqrt{s_N}\norm{\bs{h}_N}{\pspace} \longrightarrow 0,
  \end{align*}
  where we use \eqref{eq:argasin} to pass the limit. In particular, $\norm{\hperturb}{\pspace} \le \sqrt{c s_N}$ for $N$ large enough. Plugging this into \eqref{eq:pluggingthisinto}, we have
  \begin{align*}
    \abs{\ip{x + \lambda^\intercal \bs{h}_N}{h^{i}_N}{\pspace} } \le 2\sqrt{c s_N}\norm{h^{i}_N}{\pspace}, \quad \text{for } i \le k \text{ and }N\text{ large enough}.
  \end{align*}
  This is the statement: $x \in A_N -  \lambda^\intercal \bs{h}_N $ for $N$ large enough. Thus, \eqref{eq:onehalf} holds.
  \\\\
  \eqref{eq:expp} and \eqref{eq:var}: We first argue that the following conditions are sufficient: 
  \begin{align}\label{eq:fflim}
    \mathbb{E}[\lambda^\intercal Z_N \mid \mathcal{D}_N] &\plim 0, \\\label{eq:sslim}
    \mathbb{E}[(\lambda^\intercal Z_N)^2 \mid \mathcal{D}_N] &\plim \norm{\lambda}{\rk}^2, \quad \text{for all } \lambda \in \rk.
  \end{align}
  
  Indeed, assume that \eqref{eq:fflim} and \eqref{eq:sslim} hold. Then, \eqref{eq:expp} follows by plugging $\lambda = \bs{e}_i$ into \eqref{eq:fflim} for all $i \le k$. We note that \eqref{eq:var} is equivalent to
  \begin{align}\label{eq:bnconv}
    B_N(\bs{e}_i, \bs{e}_j) \plim \delta_{ij}, \quad \text{for all }i,j \le k, 
  \end{align}
  where $\delta_{ij}$ is the Kronecker delta, and
  \begin{align}\label{eq:bndef}
    B_N(\lambda, \lambda') \coloneqq \lambda^\intercal \mathbb{V}[Z_N \mid \mathcal{D}_N] \lambda', \quad \lambda, \lambda' \in \rk.
  \end{align}
  Now, using \eqref{eq:fflim} and \eqref{eq:sslim}, we have
  \begin{align}\label{eq:onearg}
    B_N(\lambda, \lambda) = \mathbb{V}[\lambda^\intercal Z_N \mid \mathcal{D}_N] = \Epost{\lr{\lambda^\intercal Z_N}^2} - \lr{\mathbb{E}[\lambda^\intercal Z_N \mid \mathcal{D}_N]}^2 \plim \norm{\lambda}{\rk}^2.
  \end{align}
  By bilinearity and symmetry, we have the polarization identity
  \begin{align}\label{eq:billinearity}
    B_N(\lambda, \lambda')  = \frac{1}{4} \lr{B_N(\lambda + \lambda', \lambda + \lambda') - B_N(\lambda - \lambda', \lambda - \lambda')}.
  \end{align}
  Thus, \eqref{eq:onearg} yields
  \begin{align}\label{eq:translates}
    B_N(\lambda, \lambda') \plim& \frac{1}{4} \lr{\norm{\lambda + \lambda'}{\rk}^2 - \norm{\lambda- \lambda'}{\rk}^2}.
  \end{align}
  Noting the polarization identity, 
  \begin{align*}
    \lambda^\intercal \lambda' = \frac{1}{4} \lr{\norm{\lambda + \lambda'}{\rk}^2 - \norm{\lambda- \lambda'}{\rk}^2},
  \end{align*}
  \eqref{eq:translates} states
  \begin{align*}
    B_N(\lambda, \lambda') \plim \lambda^\intercal \lambda'.
  \end{align*}
  Plugging in $\lambda = \bs{e}_i$ and $\lambda' = \bs{e}_j$, we obtain \eqref{eq:bnconv}.
  
  Fix $\lambda \in \rk$. For \eqref{eq:fflim} and \eqref{eq:sslim}, we start by showing
  \begin{align}
    \mathbb{E}^{\bar{\Theta}_N}\sq{(\lambda^\intercal Z_N)^2 \bigl\vert \mathcal{D}_N} &\plim   \Vert \lambda \Vert_{\rk}^2, \label{eq:f2} \\
    \mathbb{E}^{\bar{\Theta}_N}\sq{\lambda^\intercal Z_N \mid \mathcal{D}_N} &\plim 0. \label{eq:f1}
  \end{align}
   We later 'remove' the conditioning on $\bar{\Theta}_N$, by showing $I = o_P(1)$ and $II = o_P(1)$, where
   \begin{align}\label{eq:ssqq}
     I &\coloneqq \mathbb{E}^{\bar{\Theta}_N}[(\lambda^\intercal Z_N)^2 \mid \mathcal{D}_N] - \mathbb{E}[(\lambda^\intercal Z_N)^2 \mid \mathcal{D}_N],\\\label{eq:ssqq1}
     II &\coloneqq \mathbb{E}^{\bar{\Theta}_N}[\lambda^\intercal Z_N \mid \mathcal{D}_N] - \mathbb{E}[\lambda^\intercal Z_N \mid \mathcal{D}_N].
   \end{align}
   
   Using Remark~\ref{rem:subseqence} on \eqref{eq:distlim} and \eqref{eq:moment}, we can WLOG assume that 
  \begin{align}
    \mathcal{L}^{\bar{\Theta}_N} \lr{Z_N \mid \mathcal{D}_N} &\dlim \mathcal{N}(0, I_k),\label{eq:lawconv1}\\
    M_N(\lambda) + M_N(-\lambda) &\longrightarrow 2 \exp\lr{\frac{1}{2}\norm{\lambda}{\rk}^2}, \label{eq:subsec}
  \end{align}
  hold on some $\Omega' \in \mathcal{A}$ with $\Ptheta(\Omega') = 1$. 
  
  Fix $\omega \in \Omega'$. Construct a probability space $(\mathcal{Y}, \mathcal{F}, Q)$ with random variables:
  \begin{enumerate}
    \item $W_N \sim \mathcal{L}^{\bar{\Theta}_N} \lr{\lambda^\intercal Z_N \mid \mathcal{D}_N = \mathcal{D}_N(\omega)}$ for $N \in \mathbb{N}$,
    \item $W \sim \mathcal{N}\lr{0, \norm{\lambda}{\rk}^2}$.
  \end{enumerate}
  Note $W_N \dlim W$ under $Q$ by \eqref{eq:lawconv1}. We show $\sup_{N \in \mathbb{N}} \mathbb{E}_Q W_N^4 < \infty$, which will allow us to argue by uniform integrability. By Taylor expansion, $x^4 \le 12(e^{x} + e^{-x})$ for $x \in \mathbb{R}$. Thus,
  \begin{align}\label{eq:uniint}
    \mathbb{E}_Q W_N^4 \le 12 \lr{  \mathbb{E}_Q\exp\lr{W_N} + \mathbb{E}_Q\exp\lr{-W_N}}.
  \end{align}
  By construction,
  \begin{align*}
    \mathcal{L}\lr{\exp\lr{\pm W_N}} &= \mathcal{L}^{\bar{\Theta}_N} \lr{\exp\lr{ \pm \lambda^\intercal Z_N} \mid \mathcal{D}_N = \mathcal{D}_N(\omega)}.
  \end{align*}
  Thus, recalling definition \eqref{eq:moment} and \eqref{eq:subsec}, we have
   \begin{align*}
    \mathbb{E}_Q\exp\lr{W_N} + \mathbb{E}_Q\exp\lr{-W_N}&= M_N(\lambda)(\omega) + M_N(-\lambda)(\omega) \longrightarrow 2 \exp\lr{\frac{1}{2}\norm{\lambda}{\rk}^2}.
   \end{align*}
   In particular, 
   \begin{align*}
    \sup_{N \in \mathbb{N}} \mathbb{E}_Q\exp\lr{W_N} + \mathbb{E}_Q\exp\lr{-W_N} < \infty,
   \end{align*}
    which implies $\sup_{N \in \mathbb{N}} \mathbb{E}_Q W_N^4  < \infty$ by \eqref{eq:uniint}. The corollary of \cite[Theorem~25.12]{Billingsley1995-ly} on uniform integrability now yields
  \begin{align*}
    \mathbb{E}_QW_N  &\longrightarrow \mathbb{E}_QW,\\ 
    \mathbb{E}_Q W_N^2  &\longrightarrow \mathbb{E}_Q W^2.
  \end{align*}
  By construction, the above display translates to
  \begin{align*}
    \mathbb{E}^{\bar{\Theta}_N}[\lambda^\intercal Z_N \mid \mathcal{D}_N = \mathcal{D}_N(\omega)]  &\longrightarrow  0,\\ 
    \mathbb{E}^{\bar{\Theta}_N}[(\lambda^\intercal Z_N)^2 \mid \mathcal{D}_N = \mathcal{D}_N(\omega)]   &\longrightarrow  \Vert \lambda \Vert_{\rk}^2.
  \end{align*}
  Since $\Ptheta (
    \Omega') = 1$, we have \eqref{eq:f1} and \eqref{eq:f2} almost surely, and hence also in $\Ptheta$-probability.

  We now show $I = o_P(1)$, where we recall \eqref{eq:ssqq}. By definition, $\lambda^\intercal Z_N = I' - II'$, where 
  \begin{align}
    I' &\coloneqq \lambda^\intercal \renorm\Psi \lr{\theta_N - \truth}, \label{eq:iprimedef}\\
    II' &\coloneqq \lambda^\intercal \renorm \lr{\hat{\Psi}_N - \Psi \truth}.
  \end{align} 
  Define the random signed measure $\mu_N \coloneqq \Pi^{\bar{\Theta}_N}_N(\cdot \mid \mathcal{D}_N) - \Pi_N(\cdot \mid \mathcal{D}_N)$.
  Then,
  \begin{align}
    I = \int_{E_D} \lr{\lambda^\intercal Z_N}^2 \, d\mu_N &= \int_{E_D} \lr{I' - II'}^2 \, d\mu_N, \\ &= \int_{E_D} \left(I' \right)^2 \, \mu_N +\int_{E_D} \left(II' \right)^2 \, \mu_N - 2 II' \int_{E_D} I' \, d\mu_N. \label{eq:ofrom}
  \end{align}
  Since $II'$ is a statistic, and $\mu_N$ is a difference of posteriors, we have 
  \begin{align*}
    \int_{E_D} \left(II' \right)^2 \, d\mu_N = 0.
  \end{align*}
  Furthermore, $II' = O_P\lr{\sqrt{N}}$ from Lemma~\ref{lem:asymptnorm}. Applying the latter two considerations to \eqref{eq:ofrom}, we have
  \begin{align*}
    I = \int_{E_D} \left(I' \right)^2 \, \mu_N + O_P\left(\sqrt{N} \right)\int_{E_D} I' \, d\mu_N.
  \end{align*}
  Clearly, we have $I = o_P(1)$, if we can show 
  \begin{align}\label{eq:iprimeintegral}
    \int_{E_D} \left(I' \right)^n \, d\mu_N = o_P\lr{1 / \sqrt{N}}, \quad \text{for }n \in \{1,2\}.
  \end{align}
  Now,
  \begin{align}\label{eq:decint}
    \int_{E_D} \left(I' \right)^n \, d\mu_N = \int_{\bar{\Theta}_N} \left(I' \right)^n \, d\mu_N + \int_{\bar{\Theta}_N^c} \left(I' \right)^n \, d\mu_N.
  \end{align}
  We show that each term is $o_P\lr{1 / \sqrt{N}}$. For the first term,
  \begin{align}\label{eq:termbound}
    \left\lvert \int_{\bar{\Theta}_N} \left(I' \right)^n \, d\mu_N \right\rvert \le \norm{(I')^n}{L^\infty(\bar{\Theta}_N)}\norm{\mu_N}{\text{TV}}.
  \end{align}
  Recall \eqref{eq:iprimedef}. By $\norm{\Psi}{\text{op}} < \infty$, and Lemma~\ref{lem:boundinv}, we have
  \begin{align}\label{eq:bbound}
    \abs{I'} \le C \norm{\theta_N - \truth}{\h}, \quad \text{for some constant }C > 0.
  \end{align}
  Employing the contraction property of $\bar{\Theta}_N \subset \Theta_{N, M'}$ and recalling \eqref{eq:tvineq}-\eqref{eq:expcontract}, we get
  \begin{align*}
    \norm{(I')^n}{L^\infty(\bar{\Theta}_N)} \lesssim \left(\delta^\h_N \right)^n \text{ and } \norm{\mu_N}{\text{TV}} = O_P\lr{e^{-bs_N}}.
  \end{align*}
  Applying the above display to \eqref{eq:termbound}, we have
  \begin{align}\label{eq:shown}
    \left\lvert \int_{\bar{\Theta}_N} \left(I' \right)^n \, d\mu_N \right\rvert = O_P\left(\left(\delta^\h_N \right)^n e^{-bs_N} \right) = o_P\left(\frac{1}{\sqrt{N}} \right),
  \end{align}
  where the last equality follows by polynomial to exponential comparison.

  For the other term in \eqref{eq:decint}, we note
  \begin{align*}
    \int_{\bar{\Theta}_N^c} \left(I' \right)^n \, d\mu_N &= \int_{\bar{\Theta}_N^c} \left(I' \right)^n \, d\Pi_N^{\bar{\Theta}_N}(\cdot \mid \mathcal{D}_N) -\int_{\bar{\Theta}_N^c} \left(I' \right)^n \, d\Pi_N(\cdot \mid \mathcal{D}_N), \\ &= -\int_{\bar{\Theta}_N^c} \left(I' \right)^n \, d\Pi_N(\cdot \mid \mathcal{D}_N).
  \end{align*}
  By Cauchy-Schwarz and \eqref{eq:expcontract}, we now have
  \begin{align}
    \abs{\int_{\bar{\Theta}_N^c} \left(I' \right)^n \, d\mu_N} &= \abs{\int_{\bar{\Theta}_N^c} \left(I' \right)^n \, d\Pi_N(\cdot \mid \mathcal{D}_N)}, \\
    &\le \lr{\Pi_N(\bar{\Theta}_N^c \mid \mathcal{D}_N)\mathbb{E}\left[\left( I'\right)^{2n} \Big| \mathcal{D}_N \right]}^{\frac{1}{2}}, \\&=
     \lr{O_P\lr{e^{-bs_N}}\mathbb{E}\left[\left( I'\right)^{2n} \Big| \mathcal{D}_N \right]}^{\frac{1}{2}}.\label{eq:opbound}
  \end{align}
  Recall that we wish to show 
  \begin{align*}
    \int_{\bar{\Theta}_N^c} \left(I' \right)^n \, d\mu_N = o_P\lr{\frac{1}{\sqrt{N}}}.
  \end{align*}
  By \eqref{eq:opbound}, it suffices to show 
  \begin{align*}
    \mathbb{E}\left[\left( I'\right)^{2n} \Big| \mathcal{D}_N \right] = o_P\lr{\frac{1}{N}e^{bs_N}}.
  \end{align*}
  Fix $\eta > 0$, and let $\eta_N \coloneqq \eta e^{bs_N} / N$. We then need to show
  \begin{align*}
    \Ptheta\left(\abs{\mathbb{E}\left[\left( I'\right)^{2n} \Big| \mathcal{D}_N \right]} > \eta_N \right) \longrightarrow 0.
  \end{align*}
  Note 
  \begin{align}\label{eq:rerref}
    \mathbb{E}\left[\left( I'\right)^{2n} \Big| \mathcal{D}_N \right] = \frac{\int_{E_D}\left( I'\right)^{2n}e^{\ell_N} \, d\Pi_N}{\int_{E_D}e^{\ell_N} \, d\Pi_N} = \frac{\int_{E_D}\left( I'\right)^{2n}e^{\ell_N - \ell_N(\truth)} \, d\Pi_N}{\int_{E_D}e^{\ell_N - \ell_N(\truth)} \, d\Pi_N}.
  \end{align}
  Let 
  \begin{align*}
    H_N \coloneqq \left\{ \int_{E_D} e^{\ell_N - \ell_N(\truth)} \, d\Pi_N > e^{-a s_N} \right \}.
  \end{align*}
  Since $\Ptheta \left(H_N^c \right) \longrightarrow 0$ by Hypothesis~\ref{cond:posteriorcontraction},
  \begin{align*}
    \Ptheta\left(\abs{\mathbb{E}\left[\left( I'\right)^{2n} \Big| \mathcal{D}_N \right]} > \eta_N\right) &\le \Ptheta\left(\left\{\abs{\mathbb{E}\left[\left( I'\right)^{2n} \Big| \mathcal{D}_N \right]} > \eta_N \right\} \cap H_N \right) + o(1).
  \end{align*}
  It now suffices to show
  \begin{align*}
    III' \coloneqq \Ptheta\left(\left\{\abs{\mathbb{E}\left[\left( I'\right)^{2n} \Big| \mathcal{D}_N \right]} > \eta_N \right\} \cap H_N \right) \longrightarrow 0.
  \end{align*}
  By bounding the denominator on the right-hand side of \eqref{eq:rerref} using the defining property of $H_N$, we obtain
  \begin{align}\label{eq:iiinclusion}
    \left\{\abs{\mathbb{E}\left[\left( I'\right)^{2n} \Big| \mathcal{D}_N \right]} > \eta_N \right\} \cap H_N  \subset \left\{ \int_{E_D}\left( I'\right)^{2n}e^{\ell_N - \ell_N(\truth)} \, d\Pi_N > \frac{\eta}{N} e^{(b-a)s_N} \right\}.
  \end{align} 
  We now wish to apply Markov's inequality to the latter set. By Tonelli's Theorem,
  \begin{align}
    \Etheta\sq{ \int_{E_D}\left( I'\right)^{2n}e^{\ell_N - \ell_N(\truth)} \, d\Pi_N } &= \int_{E_D} \left( I'\right)^{2n} \Etheta\sq{e^{\ell_N - \ell_N(\truth)}} \, d\Pi_N. \label{eq:statess}
  \end{align}
  Now,
  \begin{align*}
    \Etheta\sq{e^{\ell_N - \ell_N(\truth)}} = \int_{\mathcal{X} \times V} e^{\ell_N - \ell_N(\truth)} e^{\ell_N(\truth)} \, d\nu= \int_{\mathcal{X} \times V} e^{\ell_N(\truth)} \, d\nu = 1,
  \end{align*}
  where $\nu$ is the reference measure of the likelihood. Thus, \eqref{eq:statess} states
  \begin{align*}
    \Etheta\sq{ \int_{E_D}\left( I'\right)^{2n}e^{\ell_N - \ell_N(\truth)} \, d\Pi_N } = \int_{E_D} \left( I'\right)^{2n} \, d\Pi_N = \mathbb{E}^\Pi\sq{\left( I'\right)^{2n}}.
  \end{align*}
  Applying Markov's inequality to the larger set in \eqref{eq:iiinclusion}, and using the above display,
  \begin{align}\label{eq:pppbound}
    III' \le \frac{N}{\eta}e^{-(b - a)s_N}\mathbb{E}^\Pi\sq{\left( I'\right)^{2n}}.
  \end{align}
  Recalling \eqref{eq:bbound}, we have
  \begin{align*}
    \mathbb{E}^\Pi\sq{\left( I'\right)^{2n}} \le 1 + \mathbb{E}^\Pi\sq{\left( I'\right)^{4}} \lesssim 1 + \norm{\theta_0}{\h}^{4} + \mathbb{E}^\Pi \norm{\theta_N}{\h}^{4}.
  \end{align*}
  Thus, \eqref{eq:pppbound} yields
  \begin{align*}
    III' &\lesssim Ne^{-(b - a)s_N}\lr{1 + \norm{\theta_0}{\h}^{4} + \mathbb{E}^\Pi \norm{\theta_N}{\h}^{4} }, \\ &=  Ne^{-(b - a)s_N}\mathbb{E}^\Pi \norm{\theta_N}{\h}^{4} + Ne^{-(b - a)s_N}\lr{1 + \norm{\theta_0}{\h}^{4}},\\
    & \longrightarrow 0.
  \end{align*}
  Here, we pass the limit by Assumption~2 for the first term, and exponential to polynomial comparison for the second term.

  Lastly, we argue $II = o_P(1)$. Recall \eqref{eq:ssqq1}. Since $\lambda^\intercal Z_N = I' - II'$, and $II'$ is a statistic,
  \begin{align*}
    II &= \mathbb{E}^{\bar{\Theta}_N}[I' - II' \mid \mathcal{D}_N] - \mathbb{E}[ I' - II' \mid \mathcal{D}_N],\\ 
    &=\mathbb{E}^{\bar{\Theta}_N}[I'\mid \mathcal{D}_N] - \mathbb{E}[ I' \mid \mathcal{D}_N], \\ 
    &= \int_{E_D} I' \, d\mu_N .
  \end{align*}
  Thus, $II = o_P(1)$ follows, since we have already established \eqref{eq:iprimeintegral}.
\end{proof}
\begin{proof}[Proof of Theorem~\ref{thm:rnm}]
  \hspace*{0.1cm}\\\\
  \eqref{eq:finalplim}: By \eqref{eq:expp} and linearity,
  \begin{align}
  \mathbb{E}\sq{Z_N \mid \mathcal{D}_N} &= \sqrt{N}i_{D}^{1/2} \Epost{\Psi \theta_N - \Psi_N} \\ &= \sqrt{N}i_{D}^{1/2} \lr{\hat{\Psi}_N - \Psi_N} \plim 0.\label{eq:eun}
  \end{align}
  Coupling this with Lemma~\ref{lem:asymptnorm} and Slutsky's lemma, we obtain \eqref{eq:finalplim}.
  \\\\
  \eqref{eq:finaldlim}: By Remark~\ref{rem:subseqence}, \eqref{eq:eun}, and \eqref{eq:distlim}, we can WLOG assume that
  \begin{align*}
  \mathcal{L}\lr{\sqrt{N}i_{D}^{1/2}\lr{\Psi\theta_N-\Psi_N} \mid \mathcal{D}_N} &\dlim \mathcal{N}(0, I_k),\\
  \sqrt{N}i_{D}^{1/2}\lr{\hat{\Psi}_N - \Psi_N} &\longrightarrow 0,
  \end{align*}
  hold on some $\Omega' \in \mathcal{A}$ with $\Ptheta(\Omega') = 1$. 
  
  Fix $\omega \in \Omega'$. Construct $W_N \sim \Pi^\Psi_N(\cdot \mid \mathcal{D}_N = \mathcal{D}_N(\omega))$ for all $N \in \mathbb{N}$ on a common probability space $(\mathcal{Y}, \mathcal{F}, Q)$. Then,
  \begin{align*}
    \sqrt{N}\renorm \lr{W_N - \Psi_N(\omega)} &\dlim \mathcal{N}(0, I_k), \quad \text{under }Q, \\
    \sqrt{N}i_{D}^{1/2}\lr{\hat{\Psi}_N(\omega) - \Psi_N(\omega)} &\longrightarrow 0.
  \end{align*}
  By Slutsky's lemma, we conclude 
  \begin{align*}
    \sqrt{N}\renorm \lr{W_N - \hat{\Psi}_N(\omega)} \dlim \mathcal{N}(0, I_k).
  \end{align*}
  The above limit translates to
  \begin{align*}
    \mathcal{L}\lr{\sqrt{N}i_{D}^{1/2}\lr{\Psi\theta_N-\hat{\Psi}_N} \mid \mathcal{D}_N = \mathcal{D}_N(\omega)} &\dlim \mathcal{N}(0, I_k).
  \end{align*}
  Since $\Ptheta(\Omega') = 1$, we have the limit \eqref{eq:finaldlim} almost surely, and hence also in $\Ptheta$-probability. 
\end{proof}
\begin{proof}[Proof of Theorem~\ref{thm:recov}]
  We first establish $i_D^{-1} \longrightarrow L$. Coupling \eqref{eq:finalAD} and \eqref{eq:coup}, we have
  \begin{align*}
    \lr{i_{D}^{-1}}_{ij} = \ip{\proj\psi_i}{\fishermatrix^{-1}\proj\psi_j}{\h} = \ip{P_{E_D}\psi_i}{\bar{\psi}_D^j}{\h}, \quad i,j \le k.
  \end{align*}
  Using \eqref{eq:pswap}, we then have
  \begin{align*}
    \lr{i_{D}^{-1}}_{ij}   =  \ip{\psi_i}{\bar{\psi}_D^j}{\h}.
  \end{align*}
  Substitute now $\psi_i = \mathcal{I}\phi_i$ and apply self-adjointness of $\mathcal{I}$. Then,
  \begin{align*}
    \lr{i_{D}^{-1}}_{ij} =\ip{\mathcal{I}\phi_i}{\bar{\psi}_D^j}{\h} = \ip{\phi_i}{\mathcal{I}\bar{\psi}_D^j}{\h}.
  \end{align*}
  Writing $\phi_i = P_{E_D}\phi_i + \lr{\phi_i - P_{E_D}\phi_i }$, we now have
  \begin{align*}
    \lr{i_{D}^{-1}}_{ij} &= \ip{P_{E_D}\phi_i}{\mathcal{I}\bar{\psi}_D^j}{\h} + \ip{\phi_i - P_{E_D}\phi_i}{\mathcal{I} \bar{\psi}_D^j}{\h}.
  \end{align*}
  By self-adjointness of $P_{E_D}$, and convergence of projections,
  \begin{align*}
    \ip{P_{E_D}\phi_i}{\mathcal{I}\bar{\psi}_D^j}{\h} = \ip{\phi_i}{\fishermatrix\bar{\psi}_D^j}{\h} = \ip{\phi_i}{P_{E_D}\psi_j}{\h} \longrightarrow \ip{\phi_i}{\psi_j}{\h} = L_{ij}.
  \end{align*}
  Thus, it suffices to show
  \begin{align*}
    I \coloneqq \ip{\phi_i - P_{E_D}\phi_i}{\mathcal{I} \bar{\psi}_D^j}{\h} \longrightarrow 0.
  \end{align*}
  For this, Cauchy-Schwarz, \eqref{eq:bspec}, and Assumption~2 yield
  \begin{align*}
    \abs{I} \lesssim \norm{\phi_i - P_{E_D}\phi_i}{\h} \norm{\bar{\psi}_D^j}{\h}
    \lesssim \frac{\norm{\phi_i - P_{E_D}\phi_i}{\h}}{\eigmin{\fishermatrix}}   \longrightarrow 0.
  \end{align*}

  We now show $L \succ 0$. For all $x \in \rk$, we need to show
  \begin{align*}
    \sum_{i = 1}^k\sum_{j = 1}^k  x_iL_{ij}  x_j = 0 \text{ implies }x = 0.
  \end{align*}
  Recalling $\mathcal{I} = \lin^* \lin$, we have
  \begin{align*}
    L_{ij} = \ip{\phi_i}{\psi_j}{\h} = \ip{\phi_i}{\fisher \phi_j}{\h} = \ip{\lin \phi_i}{\lin \phi_j}{\h}. 
  \end{align*}
  Hence, 
  \begin{align*}
    \sum_{i = 1}^k\sum_{j = 1}^k  x_iL_{ij}  x_j = \left\Vert \lin  \left[\sum_{i = 1}^k x_i \phi_i \right] \right\Vert_{\h}^2.
  \end{align*}
  The claim now follows by Assumption~3, if we can show that $\phi_1, \ldots, \phi_k$ are linearly independent. Suppose $\sum_{i = 1}^k a_i \phi_i = 0$, where $a_1, \ldots, a_k \in \mathbb{R}$. Applying the information operator, $\sum_{i = 1}^k a_i \psi_i = 0$. Thus, $a_1 = \ldots = a_k = 0$, since $\psi_1, \ldots, \psi_k$ are linearly independent.

  For \eqref{eq:sem2} and \eqref{eq:sem1}, we transform \eqref{eq:finaldlim} and \eqref{eq:finalplim} by $i_D^{-\frac{1}{2}} \longrightarrow L^{\frac{1}{2}}$ using Slutsky's lemma. For \eqref{eq:sem2}, we formally have to argue by Remark~\ref{rem:subseqence}.
\end{proof}
\section*{Proofs for Section~\ref{sec:cred}}
\begin{proof}[Proof of Lemma~\ref{lem:pivot}]
  We start by showing
  \begin{align*}
    \Ptheta\lr{\variance \in \text{GL}_k(\mathbb{R})} \longrightarrow 1.
  \end{align*}
  Since $\hat{\Psi}_N$ is a statistic, 
  \begin{align*}
    \Vpost{\Psi \theta_N - \hat{\Psi}_N} = \Vpost{\Psi \theta_N} = \hat{\Sigma}_N.
  \end{align*}
  Thus, 
  \begin{align*}
    \mathbb{V}\sq{Z_N \mid \mathcal{D}_N} = N\renorm \Vpost{\Psi \theta_N - \hat{\Psi}_N}  \renorm = N\renorm \variance \renorm.
  \end{align*}
  Hence, \eqref{eq:var} yields 
  \begin{align}\label{eq:covlimit}
    N\renorm\variance\renorm \plim I_k.
  \end{align}
  Since $\text{GL}_k(\mathbb{R})$ is open, we have
  \begin{align*}
    \Ptheta\lr{N\renorm\variance\renorm \in \text{GL}_k(\mathbb{R})} \longrightarrow 1.
  \end{align*}
  As $\renorm$ is invertible, we obtain
  \begin{align*}
    N\renorm\variance\renorm \in \text{GL}_k(\mathbb{R}) \Longleftrightarrow \variance \in \text{GL}_k(\mathbb{R}).
  \end{align*}
  Hence, we also have
  \begin{align*}
    \Ptheta\lr{\variance \in \text{GL}_k(\mathbb{R})}  \longrightarrow 1.
  \end{align*}
  \eqref{eq:pivot1}: Using Remark~\ref{rem:subseqence} on \eqref{eq:covlimit} and \eqref{eq:finaldlim}, we can WLOG assume that
  \begin{align}
    \mathcal{L}\lr{\sqrt{N}i_{D}^{1/2} \lr{\Psi\theta_N-\hat{\Psi}_N} \mid \mathcal{D}_N } &\dlim \mathcal{N}(0, I_k),\label{eq:distlimit} \\
    N\renorm\variance\renorm &\longrightarrow I_k, \label{eq:varlim}
  \end{align}
  hold on some $\Omega' \in \mathcal{A}$ with $\Ptheta(\Omega') = 1$. 
  
  Fix $\omega \in \Omega'$. Construct $W_N \sim \Pi_N^\Psi(\cdot \mid \mathcal{D}_N = \mathcal{D}_N(\omega))$ for all $N \in \mathbb{N}$ on a common probability space $(\mathcal{Y}, \mathcal{F}, Q)$. By \eqref{eq:distlimit}, we have under $Q$ that
  \begin{align}\label{eq:anlim}
    \sqrt{N}\renorm A_N \dlim \mathcal{N}(0, I_k), \quad \text{where }A_N \coloneqq W_N - \hat{\Psi}_N(\omega).
  \end{align}
  Applying now $\norm{\cdot}{\rk}^2$, we obtain
  \begin{align}\label{eq:chiklim}
    NA_N^\intercal i_D A_N \dlim \chi_k^2, \quad \text{under } Q.
  \end{align}
  Let $o_Q(1)$ denote a sequence of $\mathbb{R}$-valued random variables on $(\mathcal{Y}, \mathcal{F}, Q)$ that converges to zero in $Q$-probability. We write $\bs{o}_Q(1)$ for $\rk$-valued random variables. 
  
  If we can show
  \begin{align}\label{eq:weshowthat}
    A_N^\intercal \variance^{-1}(\omega) A_N = N A_N^{\intercal} i_D A_N + o_Q(1),
  \end{align}
  then by \eqref{eq:chiklim} and Slutsky's lemma
  \begin{align*}
    A_N^\intercal \variance^{-1}(\omega) A_N \dlim \chi_k^2, \quad \text{under }Q.
  \end{align*}
  Since $\Ptheta(\Omega') = 1$, we then have \eqref{eq:pivot1} almost surely, and hence in $\Ptheta$-probability. 
  
  To show \eqref{eq:weshowthat}, we write
  \begin{align*}
    A_N^\intercal \variance^{-1}(\omega) A_N = \qform{\sqrt{N}\renorm A_N}{\lr{\frac{1}{N}i_D^{-\frac{1}{2}}\variance^{-1}(\omega) i_D^{-\frac{1}{2}}}}{\sqrt{N}\renorm A_N}.
  \end{align*}
  Since matrix inversion is continuous, inverting \eqref{eq:varlim} yields
  \begin{align*}
    \frac{1}{N}i_D^{-\frac{1}{2}}\variance^{-1}(\omega) i_D^{-\frac{1}{2}} = I_k + E_N,
  \end{align*}
  where $(E_N) \subset \rkbyk$ satisfies $E_N \longrightarrow 0$. Thus, 
  \begin{align}
    A_N^\intercal \variance^{-1}(\omega) A_N &= \qform{\sqrt{N}\renorm A_N}{\lr{I_k + E_N}}{\sqrt{N}\renorm A_N}, \\
    &= N A_N^\intercal i_D A_N + \qform{\sqrt{N}\renorm A_N}{E_N}{\sqrt{N}\renorm A_N}.\label{eq:rightperror}
  \end{align}
  We now see that \eqref{eq:weshowthat} is equivalent to
  \begin{align*}
    \qform{\sqrt{N}\renorm A_N}{E_N}{\sqrt{N}\renorm A_N} = o_Q(1).
  \end{align*}
  By \eqref{eq:anlim}, we have
  \begin{align*}
    E_N \sqrt{N}\renorm A_N &= \bs{o}_Q(1), \\
    \lr{\sqrt{N}\renorm A_N}^\intercal \bs{o}_Q(1) &= o_Q(1).
  \end{align*}
  Thus,
  \begin{align*}
    \qform{\sqrt{N}\renorm A_N}{E_N}{\sqrt{N}\renorm A_N} = \lr{\sqrt{N}\renorm A_N}^\intercal \bs{o}_Q(1) = o_Q(1),
  \end{align*}
  as desired.
  \\\\
  \eqref{eq:pivot2}: This follows by a near-identical argument. Indeed, applying $\norm{\cdot}{\rk}^2$ to \eqref{eq:finalplim}, we have
  \begin{align*}
    N\qform{\hat{\Psi}_N - \Psi \truth}{i_D}{\hat{\Psi}_N - \Psi \truth} \dlim \chi_k^2.
  \end{align*}
  Similarly to before, we can then argue by $Ni_D^{\frac{1}{2}}\variance^{-1} i_D^{\frac{1}{2}} \plim I_k$ and Slutsky's lemma.
\end{proof}
\begin{proof}[Proof of Theorem~\ref{thm:crd}]
  We draw inspiration from \cite[Section~4.1.3]{Nickl2023}.
  \\\\
  Case~\ref{en:case1}: Define
  \begin{align*}
    B_N \coloneqq \qform{\hat{\Psi}_N - \Psi \theta_0}{\variance^{-1}}{\hat{\Psi}_N - \Psi \theta_0},
  \end{align*}
  which satisfies $ B_N \dlim \chi_k^2$ by \eqref{eq:pivot2}. Let $F_{\chi_k^2}$ denote the distribution function of $\chi_k^2$. Then,
  \begin{align}\label{eq:asin}
    \Ptheta(\Psi \truth \in C_N) = \Ptheta\lr{B_N \le Q_{\chi_k^2}(1 - \alpha)} \longrightarrow F_{\chi_k^2}\lr{Q_{\chi_k^2}(1 - \alpha)} = 1 - \alpha.
  \end{align}
  \\
  Case~\ref{en:case2}: By Remark~\ref{rem:subseqence}, assume WLOG that
  \begin{align}\label{eq:lawconv}
    \mathcal{L}\lr{\qform{\Psi \theta_N - \hat{\Psi}_N}{\variance^{-1}}{\Psi \theta_N - \hat{\Psi}_N} \mid \mathcal{D}_N} \dlim \chi_k^2
  \end{align}
  holds on some $\Omega' \in \mathcal{A}$ with $\Ptheta\lr{\Omega'} = 1$. Fix $\omega \in \Omega'$, and let $F_N$ denote the distribution function of
  \begin{align*}
    \mathcal{L}\lr{\qform{\Psi \theta_N - \hat{\Psi}_N}{\variance^{-1}}{\Psi \theta_N - \hat{\Psi}_N} \mid \mathcal{D}_N = \mathcal{D}_N(\omega)}.
  \end{align*}
  Now, regarding \eqref{eq:lawconv} at $\omega$ yields
  \begin{align}\label{eq:distconv}
    F_N \stackrel{d}{\longrightarrow} F_{\chi_k^2},
  \end{align}
  in the sense of weak convergence of the corresponding probability measures. By construction, we also have 
  \begin{align*}
    F_N\lr{ R_N(\omega)} = 1 - \alpha.
  \end{align*}
  Thus, Theorem~\ref{lem:conv} and \eqref{eq:distconv} yield the 'convergence of quantiles', 
  \begin{align*}
    R_N(\omega) \longrightarrow Q_{\chi_k^2}(1-\alpha).
  \end{align*}
  Since $\Ptheta(\Omega') = 1$, the previous limit holds almost surely. In particular,
  \begin{align}\label{eq:rnlim}
    R_N \plim Q_{\chi_k^2}(1-\alpha). 
  \end{align} 
  Let $a_N \coloneqq Q_{\chi^2_k}(1 - \alpha)/R_N$. Then,
  \begin{align*}
    \Ptheta(\Psi \truth \in C_N) = \Ptheta\lr{B_N \le R_N} = \Ptheta\lr{a_N B_N \le Q_{\chi^2_k}(1 - \alpha)}.
  \end{align*}
  Since $a_N \plim 1$, Slutsky's lemma yields $a_NB_N \dlim \chi_k^2$. Thus,
  \begin{align*}
    \Ptheta(\Psi \truth \in C_N) = \Ptheta\lr{a_N B_N \le Q_{\chi^2_k}(1 - \alpha)} \longrightarrow F_{\chi_k^2}\lr{Q_{\chi_k^2}(1 - \alpha)} = 1 - \alpha.
  \end{align*}
  \\
  For the remainder of this proof, we argue universally for Case~\ref{en:case1} and Case~\ref{en:case2}. We proceed to establish \eqref{eq:diambound}. Since $C_N$ is an ellipsoid,
  \begin{align*}
    \text{diam}\lr{C_N} = 2\sqrt{R_N\lambda_{\text{max}}(\variance)}.
  \end{align*}
  Recalling \eqref{eq:rnlim}, we have $R_N = (1 + o_P(1)) Q_{\chi^2_k}(1 - \alpha)$ in both cases. We thus need to show
  \begin{align*}
    \lambda_{\text{max}}(\variance) = (1 + o_P(1))\lambda_{\text{max}}\lr{\frac{1}{N}i_D^{-1}}.
  \end{align*}
  Equivalently, we can show
  \begin{align}\label{eq:lammax}
    \lambda_{\text{max}}(N \variance) = (1 + o_P(1))\lambda_{\text{max}}\lr{i_D^{-1}}.
  \end{align}
  We have $N\renorm\variance\renorm \plim I_k$, see \eqref{eq:covlimit}. Hence, we can write 
  \begin{align*}
    N\renorm\variance\renorm = I_k + C_N,
  \end{align*}
  where $(C_N) \subset \rkbyk$ satisfies $C_N \plim 0$. Thus,
  \begin{align}\label{eq:hatref}
    N\variance  =  i_D^{-\frac{1}{2}} \lr{N \renorm\variance\renorm} i_D^{-\frac{1}{2}} =  i_D^{-\frac{1}{2}}(I_k + C_N)i_D^{-\frac{1}{2}} = i_D^{-1} + i_D^{-\frac{1}{2}}C_N i_D^{-\frac{1}{2}}.
  \end{align}
  Recall $\norm{C_N}{\text{op}} \plim 0$ and $\norm{i_D^{-\frac{1}{2}}}{\text{op}}^2 = \eigmax{i_D^{-1}}$ by \eqref{eq:eigsq}. Then,
  \begin{align}
    \norm{i_D^{-\frac{1}{2}}C_N i_D^{-\frac{1}{2}}}{\text{op}} &\le \norm{C_N}{\text{op}}\norm{i_D^{-\frac{1}{2}}}{\text{op}}^2, \\ &= \norm{C_N}{\text{op}}\eigmax{i_D^{-1}}, \\ &=  o\lr{\eigmax{i_D^{-1}}}.\label{eq:obound}
  \end{align}
  Now, \eqref{eq:lammax} is equivalent to
  \begin{align*}
    \frac{\abs{\eigmax{N \variance} - \eigmax{i_D^{-1}}}}{\eigmax{i_D^{-1}}} \plim 0.
  \end{align*}
  Recall \eqref{eq:hatref}. Since $\lambda_{\text{max}}$ is Lipschitz continuous on the space of symmetric $k\times k$ matrices \cite[Corollary~6.3.8]{Horn_Johnson_2012}, we obtain
  \begin{align*}
    \abs{\eigmax{N \variance} - \eigmax{i_D^{-1}}} \lesssim \norm{N\variance - i_D^{-1}}{\text{op}} = \norm{i_D^{-\frac{1}{2}}C_N i_D^{-\frac{1}{2}}}{\text{op}}.
  \end{align*}
  Hence, recalling \eqref{eq:obound}, we have
  \begin{align*}
    \frac{\abs{\eigmax{N \variance} - \eigmax{i_D^{-1}}}}{\eigmax{i_D^{-1}}} \lesssim  \frac{\norm{i_D^{-\frac{1}{2}}C_N i_D^{-\frac{1}{2}}}{\text{op}}}{\eigmax{i_D^{-1}}}\plim 0.
  \end{align*}

  We proceed to show
  \begin{align*}
    \text{diam}\lr{C_N} = \frac{1 + o_P(1)}{\sqrt{N\eigmin{\fishermatrix}}}.
  \end{align*}
  By \eqref{eq:diambound} and \eqref{eq:eiginv}, it suffices to show $\eigmax{i_D^{-1}} \lesssim \eigmax{\fishermatrix^{-1}}$. We argue similarly to Lemma~\ref{lem:boundinv}. Let $x \in S^{k - 1}$, where $S^{k - 1} \coloneqq \curly{y \in \rk : \norm{y}{\rk} = 1}$. Note that for $D$ large enough, we have
  \begin{align}\label{eq:takesup}
    x^\intercal i_D^{-1} x = x^\intercal\mathcal{J}_D[\fishermatrix]^{-1}\mathcal{J}_D^\intercal x = I \cdot II,
  \end{align}
  where 
  \begin{enumerate}
    \item $I \coloneqq y^\intercal [\fishermatrix]^{-1} y$ and $y \coloneqq \mathcal{J}_D^\intercal x / \norm{\mathcal{J}_D^\intercal x}{\mathbb{R}^D}$,
    \item $II \coloneqq \norm{\mathcal{J}_D^\intercal x}{\mathbb{R}^D}^2$.
  \end{enumerate}
  Clearly, $I \le \eigmax{\fishermatrix^{-1}}$ and $II = x^\intercal \mathcal{J}_D \mathcal{J}_D^\intercal  x \le \eigmax{\mathcal{J}_D \mathcal{J}_D^\intercal}$. Apply the bounds for $I$ and $II$ to \eqref{eq:takesup}, and take the supremum over $x \in S^{k -1}$. Then,
  \begin{align*}
    \eigmax{i_D^{-1}} \le \eigmax{\mathcal{J}_D \mathcal{J}_D^\intercal} \eigmax{\fishermatrix^{-1}}.
  \end{align*}
  Arguing as in Lemma~\ref{lem:boundinv}, we have
  \begin{align*}
    \eigmax{\mathcal{J}_D \mathcal{J}_D^\intercal} \longrightarrow \eigmax{P},
  \end{align*}
  where $P \coloneqq \lr{\ip{\psi_i}{\psi_j}{\h}}_{i,j \le k}$. In particular, 
  \begin{align*}
    \sup_{D \in \mathbb{N}} \eigmax{\mathcal{J}_D \mathcal{J}_D^\intercal} < \infty.
  \end{align*}
\end{proof}
\begin{proof}[Proof of Corollary~\ref{cor:int}]
  Recall $k = 1$. Clearly, $\variance$ and $\hat{\Psi}_N$ are scalar-valued, and
  \begin{align*}
    I_N = \curly{x \in \mathbb{R}: \lr{x - \hat{\Psi}_N}\variance^{-1} \lr{x - \hat{\Psi}_N} \le  \variance^{-1} R_N^2 }.
  \end{align*}
  Hence, the statistic $\variance^{-1} R_N^2$ is a valid choice for Case~\ref{en:case2} in Theorem~\ref{thm:crd}.
\end{proof}
\section*{Proofs for Section~\ref{sec:darcy}}
\begin{lemma}\label{lem:convproj}
Let $\beta \in \mathbb{N}$. Then,
\begin{enumerate}
  \item If $\phi \in h^\beta(\mathcal{X})$, then $\Vert \phi - P_{E_D}\phi \Vert_{L^2} = o\lr{D^{-\frac{\beta}{d}}}$,
  \item If $\phi \in C_{0, \Delta}^\infty(\mathcal{X})$, then $\Vert \phi - P_{E_D}\phi \Vert_{L^2} = o\lr{D^{-r}}$ for all $r > 0$,
  \item If $\phi \in h^\beta(\mathcal{X})$, then $\Vert \phi \Vert_{h^\beta} \lesssim D^{\frac{\beta}{d}}\Vert \phi \Vert_{L^2}$, with constant independent of $D$ and $\phi \in E_D$.
\end{enumerate}
\end{lemma}
\begin{proof}
  Claims 1 and 3 follow by the Weyl-asymptotics $\lambda_j \lesssim j^{2/d}$, see \cite[Section~A.3]{Nickl2023}. Claim~1 implies Claim~2, if we can show $C_{0, \Delta}^\infty(\mathcal{X}) \subset \bigcap_{\beta \in \mathbb{N}} h^{2\beta}(\mathcal{X})$. 
  
  Let $\beta \in \mathbb{N}$ and $\varphi \in C_{0, \Delta}^\infty(\mathcal{X})$. We need to show $\varphi \in h^{2\beta}(\mathcal{X})$, which is equivalent to
  \begin{align*}
    \sum_{j = 1}^\infty \lambda_j^{2\beta} \abs{\ip{\varphi}{e_j}{L^2}}^2 < \infty.
  \end{align*}
  Now,
  \begin{align*}
    \lambda_j^{2\beta} \abs{\ip{\varphi}{e_j}{L^2}}^2 = \abs{\ip{\varphi}{ \lambda_j^\beta e_j}{L^2}}^2 = \abs{\ip{\varphi}{ \Delta^\beta e_j }{L^2}}^2 = \abs{\ip{ \Delta^\beta \varphi}{  e_j }{L^2}}^2.
  \end{align*}
  Here, we integrate by parts in the last equality, which is justified by $\varphi \in C_{0, \Delta}^\infty(\mathcal{X})$. Thus, Parseval's identity yields
  \begin{align*}
    \sum_{j = 1}^\infty \lambda_j^{2\beta} \abs{\ip{\varphi}{e_j}{L^2}}^2 = \sum_{j = 1}^\infty \abs{\ip{\Delta^\beta \varphi}{ e_j}{L^2}}^2 = \norm{\Delta^\beta \varphi}{L^2}^2 < \infty.
  \end{align*}
\end{proof} 
\subsubsection{Properties of the Forward Map}\label{sec:analytic}
 A linearization of $\mathcal{G}$ at $\theta \in \Theta$ in the sense of Hypothesis~\ref{cond:2} was derived in \cite[Theorem~3.3.2]{Nickl2023}. It is densely defined by
\begin{align*}
  \mathbb{I}_{\theta}h \coloneqq -\mathcal{L}_{\theta}^{-1}\sq{\nabla \cdot \lr{e^{\theta} h \nabla u_\theta}}, \quad h \in \Theta,
\end{align*}
where $\mathcal{L}_{\theta}^{-1} : L^2(\mathcal{X}) \longrightarrow L^2(\mathcal{X})$ is the solution operator for \eqref{eq:darcypde} with vanishing boundary condition \cite[Section~A.2]{Nickl2023}. We confirm Hypothesis~\ref{cond:inv} by a stability estimate.
\begin{lemma}\label{lem:stab}
  $\eigmin{\fishermatrix} \gtrsim D^{-6/d}.$
\end{lemma}
\begin{proof}
  The last part of \cite[p.~116]{Nickl2023} states
\begin{align}\label{eq:stab}
  \Vert \lin  h\Vert_{L^2}   \ge C D^{-3/d}\Vert h \Vert_{L^2}, 
\end{align}
where $C > 0$ is independent of $D$ and $h \in E_D$. Noting $\mathcal{I} = \lin^* \lin$ and recalling \eqref{eq:pswap}, we have
\begin{align*}
  \Vert \lin  h\Vert_{L^2}^2 =\ip{\mathcal{I}h}{h}{L^2} = \ip{\fishermatrix h}{h}{L^2}, \quad h \in E_D.
\end{align*}
Square \eqref{eq:stab} and plug in the above display:
\begin{align*}
  \ip{\fishermatrix h}{h}{L^2} \ge C^2 D^{-6/d}\Vert h \Vert_{L^2}^2.
\end{align*}
Take now the infimum over $h \in E_D \text{ with } \norm{h}{L^2} = 1$.
\end{proof}
Let $\beta\in \mathbb{N}$ and $m \in \mathbb{N}_0$. We shall use the embeddings
\begin{align*}
  h^\beta(\mathcal{X}) &\hookrightarrow H^\beta_0(\mathcal{X}),\\
  H^\beta_0(\mathcal{X}) &\hookrightarrow C^m(\bar{\mathcal{X}}), \quad \beta > m + d/2,
\end{align*}
see \cite[Section~A.3.3]{Nickl2023}. 
\begin{lemma}\label{lem:lem41}
  Let $\beta \in \mathbb{N}$ satisfy $\beta > 1 + d/2$. Let $h, \theta, \theta \in \Theta$. Then,
  \begin{align}
    \label{eq:testi} \Vert \mathbb{I}_\theta h \Vert_{L^\infty} &\lesssim \Vert h \Vert_{h^1},
  \end{align}
  with constant only depending on $A \ge \max\curly{\Vert \theta \Vert_{h^\beta}}$. Furthermore,
  \begin{align}
    \label{eq:sesti}
    \Vert \forward_\theta - \forward_{\theta'} \Vert_{L^\infty} &\lesssim \Vert \theta - \theta' \Vert_{h^1}, \\
    \label{eq:festi}
    \Vert \forward_\theta - \forward_{\theta'} \Vert_{L^2} &\lesssim \Vert \theta - \theta' \Vert_{L^2},
  \end{align} 
   with constants only depending on $ B \ge \max\curly{\Vert \theta \Vert_{h^\beta}, \Vert \theta' \Vert_{h^\beta}}$.
\end{lemma}
\begin{proof}
  \hspace*{0.5cm}
  \\\\
  \eqref{eq:testi}: Let $\lesssim$ denote inequality up to a constant that only depends on $B$. Since $d \le 3$, we have $H^{2}(\mathcal{X}) \hookrightarrow C^0(\bar{\mathcal{X}})  $. Thus,
  \begin{align}\label{eq:lfinity}
    \Vert \mathbb{I}_{\theta} h \Vert_{L^\infty} \lesssim  \Vert \mathbb{I}_{\theta} h \Vert_{H^2} = \norm{\mathcal{L}_{\theta}^{-1}\sq{\nabla \cdot \lr{e^{\theta} h \nabla u_\theta}}}{H^2}, \quad h \in \Theta.
  \end{align}
  By \cite[Proposition~A.5.3]{Nickl2023}, we have
  \begin{align*}
   \norm{\mathcal{L}_\theta^{-1} g}{H^2}  \lesssim \norm{g}{L^2}, \quad g \in L^2(\mathcal{X}),
  \end{align*}
  where the constant only depends on $A$, since $h^\beta(\mathcal{X}) \hookrightarrow C^1(\bar{\mathcal{X}})$. Recalling \eqref{eq:lfinity}, we thus have
  \begin{align}\label{eq:recnow}
    \Vert \mathbb{I}_{\theta} h \Vert_{L^\infty} \lesssim \norm{\nabla \cdot \lr{e^{\theta} h \nabla u_\theta}}{L^2} \lesssim \norm{e^{\theta}h \nabla u_{\theta}}{H^1}.
  \end{align}
  Using the multiplicative Sobolev inequality from \cite[Section~A.1]{Nickl2023}, we have
  \begin{align*}
      \norm{e^{\theta}h \nabla u_{\theta}}{H^1} \lesssim \norm{e^{\theta}}{C^1} \norm{\nabla u_{\theta}}{H^1}\Vert h \Vert_{H^1}.
  \end{align*}
  Now, $\norm{e^{\theta}}{C^1} \lesssim 1$ follows by $h^\beta(\mathcal{X}) \hookrightarrow C^1(\bar{\mathcal{X}})$. By \cite[Proposition~A.5.3]{Nickl2023}, we also have $\Vert u_{\theta} \Vert_{H^2} \lesssim 1$. Thus, the above display yields
  \begin{align*}
    \norm{e^{\theta}h \nabla u_{\theta}}{H^1} \lesssim \Vert h \Vert_{H^1} \lesssim \Vert h \Vert_{h^1}.
  \end{align*}
  Recalling \eqref{eq:recnow}, we are done.
  \\\\
  \eqref{eq:sesti}: This follows by coupling \eqref{eq:testi} with the mean value theorem \cite[Theorem~3.2.7]{Drbek2013}.
  \\\\
  \eqref{eq:festi}: This follows by \cite[Proposition~2.1.3]{Nickl2023} and $h^\beta(\mathcal{X}) \hookrightarrow H^\beta_0(\mathcal{X})$.
\end{proof}
\begin{lemma}\label{lem:stabi}
  Let $\beta \in \mathbb{N}$ satisfy $\beta > 1 + d/2$. Then, 
  \begin{align*}
    \Vert \theta - \theta' \Vert_{L^2}  \lesssim \Vert \forward_\theta - \forward_{\theta'} \Vert_{L^2}^{\frac{\beta - 1}{\beta + 1}}, \quad \theta, \theta' \in \Theta,
  \end{align*}
  where the constant only depends on $B \ge \max\curly{\Vert \theta \Vert_{h^\beta}, \Vert \theta' \Vert_{h^\beta}} $.
\end{lemma}
\begin{proof}
  Since $f > 0$, we can check \cite[Proposition~2.1.5]{Nickl2023} using the remark that precedes \cite[Proposition~2.1.7]{Nickl2023}. The result then follows by arguing as in \cite[Proposition~2.1.7]{Nickl2023}.

  We note that this reference replaces $e^\theta$ in \eqref{eq:darcypde} by $e^\theta + f_{\text{min}}$, where $f_{\text{min}} > 0$. The arguments still work in our case. Indeed, since $h^\beta(\mathcal{X}) \hookrightarrow C^1(\bar{\mathcal{X}})$, we can use the bounds
  \begin{align*}
    \norm{e^\theta}{C^1} \le C_1 \text{ and } e^\theta \ge e^{-C_2}, 
  \end{align*}
  where $C_1 = C_1(B)$ and $C_2 = C_2(B)$ are positive constants.
\end{proof}
\begin{lemma}\label{lem:linerror}
  Let $\beta \in \mathbb{N}$ satisfy $\beta \ge 2$. Then,
  \begin{align*}
    \norm{R_\theta}{L^2} \lesssim \norm{\theta - \truth}{h^\beta}^{\frac{4}{\beta}} \norm{\theta - \truth}{L^2}^{\frac{2(\beta - 2)}{\beta}},
  \end{align*}
  with constant independent of $\theta \in \Theta$.
\end{lemma}
\begin{proof}
  Let $\lesssim$ denote inequality up to constant independent of $\theta \in \Theta$. 
  
  By \cite[Theorem~3.3.2]{Nickl2023}, we have
  \begin{align*}
  \norm{R_\theta}{L^2} \lesssim \norm{\theta - \truth}{L^\infty}^2.
  \end{align*}
  Using $ h^2(\mathcal{X})  \hookrightarrow C^0(\bar{\mathcal{X}})$, we obtain $\norm{\theta - \truth}{L^\infty}^2 \lesssim \norm{\theta - \truth}{h^2}^2$. The interpolation inequality \cite[Proposition~8.19]{Engl1996-qv} yields
  \begin{align*}
    \norm{\theta - \truth}{h^2} \lesssim \Vert \theta - \truth \Vert_{h^\beta}^{\frac{2}{\beta}} \Vert \theta - \truth \Vert_{L^2}^{\frac{\beta - 2}{\beta}}. 
  \end{align*}
  Tracking back the inequalities, we are done.
\end{proof}
\subsubsection{Setting up the remaining Hypotheses}\label{sec:setup}
Assume $N^{l} \lesssim D \lesssim N^{\frac{d}{2\alpha + d}}$ for some $l > 0$. We claim that Hypothesis~\ref{cond:posteriorcontraction} holds with 
\begin{enumerate}
  \item $\delta_N^{\forward} \coloneqq N^{-\frac{\alpha}{2\alpha + d}}$,
  \item $\mathcal{R} \coloneqq h^\alpha(\mathcal{X})$,
  \item $\delta^\h_N \coloneqq \left(\delta_N^\forward\right)^{\frac{\alpha - 1}{\alpha + 1}}$.
\end{enumerate}
This will follow from Lemma~\ref{lem:stabi}, and \cite[Exercise~2.1.1]{Nickl2023} (with spectral Sobolev spaces), if we can show 
\begin{align*}
  \norm{\forward(\truth) - \forward(P_{E_D}\truth)}{L^2} \lesssim \delta_N^\forward.
\end{align*}
However, \eqref{eq:festi} and Lemma~\ref{lem:convproj}, Claim 2 yield
\begin{align*}
  \norm{\forward(\truth) - \forward(P_{E_D}\truth)}{L^2} \lesssim \norm{\truth - P_{E_D}\truth}{L^2} \lesssim D^{-\frac{1}{l}\cdot\frac{\alpha}{2\alpha + d}} \lesssim N^{-\frac{\alpha}{2\alpha + d}}.
\end{align*}

For Hypothesis~\ref{cond:4}, we let
\begin{align*}
  d_{\Theta}(\theta, \theta') \coloneqq \norm{\theta - \theta'}{h^1}, \quad \theta, \theta' \in \Theta.
\end{align*}
Since $\Theta_{N, M} \subset B_{h^\alpha}(M)$, we can set $r_N \coloneqq 1$ by \eqref{eq:sesti}. By \eqref{eq:testi} and \eqref{eq:sesti}, we have
\begin{align*}
  \frac{\norm{R_\theta - R_{\theta'}}{L^\infty}}{\norm{\theta - \theta'}{h^1}}  \le \frac{\norm{\forward_\theta - \forward_{\theta'}}{L^\infty}}{\norm{\theta - \theta'}{h^1}} + \frac{\norm{\lin [\theta - \theta']}{L^\infty}}{\norm{\theta - \theta'}{h^1}} \lesssim 1 , \quad \theta, \theta' \in \Theta_{N, M},
\end{align*}
where the constant only depends on $M$. Thus, we can set $g_N \coloneqq 1$ 

By Lemma~\ref{lem:linerror} and $\Theta_{N, M} \subset B_{h^\alpha}(M)$, we can set $\sigma_N \coloneqq \left(\delta_N^\h \right)^{\frac{2(\alpha - 2)}{\alpha}}$. Similarly, we can set $\delta_N^\Theta \coloneqq \left(\delta_N^\h \right)^{\frac{\alpha - 1}{\alpha + 1}}$. Indeed, we can interpolate $h^1(\mathcal{X})$ by $h^\alpha(\mathcal{X})$ and $L^2(\mathcal{X})$ using \cite[Proposition~8.19]{Engl1996-qv}.
 
We proceed to define $J_N$. Using $\regset \subset B_{h^\alpha}(M)$, and subsequently \cite[Proposition~A.3.1]{Nickl2023}, we have
 \begin{align*}
  \int_0^t \sqrt{\log 2 N(\regset, d_\Theta; \varepsilon)} \, d\varepsilon &\lesssim \int_0^t \sqrt{\log 2 N\lr{B_{h^\alpha}(M), \norm{\cdot}{h^1}; \varepsilon}} \, d\varepsilon \\ &\lesssim \int_0^t \varepsilon^{-\frac{d}{2(\alpha - 1)}} \, d\varepsilon,
 \end{align*}
 with constants independent of $N$ and $t > 0$. We can thus set $J_N(t) \coloneqq t^{1 - \frac{d}{2\alpha - 2}}$.
\begin{proof}[Proof of Theorem~\ref{thm:drc}]
  The diameter bound is clear from Lemma~\ref{lem:stab} and the last bound provided in Theorem~\ref{thm:crd}.
  
  We first show
  \begin{align}
    D &\lesssim N^{\frac{d}{2\alpha + d}}, \label{eq:bbn} \\
    N \sigma_N^{3/2} &= o(1), \label{eq:flim1} \\
     \delta_N^\forward \sqrt{s_N} &\le \frac{1}{\sqrt{s_N}}, \label{eq:boundofsn} \\
     \frac{ D^{\frac{\alpha}{d}}}{\sqrt{s_N}\eigmin{\fishermatrix}}  &= o(1). \label{eq:slim1}
  \end{align}
  Recall that we required \eqref{eq:bbn} in Section~\ref{sec:setup}. We use \eqref{eq:flim1}, \eqref{eq:slim1} and \eqref{eq:boundofsn} to check assumptions.
  \\\\
  \eqref{eq:bbn}: We factorize \eqref{eq:boundonu}, and use $d \le \alpha$:
  \begin{align*}
   u <  \frac{1}{2}\cdot \frac{d}{2\alpha + d} \cdot \frac{d}{\alpha + 6} < \frac{d}{2\alpha + d}.
  \end{align*}
  \eqref{eq:flim1}: By definition, $N \sigma_N^{3/2} = N^{t_1}$, where 
  \begin{align*}
    t_1 \coloneqq 1 - \frac{3}{2} \cdot \frac{2\lr{\alpha - 2}}{\alpha} \cdot \frac{\alpha - 1}{\alpha + 1} \cdot \frac{\alpha}{2\alpha + d}.
  \end{align*}
  Noting that $\alpha \ge 14$ and $d \le 3$, we obtain $t_1 < 0$.
  \\\\
  \eqref{eq:boundofsn}: We have
  \begin{align}\label{eq:sncalc}
    s_N = N \lr{\delta_N^\forward}^2 = N^{1 - \frac{2\alpha}{2\alpha + d}} = N^{ \frac{d}{2\alpha + d}}.
  \end{align}
  Now,
  \begin{align*}
    \delta_N^\forward \sqrt{s_N} \le \frac{1}{\sqrt{s_N}} \Longleftrightarrow s_N \delta_N^\forward \le 1 \Longleftrightarrow \frac{d - \alpha}{2\alpha + d } \le 0,
  \end{align*}
  and the last statement is true, since $d \le \alpha$.
  \\\\
  \eqref{eq:slim1}: By Lemma~\ref{lem:stabi}, \eqref{eq:sncalc} and $D \lesssim N^{u}$, we have
  \begin{align*}
    \frac{ D^{\frac{\alpha}{d}}}{\sqrt{s_N} \eigmin{\fishermatrix}} \lesssim N^{- \frac{d}{4\alpha + 2d}} D^{\frac{\alpha}{d} + \frac{6}{d}} \lesssim N^{t_2},
  \end{align*}
  where
  \begin{align*}
    t_2 \coloneqq u \lr{\frac{\alpha}{d} + \frac{6}{d}} - \frac{d}{4\alpha + 2d}.
  \end{align*}
  The condition on $u$ from \eqref{eq:boundonu} is equivalent to $t_2 < 0$. Thus, $N^{t_2} \longrightarrow 0$.
  \\\\
  We now check the assumptions of Section~\ref{sec:invprob}. 
  \\\\
  \textit{Lemma~\ref{lem:asymptnorm}}: If we can show 
  \begin{align*}
    \delta_N^\forward \norm{\perturb}{L^\infty} \longrightarrow 0,
  \end{align*}
  then the assumption holds with $r \coloneqq \frac{\alpha}{2\alpha + d}$. 
  
  Recalling $h^\alpha(\mathcal{X}) \hookrightarrow C^0(\bar{\mathcal{X}})$, Lemma~\ref{lem:convproj}, Claim 2, and subsequently \eqref{eq:bspec}, we have
  \begin{align}\label{eq:ussing}
    \norm{\perturb}{L^\infty} \lesssim \norm{\perturb}{h^\alpha} \lesssim D^{\frac{\alpha}{d}}\norm{\perturb}{L^2} \lesssim \frac{D^{\frac{\alpha}{d}}}{\eigmin{\fishermatrix}}.
  \end{align}
  Using $s_N \ge 1$ and \eqref{eq:boundofsn}, we have 
  \begin{align}\label{eq:delnineq}
    \delta_N^\forward \le \delta_N^\forward \sqrt{s_N} \le \frac{1}{\sqrt{s_N}}.
  \end{align}
  Coupling this with \eqref{eq:ussing}, we obtain
  \begin{align}\label{eq:fby}
    \delta_N^\forward \norm{\perturb}{L^\infty} \le \frac{1}{\sqrt{s_N}} \norm{\perturb}{L^\infty} \lesssim \frac{D^{\frac{\alpha}{d}}}{\sqrt{s_N }\eigmin{\fishermatrix}} \longrightarrow 0.
  \end{align}
  Here, we pass the limit by \eqref{eq:slim1}.
  \\\\
  \textit{Lemma~\ref{lem:empiricalprocess}}: Consider Assumption 1. We justify Remark~\ref{rem:lip} by \eqref{eq:festi}. It now suffices to check $\sqrt{s_N} \eigmin{\fishermatrix} \longrightarrow \infty$. Passing the limit by \eqref{eq:slim1}, we have
  \begin{align*}
    \sqrt{s_N} \eigmin{\fishermatrix} \ge \frac{\sqrt{s_N} \eigmin{\fishermatrix}}{D^{\frac{\alpha}{d}}} = \lr{\frac{D^{\frac{\alpha}{d}}}{\sqrt{s_N }\eigmin{\fishermatrix}}}^{-1} \longrightarrow \infty.
  \end{align*}
  For Assumption 2, note $s_N \le N$ by \eqref{eq:sncalc}, so $1/ \sqrt{N} \le 1 / \sqrt{s_N}$. Using \eqref{eq:ussing}, we thus obtain
  \begin{align*}
    \frac{1}{\sqrt{N}}\norm{\perturb}{h^\alpha} \lesssim \frac{D^{\frac{\alpha}{d}}}{\sqrt{s_N}\eigmin{\fishermatrix}} \longrightarrow 0.
  \end{align*}
  Here, we use \eqref{eq:slim1} to pass the limit.
  
  We now check Assumptions 3-6. We shall use the following facts:
  \begin{enumerate}
    \item $\delta_N^\Theta < \sqrt{\sigma}_N $, since $(\alpha - 1)/(\alpha + 1) > (\alpha - 2)/\alpha$,
    \item $\alpha \ge 14$ and $d \le 3$,
    \item $J_N(t) = t^{1 - \frac{d}{2\alpha - 2}}$ for $t > 0$.
  \end{enumerate}
  For Assumption 3, we pass the limit by \eqref{eq:flim1}:
  \begin{align*}
    \sqrt{N}r_N  J_N\lr{\sigma_N /r_N} = \sqrt{N}\sigma^{1 - \frac{d}{2\alpha - 2}}_N \le \sqrt{N}\sigma_N^{ \frac{23}{26}} \le \lr{N \sigma_N^{\frac{3}{2}}}^{\frac{1}{2}} \longrightarrow 0.
  \end{align*}
  For Assumption 4, we pass the limit by logarithmic to polynomial comparison:
  \begin{align*}
    \sqrt{\log N} r_N^3 \delta_N^\Theta J_{N}(\sigma_N /r_N)^2 /\sigma_N^2 &= \sqrt{\log N}\delta_N^\Theta \lr{\sigma_N}^{- \frac{d}{\alpha - 1}}, \\ 
    &\le \sqrt{\log N} \sigma_N^{\frac{1}{2} - \frac{d}{\alpha - 1}},  \\ 
    &\le \sqrt{\log N} \sigma_N^{\frac{7}{26}} \longrightarrow 0.
  \end{align*}
  For Assumption 5, we pass the limit by \eqref{eq:flim1}:
  \begin{align*}
    \sqrt{N} g_N^2 \delta_N^\Theta J_{N}\lr{\delta_N^\Theta} = \sqrt{N}\lr{\delta_N^\Theta}^{2 - \frac{d}{2\alpha - 2}}\le \sqrt{N}\sigma_N^{1 - \frac{d}{4\alpha - 4}} \le \sqrt{N}\sigma_N^{\frac{49}{52}} \le \lr{N \sigma_N^{\frac{3}{2}}}^{\frac{1}{2}} \longrightarrow 0.
  \end{align*}
  For Assumption 6:
  \begin{align*}
    g_N J_{N}\lr{\delta_N^\Theta} = \lr{\delta_N^\Theta}^{1 - \frac{d}{2\alpha - 2}} \le   \lr{\delta_N^\Theta}^{\frac{25}{56}} \longrightarrow 0.
  \end{align*}
  \\
  \textit{Lemma~\ref{lem:lla}}: For Assumption~1, note $\delta_N^\h < \sqrt{\sigma_N}$, since $(\alpha - 2)/\alpha < 1$. Thus,
  \begin{align*}
    N\lr{\sigma_N^2 + \sigma_N \delta^\h_N } \le N \lr{\sigma_N^2 + \sigma_N^{\frac{3}{2}}} \le 2 N \sigma_N^{\frac{3}{2}} \longrightarrow 0,
  \end{align*}
  where we pass the limit by \eqref{eq:flim1}. Assumptions~2-3 are satisfied by Lemma~\ref{lem:convproj}, Claim 2.
  \\\\
  \textit{Lemma~\ref{lem:prjbvm}}: We first show that the covariance operator $\mathcal{C}: E_D \longrightarrow E_D$ of $\Pi_N$ with respect to 
  \begin{align}\label{eq:psp}
    \ip{x}{y}{\pspace} \coloneqq s_N \ip{x}{y}{h^\alpha} = s_N\sum_{i = 1}^D  \lambda_i^{\alpha}\ip{x}{e_i}{L^2} \ip{y}{e_i}{L^2}, \quad x, y \in E_D,
  \end{align}
  is the identity. Recalling \eqref{eq:sncalc}, we see that
  \begin{align*}
    \theta_N \coloneqq \frac{1}{\sqrt{s_N}}\sum_{i = 1}^D \lambda_i^{-\frac{\alpha}{2}} W_i e_i, \quad \text{where }W_i \stackrel{\text{iid}}{\sim} \mathcal{N}(0, 1),
  \end{align*}
  satisfies $\theta_N \sim \Pi_N$. Let $x \in E_D$. Then, 
  \begin{align*}
    \mathcal{C}x &= s_N  \mathbb{E}^\Pi \sq{\ip{\theta_N}{x}{h^\alpha} \theta_N}, \\ &= \sum_{i = 1}^D  \sum_{j = 1}^D  \lambda_i^{\alpha} \lambda_i^{-\frac{\alpha}{2}}  \ip{x}{e_i}{L^2}  \mathbb{E} \sq{W_i W_j}\lambda_j^{-\frac{\alpha}{2}} e_j, \\
    &= \sum_{i = 1}^D \ip{x}{e_i}{L^2} e_i, \\
    &= x.
  \end{align*}
  
  Consider Assumption 1. Recalling \eqref{eq:psp}, we need to show
  \begin{align*}
    \delta_N^\forward \sqrt{s_N}\norm{\perturb}{h^\alpha} \longrightarrow 0.
  \end{align*}
  By \eqref{eq:boundofsn} and \eqref{eq:ussing}, we have
  \begin{align*}
    \delta_N^\forward \sqrt{s_N}\norm{\perturb}{h^\alpha} \le \frac{1}{\sqrt{s_N}}\norm{\perturb}{h^\alpha} \lesssim  \frac{D^{\frac{\alpha}{d}}}{\sqrt{s_N}\eigmin{\fishermatrix}} \longrightarrow 0.
  \end{align*}
Here, we pass the limit by \eqref{eq:slim1}.
  
Consider Assumption 2. We need to show
\begin{align*}
  e^{-(b - a)s_N}\sqrt{s_N} \mathbb{E}^\Pi\norm{\theta_N}{h^\alpha} \longrightarrow 0.
\end{align*}
Let
  \begin{align*}
    Z \coloneqq \sum_{i = 1}^\infty \lambda_i^{-\frac{\alpha}{2}} W_i e_i ,\quad \text{where }W_i \stackrel{\text{iid}}{\sim} \mathcal{N}(0, 1).
  \end{align*}
  This defines a Gaussian random variable taking values in $L^2(\mathcal{X})$. By Fernique's theorem \cite[Corollary~2.8.6]{Bogachev1998-aw}, we have $\mathbb{E}\norm{Z}{L^2} < \infty$. Since $P_{E_D}Z/ \sqrt{s_N} \sim \Pi_N$, we see that
  \begin{align*}
    \mathbb{E}^\Pi \norm{\theta_N}{L^2} \le \frac{1}{\sqrt{s_N}} \mathbb{E}\norm{Z}{L^2}.
  \end{align*}
  Recall Lemma~\ref{lem:convproj}, Claim 2, and $D \lesssim N^u$. Then,
  \begin{align*}
    e^{-(b - a)s_N}\sqrt{s_N} \mathbb{E}^\Pi\norm{\theta_N}{h^\alpha} &\lesssim e^{-(b - a)s_N}\sqrt{s_N} D^{\frac{\alpha}{d}} \mathbb{E}^\Pi \norm{\theta_N}{L^2}, \\ &\lesssim e^{-(b - a)s_N}N^{\frac{\alpha u}{d}} \mathbb{E}\norm{Z}{L^2}\longrightarrow 0.
  \end{align*}  
  Here, we pass the limit by polynomial to exponential comparison.
\end{proof}
\section*{Proofs for Section~\ref{sec:schr}}
\begin{proof}[Proof of Theorem~3]
Let $\mathcal{L}_{\truth}^{-1} : L^2(\mathcal{X}) \longrightarrow L^2(\mathcal{X})$ be the self-adjoint solution operator of \eqref{eq:schrodinger} with vanishing boundary condition \cite[Section~A.2]{Nickl2023}. By \cite[Theorem~3.3.1]{Nickl2023}, we can set 
\begin{align*}
\lin \coloneqq \mathcal{L}_{\truth}^{-1} M,
\end{align*}
where $M : L^2(\mathcal{X}) \longrightarrow L^2(\mathcal{X})$ is defined by
\begin{align*}
  M \phi \coloneqq e^{\truth} u_{\truth} \phi, \quad \phi \in L^2(\mathcal{X}).
\end{align*}
We start with Assumptions 1-3 of Theorem~\ref{thm:recov}.

We first argue that $\mathcal{I} : C_c^\infty(\mathcal{X}) \longrightarrow C_c^\infty(\mathcal{X})$ forms a bijection. Then, Assumption~1 holds for some $\phi_1, \ldots, \phi_k \in C_c^\infty(\mathcal{X})$. Clearly,
\begin{align*}
  \mathcal{I} = \lin^* \lin = M \lr{ \mathcal{L}_{\truth}^{-1} }^2 M.
\end{align*}
Furthermore, we can argue by \cite[Theorem~6]{Evans2022-vw} that $\mathcal{L}_{\truth}^{-1} : C_c^\infty(\mathcal{X}) \longrightarrow C_c^\infty(\mathcal{X})$ forms a bijection. Hence, it suffices to argue that $M : C_c^\infty(\mathcal{X}) \longrightarrow C_c^\infty(\mathcal{X})$ forms a bijection. By the Feynman-Kac formula \cite[Section~A.4]{Nickl2023}, we have $u_{\truth} > 0$. Since $u_{\theta_0} \in C^\infty(\bar{\mathcal{X}})$, we can then check that $M : C_c^\infty(\mathcal{X}) \longrightarrow C_c^\infty(\mathcal{X})$ forms a bijection with inverse
\begin{align*}
  M^{-1}f \coloneqq \frac{1}{e^{\truth} u_{\truth}}f, \quad f \in C_c^\infty(\mathcal{X}).
\end{align*}

We proceed with Assumption 2-3. Arguing by \cite[Proposition~A.5.1]{Nickl2023}, we have
\begin{align*}
  \Vert \lin  h \Vert_{L^2} \gtrsim \Vert h \Vert_{\lr{H^2_0}^*}, \quad h \in L^2(\mathcal{X}).
\end{align*}
This implies Assumption 3. Applying \cite[Lemma~4.9]{Nickl2022} to the above display,
\begin{align}
  \Vert \lin  h \Vert_{L^2} \gtrsim  D^{-2/d} \Vert h \Vert_{L^2}, 
\end{align}
with constant independent of $D$ and $h \in E_D$. Arguing as in Lemma~\ref{lem:stab}, we have
\begin{align*}
  \eigmin{\fishermatrix} \gtrsim D^{-4/d}.
\end{align*}
Assumption~2 now follows by Lemma~\ref{lem:convproj}, Claim 2. 

Theorem~\ref{thm:recov} now yields $i_D^{-1} \longrightarrow L$ for some $L \succ 0$. Since $\lambda_{\text{max}}$ is continuous on the space of symmetric $k \times k$ matrices \cite[Corollary~6.3.8]{Horn_Johnson_2012},
\begin{align*}
  \eigmax{i_D^{-1}} = (1 + o(1))\eigmin{L}.
\end{align*}
Thus, \eqref{eq:diambound} can be written as $C(1 + o_P(1))/\sqrt{N}$ for some constant $C > 0$.
\end{proof}

\section*{General Results}\label{sec:supgen}
        \begin{remark}\label{rem:subseqence}
          Let $H$ be a hypothesis. Let $(X_N), (Y_N), X$ and $Y$ be $\mathbb{R}^n$-valued random variables on $(\Omega, \mathcal{A}, \Ptheta)$. Suppose we wish to prove
          \begin{align*}
            \text{$H$ and $X_N \plim X$ implies $Y_N \plim Y$},
          \end{align*}
          using only properties of $(X_N)$ that are invariant under extracting subsequences. By a standard subsequence argument, we can WLOG assume $X_N \stackrel{\text{a.s}}{\longrightarrow} X$. 

          In particular, $X_N \plim X$ can represent a hypothesis of the form
          \begin{align*}
            \mu_N \dlim \mu, \quad \text{in }\Ptheta\text{- probability},
          \end{align*}
          where $(\mu_N)$ and $\mu$ are random measures. Indeed, the above dispaly is equivalent to $d_{\text{BL}}\lr{\mu_N, \mu} \plim 0$. In this setting, we can thus WLOG assume $d_{\text{BL}}\lr{\mu_N, \mu} \stackrel{\text{a.s}}{\longrightarrow} 0$. Equivalently, we can assume that $\mu_N \dlim \mu$ holds $\Ptheta$-almost surely.
        \end{remark}

        \begin{theorem}\label{lem:conv}
          Let $\curly{F_n}_{n \in \mathbb{N}}$ and $F$ be distribution functions on $\mathbb{R}$ such that
          \begin{align*}
            F_n \dlim F, \quad n \longrightarrow \infty,
          \end{align*}
           where the limit denotes weak convergence of the corresponding probability measures. 
           
           Let $p \in (0, 1)$, and assume $F(x) = p$ has a unique solution $x = q_p$. Suppose $\curly{x_n}_{n \in \mathbb{N}}$ is real-valued sequence that satisfies
           \begin{align*}
            F_n(x_n) = p, \quad \text{for all }n \in \mathbb{N}.
           \end{align*}
           Then, $x_n \longrightarrow q_p$ for $n \longrightarrow \infty$.
        \end{theorem}
        \begin{proof}
          The proof has similarities with \cite[Chapter 14, Proposition~5]{Fristedt1997}. Let $\varepsilon > 0$. We have to show $\lvert q_p - x_n \rvert < \varepsilon$ for $n$ large enough. We will only show $x_n < q_p + \varepsilon$ for $n$ large enough. The other inequality follows similarly.
          
          Since $F$ has countably many discontinuities, we can pick $\delta \in (0 ,\varepsilon)$ such that $F$ is continuous at $q_p + \delta$. Then, 
          \begin{align}\label{eq:distconv1}
            F_n(q_p + \delta) \longrightarrow F(q_p + \delta), \quad n \longrightarrow \infty.
          \end{align}
          Since $q_p$ is the unique solution to $F(x) = p$, we have
          \begin{align*}
            p = F(q_p) < F(q_p + \delta).
          \end{align*}
          Thus, recalling \eqref{eq:distconv1},
          \begin{align*}
            p < F_n(q_p + \delta), \quad \text{for }n \text{ large enough}.
          \end{align*}
          However, $p = F_n(x_n)$, so
          \begin{align*}
            F_n(x_n) < F_n(q_p + \delta), \quad \text{for }n \text{ large enough}.
          \end{align*}
          Since $F_n$ is non-decreasing and $\delta < \varepsilon$, we conclude that
          \begin{align*}
            x_n < q_p + \delta < q_p + \varepsilon, \quad \text{for }n \text{ large enough}.
          \end{align*}
        \end{proof}
        \begin{theorem}\label{thm:empb}
        Let $A$ be a separable metric space. Let $\mathcal{X}$ be a measurable space. Let $\mathcal{F} = \curly{f_\alpha : \alpha \in A}$ be a class of measurable real-valued functions on $\mathcal{X}$ with $0 \in \mathcal{F}$. Let $F : \mathcal{X} \longrightarrow (0, \infty)$ be measurable with $F(x) \ge \sup_{f \in \mathcal{F}} f(x)$ for $x \in \mathcal{X}$. Let $P$ be a probability measure on $\mathcal{X}$ such that $\mathcal{F} \subset L^2(P)$. Let $X_1, \ldots, X_n \stackrel{\text{iid}}{\sim} P$. Assume $\alpha \in A \mapsto f_\alpha(x) $ is continuous for all $x \in \mathcal{X}$. Let $s \ge \sup_{f \in \mathcal{F}} \Vert f\Vert_{L^2(P)}$. Let $\mu \coloneqq \mathbb{E}f(X_1)$. Then,
        \begin{align*}
          \mathbb{E}\sup_{f \in \mathcal{F}} \left\lvert \sum_{i = 1}^n  \lr{f(X_i) - \mu} \right\rvert \le C \max\left\{ \sqrt{N}\Vert F \Vert_{L^2(P)}I(d), u \lr{\frac{I(d)}{d}}^2  \right \},
        \end{align*}
        where 
        \begin{enumerate}
          \item $I(t) \coloneqq  \int_0^t \sup_Q \sqrt{\log 2 N(\mathcal{F}, \Vert \cdot \Vert_{L^2(Q)}, \varepsilon \Vert F \Vert_{L^2(Q)})} \, d\varepsilon$ for $t > 0,$ where $Q$ represents a probability measure on $\mathcal{X}$ with finite support.
          \item $u \coloneqq \Vert \text{max}_{i\le n} F(X_i) \Vert_{L^2(P)}$.
          \item $d \coloneqq s/\Vert F \Vert_{L^2(P)}$.
          \item $C$ is a universal constant.
         \end{enumerate}
        \end{theorem}
        \begin{proof}
          This is \cite[Theorem~3.5.4]{Giné_Nickl_2021} with the following modifications:
          \begin{enumerate}
            \item Instead of letting $s \coloneqq \sup_{f \in \mathcal{F}} \Vert f\Vert_{L^2(P)}$, we allow a bound $s \ge \sup_{f \in \mathcal{F}} \Vert f\Vert_{L^2(P)}$. By inspection, the proof still works.
            \item We do not require that $\mathcal{F}$ is countable. Instead, we require that $A$ is separable and that $\alpha \in A \mapsto f_\alpha(x) $ is continuous for all $x \in \mathcal{X}$.
          \end{enumerate}
          We justify the second modification. Let $\curly{\alpha_m}_{m \in \mathbb{N}} \subset A$ be dense and define $\tilde{\mathcal{F}} \coloneqq \curly{f_{\alpha_m}}_{m \in \mathbb{N}}$. We can WLOG assume that $0 \in \tilde{\mathcal{F}}$. We now apply \cite[Theorem~3.5.4]{Giné_Nickl_2021} to $\tilde{\mathcal{F}}$, where we 'reuse' the quantities $F, s, u, d$. The quantity corresponding to $I(t)$ will have to be redefined to
          \begin{align*}
            \tilde{I}(t) \coloneqq \int_0^t \sup_Q \sqrt{\log 2 N(\tilde{\mathcal{F}}, \Vert \cdot \Vert_{L^2(Q)}, \varepsilon \Vert F \Vert_{L^2(Q)})} \, d\varepsilon, \quad t > 0.
          \end{align*}
          However, since $\tilde{\mathcal{F}} \subset \mathcal{F}$, it is clear that $\tilde{I}(t) \le I(t)$. We now conclude that
          \begin{align*}
            \mathbb{E}\sup_{f \in \mathcal{\tilde{F}}} \left\lvert \sum_{i = 1}^n  \lr{f(X_i) - \mu} \right\rvert \le C \max\left\{ \sqrt{N}\Vert F \Vert_{L^2(P)}I(d),  u I^2(d)/ d^2  \right \}.
          \end{align*}
          We now wish to replace the above supremum over $\mathcal{\tilde{F}}$ by one over $\mathcal{F}$. Since 
          \begin{align*}
            \alpha \in A \mapsto \abs{\sum_{i = 1}^n \lr{f_\alpha(X_i) - \mu}}
          \end{align*}
          is continuous, and $\curly{\alpha_m}_{m \in \mathbb{N}} \subset A$ is dense,
          \begin{align*}
            \sup_{f \in \tilde{\mathcal{F}}} \left\lvert \sum_{i = 1}^n  f(X_i) \right\rvert = \sup_{m \in \mathbb{N}} \left\lvert \sum_{i = 1}^n  \lr{f_{\alpha_m}(X_i) - \mu} \right\rvert = \sup_{\alpha \in A} \left\lvert \sum_{i = 1}^n \lr{ f_\alpha(X_i) - \mu} \right\rvert = \sup_{f \in \mathcal{F}} \left\lvert \sum_{i = 1}^n  f(X_i) \right\rvert.
          \end{align*}
        \end{proof}    

\bibliographystyle{imsart-number} 
\bibliography{bibliography}       
\end{document}